\documentclass[11pt,reqno]{amsart}
\usepackage{fullpage}
\usepackage{amsmath,amsthm,verbatim,amssymb,amsfonts,amscd, graphicx,bbm}
\usepackage{graphics}
\usepackage{mathtools}
\usepackage{mathrsfs}

\usepackage{hyperref}

\newtheorem{theorem}{Theorem}

\newtheorem{lemma}[theorem]{Lemma}

\theoremstyle{definition}
\newtheorem{remark}[theorem]{Remark}

\newcommand{\cref}[1]{Corollary~\ref{c.#1}}

\numberwithin{equation}{section}
\numberwithin{theorem}{section}

\newcommand{\Z}{\mathbb{Z}}

\newcommand{\R}{\mathbb{R}}

\newcommand{\E}{\mathbb{E}}

\newcommand{\F}{\mathcal{F}}
\newcommand{\scF}{\mathscr{F}}
\newcommand{\G}{\mathcal{G}}

\newcommand{\bT}{\mathbb{T}}
\newcommand{\bP}{\mathbb{P}}
\newcommand{\cF}{\mathcal{F}}

\newcommand{\cG}{\mathcal{G}}
\newcommand{\cL}{\mathcal{L}}
\newcommand{\cH}{\mathcal{H}}
\newcommand{\cN}{\mathcal{N}}

\renewcommand{\tilde}{\widetilde}
\renewcommand{\div}{\mathrm{div}}
\newcommand{\ol}{\overline}

\newcommand{\eps}{\varepsilon}
\newcommand{\gam}{\gamma}
\newcommand{\bone}{\mathbbm{1}}
\newcommand{\Var}{\mathrm{Var}}
\title{Limiting distribution of elliptic homogenization error with periodic diffusion and random potential}

\author{Wenjia Jing}
\address{Department of Mathematics,
  University of Chicago,
  5734 S. University Avenue,
  Chicago, IL 60637, USA.}
\email{wjing@math.uchicago.edu}

\date{\today}

\begin{document}
 
\maketitle

\begin{abstract}

We study the limiting probability distribution of the homogenization error for second order elliptic equations in divergence form with highly oscillatory periodic conductivity coefficients and highly oscillatory stochastic potential. The effective conductivity coefficients are the same as those of the standard periodic homogenization, and the effective potential is given by the mean. We show that in the short range correlation setting, the limiting distribution of the random part of the homogenization error, as random elements in proper Hilbert spaces, is Gaussian and can be characterized by the homogenized Green's function, the homogenized solution and the statistics of the random potential. Similar results hold for random potentials that are functions of long range correlated Gaussian random fields. These results generalize previous ones in the setting with slowly varying diffusion coefficients, and the current setting with fast oscillations in the differential operator requires new methods to prove compactness of the probability distributions of the random fluctuation.

\smallskip


\end{abstract}

%%%%%%%%%%%%%%
%%%%%%%%%%%%%%

\section{Introduction}

In this article we study the limiting distribution, in certain Hilbert spaces, of the homogenization error for second order elliptic equations in divergence form with highly oscillatory periodic diffusion coefficients and highly oscillatory random potential. 

More precisely, we consider the following Dirichlet problem on an open bounded subset $D \subset \R^n$, with homogeneous boundary condition and a source term $f \in L^2(D)$,
\begin{equation}
\label{e.rpde}
\left\{
\begin{aligned}
&-\frac{\partial}{\partial x_i} \left(a_{ij} \left(\frac{x}{\eps}\right) \frac{\partial u^\eps}{\partial x_j} (x,\omega)\right) + q \left(\frac{x}{\eps},\omega\right) u^\eps(x,\omega) = f(x),\quad &&\quad x\in D,\\
& u^\eps(x) = 0, \quad &&\quad x \in \partial D.
\end{aligned}
\right.
\end{equation}
The conductivity coefficients $(a_{ij}(\frac{\cdot}{\eps}))$ and the potential $q(\frac{\cdot}{\eps},\omega)$ are highly oscillatory in space, and $0<\eps \ll 1$ indicates the small scale on which these coefficients oscillate. We assume that the conductivity coefficients are deterministic and periodic, and the potential is a stationary random field on some probability space $(\Omega,\F,\bP)$. More precise assumptions are given in section \ref{sec:results}. It is well known that, under mild assumptions like stationary ergodicity of $q(x,\omega)$, the equation above homogenizes, i.e. $u^\eps$ converges, almost surely in $\Omega$, weakly in $H^1(D)$ and strongly in $L^2(D)$ to the solution of the deterministic homogenized problem
\begin{equation}
\label{e.hpde}
\left\{
\begin{aligned}
&-\ol{a}_{ij} \frac{\partial^2 u}{\partial x_i \partial x_j}(x) + \ol{q} u(x) = f(x),\quad &&\quad x\in D,\\
& u(x) = 0, \quad &&\quad x \in \partial D.
\end{aligned}
\right.
\end{equation}
Here, the effective conductivity coefficients $(\ol{a}_{ij})$ are constants defined by
\begin{equation}
\ol{a}_{ij} = \int_{\bT^d} a_{ik}(y) \left( \delta_{kj} + \frac{\partial \chi^k}{\partial x_j}(y)\right) dy,
\end{equation}
where ${\bT^d} = [0,1]^d$ denotes the unit cell and the correctors $\chi^k$, $k = 1, \cdots, d$, are given by the unique solution of the corrector equation
\begin{equation}\label{e.corrector}
-\frac{\partial}{\partial x_i} \left(a_{ij}(y) \left(e_k + \frac{\partial \chi^k}{\partial x_j} (y) \right) \right) = 0, \quad \text{ on } \bT^d,
\end{equation}
with the normalization condition $\int_{\bT^d} \chi^k dy = 0$; $e_k$ above is the $k$th standard unit basis vector of $\R^d$. We note that this formula for $(\ol{a}_{ij})$ is exactly the classic periodic homogenization formula for effective conductivity. The effective potential $\ol{q}$ in \eqref{e.hpde} is given by the constant
\begin{equation}
\ol{q} = \E\, q(0,\omega),
\end{equation}
where $\E$ denotes the mathematical mean with respect to $\bP$.

In this paper we study the law (probability distribution) of the homogenization error $u^\eps - u$, viewed as random elements in certain Hilbert spaces. We split this error as two parts: $\E u^\eps - u$ and $u^\eps - \E u^\eps$. In view of the deterministic oscillations in the diffusion coefficients, we expect that the periodic homogenization error, in the replacement of $(a^{ij}(\frac{\cdot}{\eps}))$ to $(\ol{a}_{ij})$, makes significant contributions to the deterministic error $\E u^\eps - u$. Indeed, we show later that this error is essentially of order $O(\eps)$, the same as periodic homogenization. On the other hand, the effect of the random potential $q(\frac{\cdot}{\eps},\omega)$ becomes visible in the random fluctuation $u^\eps - \E u^\eps$, in which the (large) mean is removed. We are interested in characterizing the size and the law of this random fluctuation, and the answers depend on finer information of the random potential $q$, such as the decay rate of the correlations in $q$ and higher order moments of $q$; see section \ref{sec:results} for notions and definitions.

We find that, when $q(x,\omega)$ has short range correlations, the random fluctuation $u^\eps - \E u^\eps$ scales like $\eps^{\frac{d}{2}\wedge 2}$ in $L^1(\Omega,L^2(D))$ norm, and scales like $\eps^{\frac d 2}$ when integrated against a test function. Moreover, the law of the scaled random fluctuation $\eps^{-\frac d 2}(u^\eps - \E u^\eps)$ in $L^2(D)$ for $d = 2,3$ and in $H^{-1}(D)$ for $d = 4,5$, converges to Gaussian distributions as follows (see Theorem \ref{t.sl.23} for details): 
\begin{equation*}
  \frac{u^\eps - \E u^\eps}{\sqrt{\eps^d}} \xrightarrow{\rm{distribution}} \sigma \int_D G(x,y)u(y) dW(y).
\end{equation*}
Here, $W(y)$ is the standard multiparameter Wiener process and hence the law of the right hand above defines a Gaussian probability measure on $L^2(D)$ or $H^{-1}(D)$. This Gaussian distribution is determined by: $G(x,y)$ which is the Green's function associated to the homogenized problem \eqref{e.hpde}, $u(y)$ which is the homogenized solution, and $\sigma$ which is some statistical parameter of the random potential $q(x,\omega)$.

We also consider the case when $q(x,\omega)$ has long range correlations and is constructed as a function of Gaussian random field. Then the random fluctuation scales like $\eps^{\frac{\alpha}{2}\wedge 2}$ in $L^1(\Omega,L^2(D))$ norm, and scales like $\eps^{\frac \alpha 2}$ when integrated against a test function. Moreover, the law of the scaled random fluctuation $\eps^{-\frac \alpha 2}(u^\eps - \E u^\eps)$ in $L^2(D)$ for $d = 2,3$ and in $H^{-1}(D)$ for $d = 4,5$, converges to a Gaussian distribution that can be written as a stochastic integral as above, but with $dW$ replaced by $\dot{W}^\alpha dy$ where $\dot{W}^\alpha$ is a centered Gaussian random field with correlation function $|x-y|^{-\alpha}$, and $\sigma$ replaced by some other statistical parameter; see Theorem \ref{thm:main3} for details.

The study of limiting distribution of the homogenization error goes back to \cite{FOP-82}, where the Laplace operator with a random potential formed by Poisson bumps was considered. General random potential with short range correlations was considered recently in \cite{B-CLH-08}, and in \cite{BJ-DCDS-10, BJ-CMS-11} for other non-oscillatory differential operators with random potential. Long range correlated random potential was considered in \cite{BGGJ12_AA}. When oscillatory differential operators are considered, limiting distribution of homogenization error was obtained in \cite{BP-99} for short range correlated elliptic coefficients, and in \cite{BGMP-AA-08} for long range correlated case, all in the one dimensional setting. The main results of this paper show that the general framework developed in \cite{B-CLH-08,BJ-CMS-11,BGGJ12_AA}, in order to characterize the random fluctuation caused by the random potential, applies even when there are oscillations in the differential operators, as long as these oscillations are not statistically related to those of the random potential. 

Our approach is as follows: we introduce an auxiliary problem with periodic diffusion coefficients and homogenized potential; let $v^\eps$ be the solution. Then the deterministic homogenization error $\E u^\eps - u$ is essentially characterized by $v^\eps - u$, which amounts to classical periodic homogenization theory. The random fluctuation $u^\eps - \E u^\eps$ is then the same as $(u^\eps - v^\eps) - \E(u^\eps - v^\eps)$, which can be represented as a truncated Neumann series. The first term $X^\eps$ in this series contributes to the limiting distribution. By Prohorov's theorem, we need to show that the probability measures of $\{X^\eps\}$ is tight in the proper Hilbert space, and that their characteristic functions converges. The latter is essentially the convergence in distribution of the integration of $X^\eps$ against test functions, in view of the uniform in $\eps$ estimates of the Green's functions associated to the oscillatory diffusion, this step is the same as the earlier setting with non-oscillatory diffusion. The role of oscillations in the diffusion, however, becomes prominent in the step of proving tightness of the measures of $\{X^\eps\}$. The simple and natural method used in \cite{BGGJ12_AA} fails completely; see section \ref{s.discussion} for details. New ideas are needed: we obtain tightness of the the measures of $\{X^\eps\}$ in $L^2(D)$ by controlling the mean square of $H^s$ norm of $\{X^\eps\}$, for some $0 < s < \frac{1}{2}$; similarly, we get tightness in $H^{-1}(D)$ by controlling the mean square of $H^{-s}$ norm with $\frac{1}{2} < s < 1$. The constraints on the spatial dimension $d$ arises naturally in the proof of such controls.

Our analysis relies on uniform estimates of the Green's function associated to the periodic homogenization problem; we refer to \cite{AL87,AL91_Lp} for the classical results, and to \cite{KLS12_ARMA,KLS14_GN} for recent development in this direction. We refer to \cite{MaO14,GO14} for recent results on uniform estimates of Green's function for equations with highly oscillatory random diffusion coefficients in spatial dimension higher than one. We remark also that in this setting, limiting distribution of the corrector function and that of the full random fluctuation $u^\eps - \E u^\eps$, in negative H\"older space, were obtained in \cite{MN15} and \cite{GM15} respectively, in the discrete setting; see also \cite{MO14}. Such results are apparently more challenging to obtain, and the proofs are based on calculus in the (infinite dimensional) probability space.%, and the random field is usually constructed from countable i.i.d. random variables, or is assumed to satisfy spectral gap type of conditions that facilitate the calculus.

The rest of this paper is organized as follows: in section \ref{sec:results} we make precise the main assumptions on the parameters of the homogenization problem, in particular on the properties of the random potential, and state the main results in the short range correlation setting. Homogenization of \eqref{e.rpde} and some useful results on periodic homogenization theory are recalled in section \ref{s.homog}. Section \ref{s.rerror} and \ref{s.cdist} are devoted to the proof of the main results, where we characterize how the random fluctuation scales in energy norm and in the weak topology, and determine the limiting distribution of the scaled fluctuation. We present new methods to prove the tightness of the probability measures of the random fluctuations. In section \ref{s.lrc} we state and prove the corresponding results in the long range correlation setting. We make some comments and further discussions in section \ref{s.discussion} and prove some technical results, such as tightness criterion for probability measures, in the appendix.

%%%%%%%%%%%%%%
%%%%%%%%%%%%%%

\section{Assumptions, Preliminaries and Main Results}
\label{sec:results}

%%%%%%%%%%%%%
\subsection{Assumptions on the coefficients}

Throughout this paper, we assume that the domain $D$ in \eqref{e.rpde} is an open bounded set of $\R^d$ with $C^{1,1}$ boundary. The coefficients $a_{ij}(\frac{x}{\eps})$ and $q(\frac{x}{\eps},\omega)$ are the scaled version of $a_{ij}(x)$ and $q(x,\omega)$. We make the following main assumptions on $a_{ij}$ and $q$.

\medskip

\noindent{\bf Periodic diffusion coefficients.} For the functions $(a_{ij})$, we assume:

\begin{enumerate}

\item[(A1)] (Periodicity) The function $A := (a_{ij}): \R^d \to \R^{d\times d}$ is periodic. That is, for all $x \in \R^d$, $k \in \Z^d$ and $i, j = 1,2,\cdots,d$, we have
\begin{equation}\label{e.Aper}
a_{ij}(x+k) = a_{ij}(x).
\end{equation}

\item[(A2)] (Uniform ellipticity) For all $y \in \bT^d$, the matrix $A(y) = (a_{ij}(y))$ is uniformly elliptic in the sense that, for all $\xi \in \R^d$, one has
\begin{equation}\label{e.Aelliptic}
\lambda |\xi|^2 \le \xi^T A(y) \xi = \sum_{i,j=1}^d \xi_i a_{ij}(y) \xi_j \le \Lambda |\xi|^2
\end{equation}

\item[(A3)] (Smoothness) For some $\gam, M$ with $\gam \in (0,1]$ and $M > 0$, one has
\begin{equation}\label{e.Areg}
\|A\|_{C^\gam(\bT^d)} \le M.
\end{equation} 

\end{enumerate}
We henceforth refer to the above assumptions together as $(\text{A})$.
%%%%%%%%%%%%%

\medskip

\noindent{\bf Random potential.} For the random field $q(x,\omega)$ on the probability space $(\Omega,\cF,\bP)$, we assume:

\begin{enumerate}

\item[(P)] (Stationarity and Ergodicity) There exists an ergodic group of $\bP$-preserving transformations $(\tau_x)_{x\in \R^d}$ on $\Omega$, where ergodicity means that $E \in \cF$ and
\begin{equation*}
\tau_x E = E, \quad \text{ for all } x \in \R^d
\end{equation*}
imply that $\bP(E) \in \{0,1\}$. The random potential $q(y,\omega)$ is given by $\tilde{q}(\tau_y \omega)$ where $\tilde{q}: \Omega \to \R$ is a random variable satisfying
\begin{equation}\label{e.qbdd}
0 \le \tilde{q}(\omega) \le M, \quad \text{ for all } \omega \in \Omega.
\end{equation}
\end{enumerate}

\medskip

\noindent{\bf Further assumptions on $q$.} The above assumptions are sufficient for proving homogenization result. However, to estimate the size of the homogenization error and to characterize the limiting distribution of the random fluctuation, more assumptions on the random field $q(\cdot,\omega)$ are necessary.

To simplify notations, we write in the sequel
\begin{equation*}
q(x,\omega) = \ol{q} + \nu(x,\omega)
\end{equation*}
where $\ol{q}$ is the mean of $q$ and $\nu$ is the fluctuation. Note that $\ol{q}$ is a deterministic constant and $\nu$ is a mean zero stationary ergodic random field. The auto-correlation function $R(x)$ of $q$ (and hence $\nu$) is defined as
\begin{equation}\label{e.qdec}
R(x) = \E \left[ \nu(x+y,\omega) \nu(y,\omega) \right], \quad \quad \sigma^2 := \int_{\R^d} R(x) dx.
\end{equation}
By Bochner's theorem, $R(x)$ is a positive definite function and $\sigma^2 \ge 0$. We assume that $\sigma > 0$. When $R$ is integrable on $\R^d$, i.e. $\sigma^2 < \infty$, we say that $q$ has short range correlations; we say $q$ has long range correlations if otherwise.

%%%%%%%%%%%
\bigskip

\noindent {\bf Short range correlated random fields.} In this case, we make an assumption on the rate of decay of the correlation function. We denote by $\mathcal{C}$ the set of compact sets in $\R^d$, and for two sets $K_1, K_2$ in $\mathcal{C}$, the distance $d(K_1,K_2)$ is defined to be
$$
d(K_1,K_2) = \min_{x \in K_1, y \in K_2} \, |x - y|.
$$
Given any compact set $K \subset \mathcal{C}$, we denote by $\cF_K$ the $\sigma$-algebra generated by the random variables $\{q(x) \,:\, x \in K\}$. We define the ``maximal correlation coefficient" $\varrho$ of $q$ as follows: for each $r > 0$, $\varrho(r)$ is the smallest value such that the bound
\begin{equation}\label{e.varrho}
\E \left( \varphi_1(q) \varphi_2(q)\right) \le \varrho(r) \sqrt{\E \left(\varphi^2_1(q)\right) \, \E \left(\varphi^2_2(q) \right)}
\end{equation}
holds for any two compact sets $K_1, K_2 \in \mathcal{C}$ such that $d(K_1,K_2) \ge r$ and for any two random variables of the form $\varphi_i(q)$, $i=1,2$, such that $\varphi_i(q)$ is $\cF_{K_i}$-measurable and $\E \varphi_i(q) = 0$. We assume that

\begin{enumerate}
\item[(S)] The maximal correlation function satisfies $\varrho^{\frac 1 2} \in L^1(\R_+,r^{d-1} dr)$, that is
$$
\int_0^\infty \varrho^{\frac 1 2}(r) r^{d-1} dr < \infty.%(1+\log(r\vee 1)) dr < \infty.
$$
\end{enumerate}

Assumptions on the mixing coefficient $\varrho$ of random media have been used in \cite{B-CLH-08, BJ-CMS-11,HPP13}; we refer to these papers for explicit examples of random fields satisfying the assumptions. We note that the auto-correlation function $R(x)$ can be bounded by $\varrho$. For any $x \in \R^d$,
\begin{equation*}
|R(x)| = |\E(q(x)-\E q)(q(0) - \E q)| \le \varrho(|x|) \mathrm{Var} (q).
\end{equation*}
By \eqref{e.qbdd}, $q$ and hence its variance is bounded. In view of (S) and the fact that one can assume $\varrho \in [0,1]$ (hence $\varrho \le \sqrt{\varrho}$), we find that $R$ is integrable. Therefore, (S) implies that $q(x,\omega)$ has short range correlations. In fact, (S) is a much stronger assumption, and not necessary for the main results of this paper to hold. In section \ref{s.discussion} we will provide alternative and less restrictive assumptions that is sufficient. However, using the assumption (S) and Lemma \ref{l.cumu} below, we can simplify significantly certain fourth order moment estimates of the random potential $\nu(x,\omega)$; such estimates appear often in the study of limiting distribution of the homogenization error.
%%%%%%%%%%%%
% Long range assumptions
\bigskip

\noindent{\bf Long range correlated potentials constructed from Gaussian fields.} There is no general tools like central limit theorems when the random field has long range correlations. We restrict to a specific family of random fields, constructed from long range correlated Gaussian random fields. Let $q(x,\omega) = \ol{q} + \nu(x,\omega)$ with $\ol{q}$ being a nonnegative constant; we assume

\begin{enumerate}
\item[(L1)]  $\nu(x,\omega) = \Phi(g(x))$; $g(x,\omega)$ is a centered stationary Gaussian random field 
with unit variance. Furthermore, the correlation function of $g(x,\omega)$ has heavy 
tail. That is, for some positive constant $\kappa_g$ and some real number $\alpha \in (0, 
d)$,
\begin{equation}
R_g(x) := \E\{g(y,\omega)g(y+x,\omega)\} \sim \kappa_g|x|^{-\alpha} \ \mbox{as}\ 
|x|\rightarrow\infty.
\label{e.tail}
\end{equation}

\item[(L2)] The function $\Phi: \R \to \R$ 
satisfies $-\ol{q} \le \Phi \le M - \ol{q}$, and has Hermite rank one, i.e.
\begin{equation}
\int_\R \Phi(s) e^{-\frac{s^2}{2}} ds = 0, \quad \quad V_1 := \int_\R s\Phi(s) e^{-\frac{s^2}{2}} ds \ne 0.
\label{e.odd}
\end{equation}

\item[(L3)]The Fourier transform $\hat{\Phi}$ of the function $\Phi$ satisfies
\begin{equation}
\label{e.Phihat}
\int_{\R} |\hat{\Phi}(\xi)| \big( 1 + |\xi|^3 \big) < \infty.
\end{equation}
\end{enumerate}
We henceforth refer assumption (L) to the above conditions all together.

\medskip

The assumption (L2) makes $\nu(x,\omega) = \Phi(g(x,\omega))$ mean zero, and the bounds on $\Phi$ ensures that $0 \le q(x,\omega) \le M$, which is \eqref{e.qbdd}. From the above construction, we check that $\nu(x,\omega)$ is stationary ergodic and has long-range correlation function that 
decays like $\kappa |x|^{-\alpha}$, $\kappa = V_1^2 \kappa_g$; see Lemma \ref{lem:tail} for the details. Assumption (L3) allows one to derive a (non-asymptotic) estimate, see Lemma \ref{lem:fourth} in the appendix, for the fourth-order moments of $\nu(x,\omega,\omega)$.  

\medskip

\noindent{\bf Notations.} Throughout the paper, by {\it universal parameters} we refer to $\lambda, \Lambda, \gamma$ and $M$ in the assumptions (A), the autocorrelation function $R$, $\sigma^2$, and the mixing coefficients $\varrho$ in the short range correlation setting, $\alpha, R_g, \kappa_g, \Phi$ and $\kappa$ in the long range correlation setting, the domain $D$ and its boundary $\partial D$, and the dimension $d$. If a constant $C$ depends only on these parameters, we say either $C$ depends on universal parameters or $C$ is a universal constant. For the random potential $\nu(x,\omega)$ and the functions $\varrho(x)$, $R(x)$, etc., which are related to $\nu$, we use $\nu^\eps, \varrho^\eps, R^\eps$, etc., to denote the scaled version. For instance, $\nu^\eps(x,\omega)$ is short-hand notation for $\nu(\frac{x}{\eps})$. We use the notation $H^s(K)$, $s \ge 0$, for the Sobolev or the fractional Sobolev space $W^{s,2}(K)$ on some domain $K \subset \R^d$; when $K$ is bounded, we use $H^{s}_0(K)$ for the subspace that consists of functions having trace zero at $\partial K$; note that $H^s_0(\R^d) = H^s(\R^d)$. We denote by $H^{-s}(K)$, $s>0$, the dual space $(H^s_0(K))'$.  For any Hilbert space $\cH$, $(\cdot,\cdot)_{\cH}$ denotes the inner product in $\cH$; when $\cH = L^2(D)$, we very often omit the subscript and write $(\cdot,\cdot)$ instead. We use $\langle f, g \rangle$ whenever the formal integral $\int_D fg$ makes sense. We typically use $\bone_A$ for the indication function of a set $A \subset \R^d$, or if $A$ is a statement the indication function of $A$ being true. Finally, for two real numbers $a$ and $b$, $a\wedge b$ is a short-hand notation for $\min\{a,b\}$, and $a \vee b$ means $\max\{a,b\}$. 

%\medskip

%{\bf Notations.}
%%%%%%%%%%%%%
\subsection{Probability distribution on functional spaces}

We view the random fluctuation $u^\eps - \E u^\eps$ in the homogenization error as random elements in certain functional spaces, and aim to find the limit of its law in that space. It turns out that the choice of functional spaces depends on the spatial dimension $d$.

When $d = 1$, one can choose the space  $C(D)$ of continuous functions. In fact, convergence in distribution in $C(D)$ was proved in \cite{B-CLH-08}, for random diffusion coefficient $a(x,\omega)$ with random potential $q(x,\omega)$, both having short range correlations. In this paper, we prove that for $d = 2,3$, the space can be chosen as $L^2(D)$ and for $d= 4,5$, the space can be chosen as $H^{-1}(D)$. Note that both choices are Hilbert spaces. We recall some facts concerning weak convergence of probability measures on Hilbert spaces. We refer to the books of Billingsley \cite{B-CPM} and Parthasarathy \cite{Partha} for more details.

%%%%%%%%%%%%%
\subsubsection*{Probability distributions on a Hilbert space}

Let $\cH$ be a separable Hilbert space, and let $X(\omega)$ be a $\cH$ valued random element on the probability space $(\Omega,\cF,\bP)$. Then $X$ determines a probability measure $P^X$ on $(\cH,\mathcal{B}(\cH))$, where $\mathcal{B}(\cH)$ denotes the Borel $\sigma$-algebra generated by open sets in $\cH$, by
\begin{equation}
  P^{X}(\mathcal{S}) = \bP(X \in \mathcal{S}), \quad
  \quad \text{for any } \mathcal{S} \in \mathcal{B}(\cH).
\end{equation}

We say a family $\{X^\eps\}_{\eps \in (0,1)}$ of random elements in $\cH$ converge in probability distribution (or in law), as $\eps \to 0$, to another random element $X$ on $\cH$, if the probability measures $P^{X^\eps}$ converges weakly to $P^X$, i.e. for any real bounded continuous functional $f : \cH \to \R$, 
\begin{equation*}
  \int_{\cH} f(g) \ dP^{X^\eps}(g) \longrightarrow \int_{\cH} f(g) \ dP^X(g).
\end{equation*}
In particular, any probability measure $P$ on a separable Hilbert space $\cH$ is determined by its characteristic function $\phi^P: \cH \to \mathbb{C}$,
\begin{equation}
  \phi^P(h) = \int_{\cH} e^{i(h,g)_{\cH}} \ dP (g),
\end{equation}
Moreover, the following result holds:

\begin{theorem}[{\cite[Chapter VI, Lemma 2.1]{Partha}}]\label{t.Partha} Let $\{X^\eps\}_{\eps\in(0,1)}$ and $X$ be random elements in $\cH$, possibly defined on different probability spaces. $X^\eps$ converges to $X$ in law in $\cH$, as $\eps \to 0$, if the family of probability measures $\{P^{X^\eps}\}_{\eps \in (0,1)}$ is tight and for any $h \in \cH$,
  \begin{equation}\label{e.t.Partha1}
	\lim_{\eps \to 0} \phi^{P^{X^\eps}}(h) = \phi^{P^X} (h).%, \quad \quad \text{for all } h \in L^2(D).
  \end{equation}
\end{theorem}

\begin{remark}\label{r.Partha} Let $\cH = L^2(D)$, which is a separable Hilbert space, and let $X$ be a random element in $L^2(D)$ defined on the probability space $(\Omega,\cF,\bP)$. The characteristic function of $P^X$ can be calculated as follows: for any $h \in L^2(D)$, 
  \begin{equation}
	\phi^{P^X}(h) = \int_\R e^{iz} dP^X(\{(h,g) > z\}) = \int_\R e^{iz} d\,\bP(\{(h,X(\omega)) > z\}) = \E \,e^{i(h,X)}.
	\label{e.t.Partha2}
  \end{equation}
  Therefore, to prove that $X^\eps$ converges in distribution to $X$ as $L^2$ paths, it suffices to show that $\{P^{X^\eps}\}$ is tight and that for any $h \in L^2(D)$,
  \begin{equation}
	(h,X^\eps) \xrightarrow{\mathrm{distribution}} (h,X),
	\label{e.t.Partha3}
  \end{equation}
that is, the random variables $(h,X^\eps)$ converges in distribution to the random variable $(h,X)$.
\end{remark}

In Theorem \ref{t.tight} in the appendix, we provide a tightness criterion for $\{P^{X^\eps}\}$ on $L^2(D)$, with the assumption that $\{X^\eps(\cdot,\omega)\}$ are in $H^s_0(D)$ for certain $s > 0$. The criterion is sufficient but by no means necessary. Nevertheless, it is very handy for our analysis since the random fields $X^\eps$ that we are dealing with come from solutions of \eqref{e.rpde}, and hence are naturally in $H^s_0(D)$.

%%%%%%%%%%%%
% Main results
\subsection{Main results}

We now state the main results of the paper under the assumption that $q(x,\omega)$ has {\it short range correlations}. Analogous results for long range correlation setting will be presented in section \ref{s.lrc}. 

The first main theorem concerns how does the homogenization error scales.

\begin{theorem} \label{t.size}
Let $D \subset \R^d$ be an open bounded $C^{1,1}$ domain, $u^\eps$ and $u$ be the solutions to \eqref{e.rpde} and \eqref{e.hpde} respectively. Suppose that {\upshape(A)(P)} and {\upshape(S)} hold, $f \in L^2(D)$ and $2 \le d \le 7$. Then, there exists positive constant $C$, depending only on the universal parameters, such that
\begin{equation}\label{e.t.size}
\E\,\|u^\eps - u\|_{L^2} \le C \eps \|f\|_{L^2}.
\end{equation}
Moreover,
\begin{equation}\label{e.t.size.f}
\E\,\|u^\eps - \E u^\eps\|_{L^2} \le \begin{cases}
C \eps^{2\wedge \frac{d}{2}} \|f\|_{L^2}, \quad&\text{if }\; d \ne 4,\\
C \eps^2 |\log \eps|^{\frac 1 2} \|f\|_{L^2}, \quad&\text{if }\; d = 4.
\end{cases}
\end{equation}
Moreover, for any $\varphi \in L^2(D)$,
\begin{equation}\label{e.t.size.fw}
\E ~\left|(u^\eps - \E u^\eps, \varphi)_{L^2}\right| \le C \eps^{\frac d 2} \|\varphi\|_{L^2} \|f\|_{L^2}.
\end{equation}
\end{theorem}

This theorem provides $L^1(\Omega,L^2(D))$ estimates of $u^\eps - u$ and its random part, and its proof is detailed in section \ref{s.rerror}. We note that the size of the full homogenization error is much larger than that of its random part. This is because the oscillations in the diffusion coefficients cause some deterministic fluctuation in the solution of size $O(\eps)$, as in standard periodic homogenization. The additional random fluctuation caused by the short range correlated random potential scales like $\eps^{(d\wedge4)/2}$ in energy norm, and scales like $\eps^{d/2}$ in the weak topology. These results agree with the case of non-oscillatory diffusion coefficients; see \cite{B-CLH-08, BJ-CMS-11}. The next result exhibits the limiting law of the rescaled random fluctuation $\eps^{-\frac{d}{2}}(u^\eps - \E u^\eps)$.

%%%%%%%%
\begin{theorem}
\label{t.sl.23}
Suppose that the assumptions in Theorem \ref{t.size} hold. Let $\sigma$ be defined as in \eqref{e.qdec} and $G(x,y)$ be the Green's function of \eqref{e.hpde}. Let $W(y)$ denote the standard $d$-parameter Wiener process. Then 
\begin{itemize}
\item[\upshape(i)]  For $d = 2, 3$, as $\eps \to 0$, 
\begin{equation}\label{e.sl.23}
\frac{u^\eps - \E u^\eps}{\sqrt{\eps^d}} \xrightarrow{\rm{distribution}} \sigma \int_D G(x,y) u(y) dW(y), \quad \text{ in  } L^2(D).
\end{equation}
\item[\upshape(ii)] For $d = 4,5$, as $\eps \to 0$, the above holds as convergence in law in $H^{-1}(D)$.
\end{itemize}
\end{theorem}

The proof of item (i) above can be found towards the end of section \ref{s.cL2}, and that of item (ii) is at the end of section \ref{s.cHs}.

\begin{remark} The integral on the right hand side of \eqref{e.sl.23} is understood, for each fixed $x$, as Wiener integral in $y$ with respect to the multiparameter Wiener process $W(y)$. Let $X$ denote the result. For $d = 2,3$, because the Green's function $G(x,y)$ is square integrable, $X$ is a random element in $L^2(D)$. For $d = 4,5$, $X$ is understood through the Fourier transform of its distribution: given $h^* \in H^{-1}(D)$, $\phi^{P^X}(h^*)$ is defined to be $\mathbf{E} \exp \left( i \sigma \int_D \langle G(\cdot,y), h^*(\cdot) \rangle u(y) dW(y) \right)$, where $\mathbf{E}$ is the expectation with respect to the law of $W$.
\end{remark}
%%%%%%%%%%%%%%
%%%%%%%%%%%%%%

\section{Homogenization and Periodic Error Estimates}
\label{s.homog}

The following homogenization result for \eqref{e.rpde}, without the random potential $q^\eps(x,\omega)$, is well known. The effect of the presence of $q^\eps$ turns out to be minor for homogenization;  nevertheless, we include a proof here for the sake of completeness.

\begin{theorem}
\label{t.homog} Assume {\upshape (A1)(A2)} and {\upshape (P)}. Then there exists $\Omega_1 \in \F$ such that $\bP(\Omega_1) = 1$, and for all $\omega \in \Omega_1$, the solution $u^\eps$ of \eqref{e.rpde} converges to the solution $u$ of \eqref{e.hpde} weakly in $H^1(D)$ and strongly in $L^2(D)$, for any $f \in H^{-1}(D)$.
\end{theorem}

Let $\cL_\eps$ denote the differential operator
\begin{equation}\label{e.Leps}
-\frac{\partial}{\partial x_i} \left( a_{ij} \left(\frac{x}{\eps}\right) \frac{\partial}{\partial x_j} \right) + \ol{q},
\end{equation}
and let $\cL^{\eps,\omega}$ be the differential operator $\cL_\eps + \nu(\frac{x}{\eps},\omega)$. We remark that $\cL_\eps$ has highly oscillatory but deterministic coefficients while $\cL^{\eps,\omega}$ has, in addition, a highly oscillatory and random potential. Let $\cG^{\eps,\omega}$ and $\cG_\eps$ be the solution operator of the Dirichlet boundary problems associated to $\cL^{\eps,\omega}$ and $\cL_\eps$. %This means $\cG^{\eps,\omega}: H^{-1}(D) \to H^1_0(D)$ maps $f \in H^{-1}(D)$ to the function $\cG^{\eps,\omega} f$ which solves \eqref{e.rpde}.
Owing to the conditions \eqref{e.Aelliptic} and \eqref{e.qbdd}, $\cG^{\eps,\omega}$ is well defined for any $\omega \in \Omega$. Moreover, we have the standard estimate, for any $\omega \in \Omega$ and $\eps > 0$,
\begin{equation}
\| \cG^{\eps,\omega} f\|_{H^1(D)} \le C \|f\|_{H^{-1}(D)}
\end{equation}
with some constant $C$ that depends on on the universal parameters, and neither on $\omega$ nor $\eps$. By the same token, $\cG_\eps$ is well defined and shares the same estimate above.

\begin{proof}[Proof of Theorem \ref{t.homog}] {\it Step 1}. For each $\omega \in \Omega$, the solution $u^\eps$ of \eqref{e.rpde} is given by $\cL^{\eps,\omega} f$, which satisfies the standard estimates
\begin{equation*}
\|u^\eps\|_{H^1(D)} + \|A^\eps \nabla u^\eps\|_{L^2(D)} + \|q^\eps(x,\omega) u^\eps\|_{L^2(D)} \le C,
\end{equation*}
where $C$ depends the universal parameters and $f$, and is uniform in $\eps$ and $\omega$. As a result, due to the compact embeddings $H^1(D) \hookrightarrow L^2(D) \hookrightarrow H^{-1}(D)$,  through a subsequence $\eps_j(\omega) \to 0$ which by an abuse of notation is still denoted by $\eps$, we have
\begin{equation}\label{e.t.homog0}
\begin{aligned}
&\nabla u^\eps(\cdot,\omega) \xrightharpoonup[\eps \to 0]{L^2} \nabla v(\cdot,\omega), &\quad \quad  &A\left(\frac{\cdot}{\eps}\right)\nabla u^\eps(\cdot,\omega) \xrightharpoonup[\eps \to 0]{L^2} \nabla \xi(\cdot,\omega),\\
&u^\eps(\cdot,\omega) \xrightarrow[\eps \to 0]{L^2} v(\cdot,\omega), &\quad\quad
&q\left(\frac{\cdot}{\eps},\omega\right) u^\eps(\cdot,\omega) \xrightarrow[\eps \to 0]{H^{-1}} p(\cdot,\omega).
\end{aligned}
\end{equation}
for some function $v(\cdot,\omega) \in H^1(D)$ and some vector valued function $\xi(\cdot,\omega) \in [L^2(D)]^d$.

{\it Step 2}. Recall that $\{\chi^k\}_{k=1}^d$ are the correctors defined in \eqref{e.corrector}, and we can extend them periodically to functions defined on $\R^d$. Since $A(y)(e_k + \nabla \chi^k(y))$ is periodic, we have that
\begin{equation}\label{e.t.homog1}
A\left(\frac{x}{\eps}\right) \left(e_k + \left(\nabla \chi^k\right) \left(\frac{x}{\eps}\right) \right) \xrightharpoonup{\ L^2\ } \int_{\bT^d} A(y)(e_k(y) + \nabla \chi^k(y)) dy = \ol{A} e_k.
\end{equation}
For the same reason and the fact $\int_{\bT^d} \nabla \chi^k dy = 0$, we have
\begin{equation}\label{e.t.homog2}
e_k + \left(\nabla \chi^k\right) \left(\frac{x}{\eps}\right) \xrightharpoonup{\ L^2\ } \int_{\bT^d} e_k + \nabla \chi^k(y) dy = e_k.
\end{equation}
Now fix an arbitrary function $\varphi \in C^\infty_0(D)$. For each fixed $\omega \in \Omega$, let $\eps(\omega) \to 0$ be the subsequence in Step 1. Consider the integral
\begin{equation*}
\int_D A\left(\frac{x}{\eps}\right) \nabla u^\eps(x,\omega) \cdot \nabla \left\{x_k + \eps \chi^k\left(\frac{x}{\eps}\right)\right\} \varphi(x) dx.
\end{equation*}
On one hand, in view of the third item in \eqref{e.t.homog0}, \eqref{e.t.homog2}, and the facts that $\div(A^\eps \nabla u^\eps) = -f + q^\eps u^\eps$ converges in $H^{-1}$ (to $-f + p(\cdot,\omega)$ where $p$ is defined in \eqref{e.t.homog0}) and that $e_k + (\nabla \chi^k)(x/\eps)$ is curl-free, by the div-curl lemma \cite[Lemma 1.1]{Jikov_book}, the above integral satisfies
\begin{equation*}
\lim_{\eps \to 0} \ \int_D A\left(\frac{x}{\eps}\right) \nabla u^\eps(x,\omega) \cdot \nabla \left\{x_k + \eps \chi^k\left(\frac{x}{\eps}\right)\right\} \varphi(x) dx = \int_D \xi(x,\omega) \cdot e_k \varphi(x) dx.
\end{equation*}
On the other hand, in view of the first item in \eqref{e.t.homog0}, \eqref{e.t.homog1}, and the facts that $\div(A^\eps (e_k + \nabla \chi^k(x/\eps)))$ converges in $H^{-1}$ (they all equal to zero) and that $\nabla u^\eps$ is curl-free, by the div-curl lemma, we have
\begin{equation*}
\lim_{\eps \to 0} \ \int_D A\left(\frac{x}{\eps}\right) \nabla u^\eps(x,\omega) \cdot \nabla \left\{x_k + \eps \chi^k\left(\frac{x}{\eps}\right)\right\} \varphi(x) dx = \int_D \nabla v \ol{A} \cdot e_k \varphi(x) dx.
\end{equation*}
The two limits above must be equal, and it follows that $\xi(\cdot,\omega) = \ol{A}\nabla v(\cdot,\omega)$ in distribution.

{\it Step 3}. Recall that the stationary random potential $q(x,\omega)$ can be written as $\tilde{q}(\tau_x \omega)$ where $\tilde{q}$ is an essentially bounded random variable on $\Omega$. By Birkhoff ergodic theorem \cite[Theorem 7.2]{Jikov_book}, there exists $\Omega_1 \in \F$ with $\bP(\Omega_1) = 1$, and for each $\omega \in \Omega_1$, the whole sequence
\begin{equation}\label{e.t.homog3}
q\left(\frac{x}{\eps},\omega\right) = \tilde{q} \left( \tau_{\frac x \eps} \omega \right) \xrightharpoonup{\ L^\alpha_{\mathrm{loc}}(\R^d) \ } \ol{q} = \E q(0,\omega),
\end{equation}
for any $\alpha \in (1,\infty)$. From the weak formulation of $u^\eps$, for any $\omega \in \Omega_1$ and for any $\varphi \in C^\infty_0(D)$, we have
\begin{equation*}
\int_D A\left(\frac{x}{\eps}\right) \nabla u^\eps(x,\omega) \nabla \varphi(x) dx + \int_D q\left(\frac{x}{\eps},\omega\right) u^\eps(x) \varphi(x) dx = \int_D f(x) \varphi(x) dx.
\end{equation*}
Pass to the limit along the subsequence $\eps(\omega)$ found in Step 1, we have
\begin{equation*}
\int_D \ol{A}\nabla v \cdot\nabla \varphi + \int_D \ol{q} v(x) \varphi(x) dx = \int_D f(x) \varphi(x) dx - \lim_{\eps \to 0} \int_D q\left(\frac{x}{\eps},\omega\right) (u^\eps - v)\varphi(x) dx.
\end{equation*}
The first term on the left follows from the div-curl lemma and the fact that $\xi = \ol{A}\nabla v$; the second term on the left is due to \eqref{e.t.homog3}. Finally, the last term on the right hand side is zero since $q$ is uniformly bounded and $u^\eps - v$ converges to zero strongly in $L^2(D)$. Consequently, the above limit shows that $v$ solves the homogenized equation \eqref{e.hpde}. By uniqueness of the homogenized problem, we must have that $v = u$ and $v$ is deterministic.

Finally, for each $\omega \in \Omega_1$, by the weak compactness in $H^1(D)$ and the uniqueness of the possible limit, the whole sequence $u^\eps$ converges to $u$. This proves the homogenization theorem.
\end{proof}

\begin{remark} We remark that the same proof works in the case when $(a_{ij})$ is not symmetric; indeed, it suffices to replace $\chi^k$ above by the solution of the adjoint corrector equation. The same idea of proof can also be carried out in the case when $(a_{ij})$ are stationary ergodic random fields; indeed, the corrector equation in that case is much more involved but, by now, its solution and analogs of \eqref{e.t.homog1} and \eqref{e.t.homog2}, are well known.
\end{remark}

%%%%%%%%%%%
\subsection{Decomposition of the homogenization error}

To separate the fluctuations in the homogenization error $u^\eps - u$ that are due to the periodic oscillations in the diffusion coefficients from those due to the random potential, we introduce the function $v^\eps$ which solves the following deterministic problem:
\begin{equation}
\label{e.opde}
\left\{
\begin{aligned}
&-\frac{\partial}{\partial x_i} \left(a_{ij} \left(\frac{x}{\eps}\right) \frac{\partial}{\partial x_j} v^\eps(x,\omega)\right) + \ol{q} v^\eps(x,\omega) = f(x),\quad &&\quad x\in D,\\
& v^\eps(x) = 0, \quad &&\quad x \in \partial D.
\end{aligned}
\right.
\end{equation}
Here, the potential field is already homogenized, and we expect that $v^\eps - u$ filters out the effect of the random potential. The problem above is well posed and its solution $v^\eps$ is given by $\cG_\eps f$.

The standard periodic homogenization theory yields that $v^\eps$ converges weakly in $H^1(D)$ and strongly in $L^2(D)$ to $u$, for any $f \in H^{-1}(D)$. Using this function, we can write the homogenization error for \eqref{e.rpde} as
\begin{equation}\label{e.uepsd}
u^\eps - u = (u^\eps - v^\eps) + (v^\eps - u).
\end{equation}
The deterministic part of the homogenization error is
\begin{equation}\label{e.uepsd.d}
\E u^\eps - u = \E (u^\eps - v^\eps) + (v^\eps - u),
\end{equation}
and the random fluctuation part of the homogenization error is
\begin{equation}\label{e.uepsd.f}
u^\eps - \E u^\eps = (u^\eps - v^\eps) - \E (u^\eps - v^\eps).
\end{equation}

The deterministic part of the homogenization error hence contains two parts, the mean of $u^\eps - v^\eps$ and the periodic homogenization error $v^\eps - u$. Estimates for the second part amounts to convergence rate of periodic homogenization, and we recall some of the well known results below, together with uniform in $\eps$ estimates of the Green's function associated to $\cG_\eps$.  We postpone the estimates for $\E(u^\eps - v^\eps)$ to the next section.
%
%
%The periodic corrector functions $\chi^k$ are introduced in \eqref{e.corrector}. Define
%\begin{equation*}
%u_1(x,y) = \chi^k(y) \frac{\partial u}{\partial x_k}(x), \quad \quad
%\text{and}
%\quad\quad u^\eps_1(x) = \eps u_1\left(x,\frac{x}{\eps}\right) = \eps \chi^k\left(\frac{x}{\eps}\right) \frac{\partial u}{\partial x_k}(x).
%\end{equation*}
%Set $\theta^\eps(x)$ to be the solution of
%\begin{equation}\label{e.theta}
%\left\{
%\begin{aligned}
%&-\cL_\eps \theta^\eps(x) = 0, &\quad &x \in D,\\
%&\theta^\eps(x) = u_1\left(x,\frac{x}{\eps}\right), &\quad &x \in \partial D.
%\end{aligned}
%\right.
%\end{equation}

\begin{theorem}[Estimates in periodic homogenization] \label{t.homogp} %Basically, can cite Avelleneda and Lin, and describe the argument by Kenig, Lin and Shen for the error estimate in L^2 for C^{1,\alpha} domains. Cite also Moskow and Vogelius for the special d=2 case
Let $D \subset \R^d$ be an open bounded $C^{1,1}$ domain, $v^\eps$ and $u$ be the solutions to \eqref{e.opde} and \eqref{e.hpde} respectively. Let $G_\eps(x,y)$, $x,y \in D$, be the Green's function associated to the Dirichlet problem of \eqref{e.opde}. Assume {\upshape(A)}. Then there exists positive constant $C$, depending only on the universal parameters, such that
\begin{itemize}
\item[\upshape(i)] For any $f \in L^2(D)$, $\|v^\eps - u\|_{L^2} \le C \eps \|f\|_{L^2}$.

\item[\upshape(ii)] For $d\ge 2$ and for any $x,y \in D$, $x \ne y$, $G_\eps(x,y)$ satisfies
\begin{equation}
|G_\eps(x,y)| \le \begin{cases} C|x-y|^{2-d}, \quad &\text{if } d\ne 2\\
C(1+|\log |x-y| |), \quad &\text{if } d = 2
\end{cases}
\label{e.Gsing}
\end{equation}
and
\begin{equation}
|\nabla G_\eps(x,y)| \le C|x-y|^{1-d}.
\label{e.DGsing}
\end{equation}
\end{itemize}
\end{theorem}

The $O(\eps)$ error estimates in $L^2$ was proved in \cite{MV97} for $d=2$, and in \cite{G06} for general $C^{1,1}$ domains; see also \cite{KLS12_ARMA}. The uniform in $\eps$ estimates on the Green's function and its gradient can be find in \cite{AL87,AL91_Lp}. The homogeneities in these bounds are the same as those for the Green's function associated to constant coefficient equations, namely the Laplace equations. The striking fact that these bounds still hold for oscillatory equations is due to the fact that the problem \eqref{e.opde} homogenizes. Periodicity or other structural assumptions on the coefficients are necessary.

%%%%%%%%%%%%%
%%%%%%%%%%%%%
\section{Estimates for the Homogenization Error}
\label{s.rerror}

In this section, we estimate the size of the homogenization error $u^\eps - u$. In view of the decomposition \eqref{e.uepsd}, \eqref{e.uepsd.d}, \eqref{e.uepsd.f} and the error estimates in Theorem \ref{t.homogp}, it suffices to focus on the intermediate homogenization error $u^\eps - v^\eps$, with $v^\eps = \cG_\eps f$ defined in \eqref{e.opde}.

%%%%%%%%%%%%%%
%\subsection{Scaling factor in energy norm}

We introduce the function $w^\eps$ which solves
\begin{equation}\label{e.weps}
\cL_\eps  w^\eps = - \nu_\eps v^\eps,
\end{equation}
with homogeneous Dirichlet boundary condition. With the notations $\cG_\eps$ and $\cG^{\eps,\omega}$ introduced earlier, $w^\eps$ is given by $-\cG_\eps \nu_\eps v^\eps$. It follows that
\begin{equation*}
\cL^{\eps,\omega} (u^\eps - v^\eps - w^\eps) = -\nu_\eps w^\eps,
\end{equation*}
and $u^\eps - v^\eps - u$ vanishes at the boundary. Hence we have $u^\eps - v^\eps - w^\eps = -\cG^{\eps,\omega} \nu_\eps w^\eps$. Due to the assumption (A), $\cG^{\eps,\omega}$ is uniformly (in $\eps$ and $\omega$) bounded as a linear operator $L^2 \to L^2$; we have
\begin{equation}\label{e.uvw}
\|u^\eps - v^\eps\|_{L^2} \le C\|w^\eps\|_{L^2}.
\end{equation}
An estimate of $u^\eps - v^\eps$ thus follows from the following result.

%%%%%%%%%%%
\begin{lemma}\label{l.weps} Let $v^\eps = \cG_\eps f$ and $w^\eps$ as above. Under the same conditions of Theorem \ref{t.size}, there exists a universal constant $C$ and
\begin{equation}\label{e.l.weps}
\E \| w^\eps \|^2_{L^2(D)} \le \begin{cases} C \eps^{d\wedge 4} \|f\|_{L^2}^2, \quad&{d \ne 4},\\
C\eps^4|\log \eps| \|f\|_{L^2}^2, \quad&{d=4}.
\end{cases}
\end{equation}
\end{lemma}

\begin{proof} Using the Green's function $G_\eps$, we write
\begin{equation}
w^\eps(x,\omega) = \int_D G_\eps(x,y) \nu\left(\frac{y}{\eps}\right) v^\eps(y) dy.
\end{equation}
The $L^2$ energy of $w^\eps$ is then
\begin{equation*}
\|w^\eps(\cdot,\omega)\|^2_{L^2} = \int_{D^3} G_\eps(x,y) G_\eps(x,z) \nu \left(\frac{y}{\eps}\right) \nu\left(\frac{z}{\eps}\right) v^\eps(y)v^\eps(z) dy dz dx.
\end{equation*}
Take the expectation and use the definition of the auto-correlation function $R$ of $q$, we have
\begin{equation}\label{e.wepsexp}
\E \|w^\eps(\cdot,\omega)\|^2_{L^2} = \int_{D^3} G_\eps(x,y) G_\eps(x,z) R\left(\frac{y - z}{\eps}\right) v^\eps(y)v^\eps(z) dy dz dx.
\end{equation}
%Using the uniform estimates \eqref{e.Gsing} on the Green's function, we find
%\begin{equation*}
%\E \|w^\eps(\cdot,\omega)\|^2_{L^2} \le C\int_{D^3} \frac{|v^\eps(y)v^\eps(z)|}{|x-y|^{d-2}|x-z|^{d-2}} \left| R\left(\frac{y - z}{\eps}\right) \right|  dy dz dx.
%\end{equation*}

Integrate over $x$ first. Apply the uniform estimates \eqref{e.Gsing} and the fact (see e.g. Lemma A.1 of \cite{BJ-CMS-11}): for any $y \ne z$, $0<\alpha, \beta < d$, 
\begin{equation}\label{e.intpot}
\int_D \frac{dx}{|x-y|^{d-\alpha} |x-z|^{d-\beta}} \le
\begin{cases}
C, \quad &\text{if } \alpha + \beta > d,\\
C(1+|\log|y-z||), \quad &\text{if } \alpha + \beta = d,\\
C|y-z|^{\alpha+\beta - d}, \quad &\text{if } \alpha + \beta < d.
\end{cases}
\end{equation}
We get
\begin{equation}\label{e.bddsG}
\int_D |G_\eps(x,y) G_\eps(x,z)| dx \le \begin{cases}
C|y-z|^{-((d-4)\wedge 0)}, \quad&\text{if }\; d \ne 4,\\
C(1+\log|y-z|), \quad&\text{if }\; d = 4.
\end{cases}
\end{equation}
Hence, if $d\ge 2$ and $d \ne 4$,
\begin{equation*}
\E \|w^\eps(\cdot,\omega)\|^2_{L^2} \le C\int_{D^2} \frac{|v^\eps(y)v^\eps(z)|}{|y-z|^{(d-4) \vee 0}} \left|R\left(\frac{y - z}{\eps}\right)\right|  dy dz dx.
\end{equation*}
When $d = 4$, the term $(|y-z|^{(d-4)\vee 0})^{-1}$ should be replaced by $1+ |\log|y-z||$. In any case, the above yields a bound of the form
\begin{equation}\label{e.lem41.1}
\E \|w^\eps(\cdot,\omega)\|^2_{L^2} \le C \int_{\R^d} |\tilde{v}^\eps(y)| \left(K^\eps*\tilde{v}^\eps\right)(y) dy.
\end{equation}
Here, $\tilde{v}^\eps = v^\eps \bone_{D}$ and $\bone_{D}$ denotes the indicator function of the set $D$, $K^\eps(y) = R(\frac{y}{\eps})|y|^{(4-d)\wedge0} \bone_{B_\rho}(y)$ if $d \ne 4$ and $K^\eps(y) = R(\frac{y}{\eps})(1+\bone_{B_\rho}(y) |\log|y||)$ if $d = 4$. Here, $B_\rho$ is the ball centered at zero with radius $\rho$ and $\rho$ is the diameter of $D$. We check that, when $d \ne 4$,
\begin{equation}\label{e.lem41.2}
\|K^\eps(y)\|_{L^1} \le \int_{\R^d} \left|R\left(\frac{y}{\eps}\right)\right|\frac{1}{|y|^{(d-4)\vee 0}} dy = \frac{\eps^d}{\eps^{(d-4) \vee 0}} \int_{\R^d} \frac{\left|R(y)\right|}{|y|^{(d-4) \vee 0}} = C\eps^{d\wedge 4},
\end{equation}
where in the last inequality we used $R \in L^\infty \cap L^1(\R^d)$. Similarly, when $d = 4$,
\begin{equation}\label{e.lem41.3}
\|K^\eps(y)\|_{L^1} = \int_{B_\rho} \left| R\left(\frac{y}{\eps}\right) \right| (1+|\log|y||) dy = \eps^4 \int_{B_{\rho/\eps}} |R(y)|(1+|\log|\eps y||) \le C\eps^4|\log \eps|.
\end{equation}
To get the last inequality, we evaluate the integral on $B_1$ and $B_{\rho/\eps}\setminus B_1$, and bound $|\log|\eps y||$ by $|\log|\eps \rho||$ for the second part. Apply H\"older's and then Young's inequalities to \eqref{e.lem41.1}, we get
\begin{equation*}
\E\|w^\eps\|^2_{L^2} \le C\|K^\eps\|_{L^1} \|v^\eps\|_{L^2}^2 \le C\|K^\eps\|_{L^1} \|f\|^2_{L^2}.
\end{equation*}
Combining this with the estimates in \eqref{e.lem41.2} and \eqref{e.lem41.3}, we complete the proof of the lemma.
\end{proof}

%%%%%%%%%%%%%
\subsection{Scaling of the energy in the random fluctuation}

Now we estimate the $L^2(D)$ norm (the energy) of the random fluctuation $u^\eps - \E u^\eps$ which, in view of \eqref{e.uepsd.f}, is the same as the fluctuation $u^\eps - v^\eps - \E(u^\eps - v^\eps)$. 

Using the first order corrector $w^\eps$ defined by \eqref{e.weps}, and following the approach of \cite{B-CLH-08,BJ-CMS-11}, we represent $u^\eps - v^\eps$ in terms of a truncated Neumann series:
\begin{equation}\label{e.Neumann}
u^\eps - v^\eps = -\cG_\eps \nu_\eps v^\eps + \cG_\eps \nu_\eps \cG_\eps \nu_\eps v^\eps + \cG_\eps \nu_\eps \cG_\eps \nu_\eps (u^\eps - v^\eps).
\end{equation}
In particular, the fluctuations in $u^\eps - v^\eps$ can be written as
\begin{equation*}
u^\eps - \E u^\eps =  - \cG_\eps \nu_\eps v^\eps + \left(\cG_\eps \nu_\eps \cG_\eps \nu_\eps v^\eps - \E \cG_\eps \nu_\eps \cG_\eps \nu_\eps v^\eps \right) + \left(\cG_\eps \nu_\eps \cG_\eps \nu_\eps (u^\eps - v^\eps) - \E \cG_\eps \nu_\eps \cG_\eps \nu_\eps (u^\eps - v^\eps)\right).
\end{equation*}
The first term above is exactly $w^\eps$, which has mean zero and its energy was estimated in Lemma \ref{l.weps}. The next lemma provides an estimate for energy of the second term in the above expansion. 

\begin{lemma} \label{l.energy2} Suppose that the assumptions of Theorem \ref{t.size} are satisfied. Then there exists a constant $C > 0$, depending only on the universal parameters and $f$, such that
\begin{equation}\label{e.energy2}
\E \left\| \cG_\eps \nu^\eps \cG_\eps \nu^\eps v^\eps - \E  \cG_\eps \nu^\eps \cG_\eps \nu^\eps v^\eps \right\|^2_{L^2(D)} \le 
\begin{cases}
C\eps^{2d} \quad&\text{if } \; d=2,3,\\
C\eps^8|\log\eps|^2 \quad&\text{if }\; d =4,\\
C\eps^8 \quad&\text{if }\; 5 \le d \le 7.
\end{cases}
\end{equation}
\end{lemma}

Let $I^\eps_2$ denote the left hand side of \eqref{e.energy2}; it has the expression
\begin{equation*}
\begin{aligned}
I^\eps_2 & = \E \int_D \left( \int_{D^2} G_\eps(x,y)G_\eps(y,z)\left[ \nu^\eps(y)\nu^\eps(z) - \E \nu^\eps(y) \nu^\eps(z)\right] v^\eps(z) dz dy\right)^2 dx\\
& = \int_{D^5} G_\eps(x,y) G_\eps(x,y') G_\eps(y,z) G_\eps(y',z')v^\eps(z) v^\eps(z')\\
&\quad\quad\quad\quad \left[\E \left(\nu\left(\frac{y}{\eps}\right)\nu\left(\frac{y'}{\eps}\right)\nu\left(\frac{z}{\eps}\right)\nu\left(\frac{z'}{\eps}\right)\right)
-
R\left(\frac{y-z}{\eps}\right) R\left(\frac{y'-z'}{\eps}\right) \right] dz' dy' dz dy dx.
\end{aligned}
\end{equation*}
It is then evident that we need to estimate certain fourth order moments of $\nu(x,\omega)$, namely, the function
\begin{equation}\label{e.Psi}
\Psi_\nu(x,y,t,s) := \E\,\nu(x)\nu(y)\nu(t)\nu(s) - (\E\,\nu(x)\nu(y)) (\E\,\nu(t)\nu(s)).
\end{equation}
Were $\nu$ a Gaussian random field, its fourth order moments would decompose as a sum of  products of pairs of $R$. This nice property does not hold for general random fields; however, the following estimate for $\rho$-mixing random fields provides almost the same convenience. 

\begin{lemma}\label{l.cumu} Suppose $\nu(x,\omega)$ is a random field with maximal correlation function $\varrho$ defined as in \eqref{e.varrho}. Then
\begin{equation}\label{e.cumu}
%\left|\E\left[ \nu(x)\nu(y)\nu(t)\nu(s)\right] - R(x-y)R(t-s)\right| \le \vartheta(|x-t|)\vartheta(|y-s|) + \vartheta(|x-s|)\vartheta(|y-t|)
|\Psi_\nu(x,t,y,s)| \le \vartheta(|x-t|)\vartheta(|y-s|) + \vartheta(|x-s|)\vartheta(|y-t|)
\end{equation}
where $\vartheta(r) = (K\varrho(r/3))^{\frac 1 2}$, $K = 4\|\nu\|_{L^\infty(\Omega\times D)}$.
\end{lemma}

We refer to \cite{HPP13} for the proof of this lemma. Estimates of this type based on mixing property already appeared in \cite{B-CLH-08}. We refer to \cite{BJ-CMS-11} for an alternative way to control terms like $\Psi_\nu$, and to section \ref{s.discussion} for some comments on the connection of condition (S) with the lemma above.

\begin{proof}[Proof of Lemma \ref{l.energy2}]
Integrate over $x$ in the expression of $I^\eps_2$, and apply the estimates \eqref{e.Gsing}, \eqref{e.bddsG} and \eqref{e.cumu}. We find, for $d\ge 3$,
\begin{equation*}
\begin{aligned}
I^\eps_2 \le C&\left[\int_{D^4} \frac{(1+\bone_{d=4}|\log|y-y'||)|v^\eps(z) v^\eps(z')|}{|y-y'|^{(d-4)\vee 0} |y-z|^{d-2} |y'-z'|^{d-2}}\ \vartheta\left( \frac{y-y'}{\eps}\right) \vartheta\left( \frac{z-z'}{\eps}\right) dz' dy' dz dy\right.\\
& + \left. \int_{D^4} \frac{(1+\bone_{d=4}|\log|y-y'||)|v^\eps(z) v^\eps(z')|}{|y-y'|^{(d-4)\vee 0} |y-z|^{d-2} |y'-z'|^{d-2}}\ \vartheta\left( \frac{y-z'}{\eps}\right) \vartheta\left( \frac{z-y'}{\eps}\right) dz' dy' dz dy\right].
\end{aligned}
\end{equation*}
For $d=2$, the terms $|y-z|^{-(d-2)}$ and $|y'-z'|^{-(d-2)}$ above should be replaced by $1+|\log|y-z||$ and $1+|\log|y'-z'||$ respectively. Let $I^\eps_{21}$ and $I^\eps_{22}$ denote the two terms on the right hand side of the estimate above. In the following, we set $\rho$ to be the diameter of $D$.

\medskip

{\it Estimate of $I^\eps_{21}$}. We use the change of variables
\begin{equation*}
\frac{y-y'}{\eps} \mapsto y, \quad \frac{z-z'}{\eps} \mapsto z, \quad y' - z' \mapsto y', \quad z' \mapsto z'.
\end{equation*}
Then the integral in $I^\eps_{21}$ becomes, for $d \ge 3$,
\begin{equation*}
\frac{C\eps^{2d}}{\eps^{(d-4)\vee0}} \int_{B^2_{\rho/\eps}} dy dz \int_{B_\rho}dy' \int_{D} dz' \frac{(1+\bone_{d=4}|\log|\eps y||)|v^\eps(z') v^\eps(z'+\eps z)|}{|y|^{(d-4)\vee 0} |y'+\eps(y-z)|^{d-2} |y'|^{d-2}} \vartheta(y) \vartheta(z).
\end{equation*}
We integrate over $y'$ first and apply \eqref{e.intpot}, then integrate over $z'$ and obtain
\begin{equation*}
I^\eps_{21} \le C\|v^\eps\|_{L^2}^2 \eps^{2d-2(d-4)\vee0}\int_{B^2_{\rho/\eps}} \frac{(1+\bone_{d=4}|\log|\eps y||)(1+\bone_{d=4}|\log|\eps(y-z)|)\vartheta(y)\vartheta(z)}{|y|^{(d-4)\vee0}|y-z|^{(d-4)\vee0}} dy dz.
\end{equation*}

When $d=3$, the integral above is bounded because $\vartheta \in L^1(\R^d)$ thanks to assumption (S), and we have $I^\eps_{21} \le C \eps^{2d}$. When $d = 2$, the situation is similar; after the integral over $y'$, there is again no singularity in the denominator. Hence, $I^\eps_{21} \le C\eps^{2d}$.

When $d \ge 5$, by Hardy-Littlewood-Sobolev inequality \cite[Theorem 4.3]{LL-A}, we have for $p,r \in (1,\infty)$,
\begin{equation*}
\int_{\R^{2d}} \frac{(\vartheta(y)/|y|^{d-4})\vartheta(z)}{|y-z|^{d-4}} dy dz \le C\left\|\frac{\vartheta(y)}{|y|^{d-4}}\right\|_{L^p(\R^d)} \|\vartheta\|_{L^r(\R^d)},% \le C(d,\varrho).
\quad
\frac{1}{p} + \frac{d-4}{d} + \frac{1}{r} = 2.
\end{equation*}
Take $p = \frac{d}{4+\delta}$ and $r = \frac{d}{d-\delta}$ for any $(d-8)\vee 0<\delta< d-4$. Then because $\vartheta \in L^\infty \cap L^1(\R^d)$, $|y|^{4-d} \in L^p(B_1)$, the above is finite and we have $I^\eps_{21} \le C\eps^{8}$.

When $d = 4$, we need to control the integral
\begin{equation*}
\int_{B^2_{\rho/\eps}} (1+|\log|\eps y||) (1+|\log|\eps(y-z)||) \vartheta(y)\vartheta(z) dy dz,
\end{equation*}
where $D^* = \{y-y'-z+z'\,|\,y,y',z,z'\in D\}$ is some bounded region formed by certain combinations of points in $D$. As a result, the logarithmic terms are bounded away from the poles. Hence, the above integral is bounded by $O(|\log\eps|^2)$, and $I^\eps_{21} \le C\eps^8|\log\eps|^2$.

\medskip

{\it Estimate of $I^\eps_{22}$}. We apply the change of variables
\begin{equation*}
\frac{y-z'}{\eps} \mapsto y, \quad \frac{y'-z}{\eps} \mapsto y', \quad z - z' \mapsto z, \quad z' \mapsto z'.
\end{equation*}
Then the integral in $I^\eps_{22}$ becomes, for $d=2$,
\begin{equation*}
C\eps^{2d} \int_{B^2_{\rho/\eps}} dy dy' \int_{D}dz' \int_{B_\rho} dz |v^\eps(z') v^\eps(z'+z)| (1+|\log|z-\eps y||)(1+\log|z-\eps y'|) \vartheta(y) \vartheta(y').
\end{equation*}
Integrate over $z'$, $z$ and then over $y'$ and $y$. We find that $I^\eps_{22} \le C\eps^{2d}$. For $d\ge 3$, the same change of variables transforms $I^\eps_{22}$ to
\begin{equation*}
C\eps^{2d} \int_{B^2_{\rho/ \eps}} dy dy' \int_{B_\rho} dz \int_{D}dz'  \frac{(1+\bone_{d=4}|\log|z-\eps(y-y')||)|v^\eps(z') v^\eps(z'+z)|}{|z-\eps(y-y')|^{(d-4)\vee 0} |z-\eps y|^{d-2} |z+ \eps y'|^{d-2}} \vartheta(y) \vartheta(y'),
\end{equation*}
After an integration over $z'$, we only need to control
\begin{equation*}
C\eps^{2d} \int_{B^2_{\rho/\eps}} dy dy' \int_{B_\rho} dz \frac{(1+\bone_{d=4}|\log|z-\eps(y-y')||)}{|z-\eps(y-y')|^{(d-4)\vee 0} |z-\eps y|^{d-2} |z+ \eps y'|^{d-2}} \vartheta(y) \vartheta(y'),
\end{equation*}

When $d=3$, an integration over $z$ removes the singularities in the denominator. Then integrate over $y$ and $y'$ yields that $I^\eps_{22} \le C\eps^{2d}$.

When $d \ge 5$, we need to control the integral, after another change of variables $\eps^{-1}z - (y-y')\mapsto z$ and $-y \mapsto y$,
\begin{equation*}
\int_{\R^{3d}}dydy'dz \frac{C\eps^{8} \vartheta(y) \vartheta(y')}{|z|^{d-4} |z-y|^{d-2} |z-y'|^{d-2}} = \int_{\R^d} dz \frac{C\eps^{8}|K(z)|^2}{|z|^{d-4}},
\end{equation*}
with $K(z) = (|y|^{-(d-2)}*\vartheta)(z)$. Since $\vartheta \in L^1\cap L^\infty(\R^d)$, we have
\begin{equation*}
|K(z)| = \int_{B_1(z)} \frac{\vartheta(y)\ dy}{|z-y|^{d-2}}  + \int_{\R^d\setminus B_1(z)} \frac{\vartheta(y)\ dy}{|z-y|^{d-2}} \le \int_{B_1(z)} \frac{\|\vartheta\|_{L^\infty}\ dy}{|y-z|^{d-2}} + \int_{\R^d\setminus B_1(z)} \vartheta(y) dy \le C.
\end{equation*}
Moreover, by Hardy-Littlewood-Sobolev inequality, we have that
\begin{equation*}
\|K\|_{L^2(\R^d)} = \||y|^{-(d-2)} * \vartheta(y)\|_{L^2(\R^d)} \le C\|\vartheta\|_{L^{2d/(d+4)}(\R^d)} \le C\|\vartheta\|^{\frac{d-4}{2d}}_{L^{\infty}} \|\vartheta\|_{L^1}^{\frac{d+4}{2d}}.
\end{equation*}
Now we showed that $K \in L^\infty\cap L^2(\R^d)$. It follows that the integral to be controlled is finite and we have $I^\eps_{22} \le C\eps^{8}$.

When $d = 4$, after the same change of variable as in the case of $d \ge 5$, we are left to control
\begin{equation*}
\eps^8 \int_{{B^2_{\rho/\eps}}} dy dy' \int_{B_{3\rho/\eps}} dz \frac{(1+\log|\eps z||)\vartheta(y)\vartheta(y') dy dy' dz}{|z-y|^2 |z-y'|^2} = \eps^8 \int_{B_{3\rho/\eps}} (1+|\log |\eps z| |) (K(z))^2 dz,
\end{equation*}
where $K(z) = (\bone_{B_{\rho/\eps}}(y) |y|^{-2}*\vartheta)(z)$. We verify again that $K\in L^\infty\cap L^{2-\delta}(\R^d)$ for any $\delta\in(0,1)$. Estimate the integral again by breaking it into pieces inside and outside $B_1$; we find $I^\eps_{22} \le C\eps^8|\log \eps|$.

\medskip

Combining these estimates above, we proved \eqref{e.energy2}.
\end{proof}

Moving on to the last term in the series \eqref{e.Neumann}, we observe that it cannot be controlled in the same manner as above. Indeed, the term $u^\eps - v^\eps$ is random and depends on $\nu(x,\omega)$ in a nonlinear way. As a result, when we move the expectation into the integral representation, like in step \eqref{e.wepsexp}, we cannot get a simple closed form in terms of $R$.

We hence choose not to address the interaction between $u^\eps- v^\eps$ and the random fluctuation $\nu^\eps$ in the potential directly. Instead, by an application of Minkowski's inequality, we have
\begin{equation*}
\|\E\, \cG_\eps \nu^\eps \cG_\eps \nu^\eps (u^\eps - u)\|_{L^2(D)} \le \E\, \| \cG_\eps \nu^\eps \cG_\eps \nu^\eps (u^\eps - u)\|_{L^2(D)}.
\end{equation*}
Thus, we use the trivial bound on the $L^1(\Omega,L^2(D))$ norm of the fluctuations in $ \cG_\eps \nu^\eps \cG_\eps \nu^\eps (u^\eps - u)$:
\begin{equation*}
r^\eps_2 := \E \| \cG_\eps \nu^\eps \cG_\eps \nu^\eps (u^\eps - u) - \left(\E\, \cG_\eps \nu^\eps \cG_\eps \nu^\eps (u^\eps - u) \right)\|_{L^2(D)} \le 2\,\E\, \| \cG_\eps \nu^\eps \cG_\eps \nu^\eps (u^\eps - u)\|_{L^2(D)},
\end{equation*}
and only control the energy of the last term in \eqref{e.Neumann} itself, as contrast to its variance.  

\begin{lemma} \label{l.energy3} Suppose that the assumptions of Theorem \ref{t.size} are satisfied. Then there exists some constant $C$, depending only on the universal parameters and $f$, such that
\begin{equation}\label{e.energy3}
\E \left\| \cG_\eps \nu^\eps \cG_\eps \nu^\eps (u^\eps - v^\eps) \right\|_{L^2(D)} \le 
\begin{cases}
C\eps^{d} \quad&\text{if } \; d=2,3,\\
C\eps^4|\log\eps|^{\frac 32} \quad&\text{if }\; d =4,\\
C\eps^{6-\frac{d}{2}} \quad&\text{if }\; d\ge 5.
\end{cases}
\end{equation}
\end{lemma}

To prove this result, we estimate the operator norm of $\cG_\eps \nu^\eps \cG_\eps$, which is random since $\nu^\eps$ depends on $\omega$, and combine it with the control of $u^\eps - v^\eps$ which is obtained already earlier.

\begin{lemma}[Mean value of the operator norm $\|\cG_\eps \nu^\eps \cG_\eps\|_{L^2\to L^2}$]
\label{l.mean.opL2} Under the same assumptions of Theorem \ref{t.size}, there exists some universal constant $C$, such that
\begin{equation}\label{e.mean.opL2}
\E \|\cG_\eps \nu^\eps \cG_\eps\|^2_{L^2 \to L^2} \le \begin{cases}
C\eps^d, \quad&\text{if } d = 2,3\\
C\eps^4|\log\eps|^2, \quad&\text{if } d = 4,\\
C\eps^{8-d}, \quad&\text{if } d \ge 5.
\end{cases}
\end{equation}
\end{lemma}

\begin{proof}
For any $h \in L^2(D)$, we have
\begin{equation*}
\|\cG_\eps \nu^\eps \cG_\eps h\|_{L^2}^2 = \int_D \left(\int_{D^2} G_\eps(x,y)\nu^\eps(y) G_\eps(y,z) h(z) dz dy\right)^2 dx.
\end{equation*}
Note that for almost every fixed $x \in D$,
\begin{equation*}
\left|\int_{D^2} G_\eps(x,y)\nu^\eps(y) G_\eps(y,z) h(z) dz dy\right| \le \|h\|_{L^2}  \left\|\int_D  G_\eps(x,y) \nu^\eps(y) G_\eps(y,\cdot)dy \right\|_{L^2}.
\end{equation*}
It then follows that
\begin{equation*}
\|\cG_\eps \nu^\eps \cG_\eps\|^2_{L^2 \to L^2} (\omega) \le \int_{D^2} \left( \int_D G_\eps(x,y) \nu^\eps(y,\omega) G_\eps(y,z)dy \right)^2 dz dx.
\end{equation*}
Taking expectation, and we fined
\begin{equation*}
\E \|\cG_\eps \nu^\eps \cG_\eps\|^2_{L^2 \to L^2} \le \int_{D^4} G_\eps(x,y) G_\eps(x,\eta) R\left(\frac{y-\eta}{\eps}\right) G_\eps(y,z) G_\eps(\eta,z) dy d\eta dz dx.
\end{equation*}
Integrate over $z$ and $x$ variables first. Using \eqref{e.Gsing} and \eqref{e.intpot}, we find that the integrals over $x$ and $z$ variables are estimated as in \eqref{e.bddsG}. Then we have
\begin{equation*}
\E \|\cG_\eps \nu^\eps \cG_\eps\|^2_{L^2 \to L^2} \le C\int_{D^2} \left(\frac{1+\bone_{d=4}|\log|y-\eta||}{|y-\eta|^{(d-4)\vee0}} \right)^2 \left|R\left(\frac{y-\eta}{\eps}\right)\right| dy d\eta.
\end{equation*}
Change variables in the above integral and carry out the analysis as before. We find that \eqref{e.mean.opL2} holds. Note that the estimates become useless for $d \ge 8$.
\end{proof}

\begin{proof}[Proof of Lemma \ref{l.energy3}] For each $\omega \in \Omega$, we have
\begin{equation*}
\|\cG_\eps \nu^\eps \cG_\eps \nu^\eps (u^\eps - v^\eps)\|_{L^2} \le M\|\cG_\eps \nu^\eps \cG_\eps\|_{L^2\to L^2}\|u^\eps - v^\eps\|_{L^2},
\end{equation*}
where $M$ is the uniform bound on the random potential in \eqref{e.qbdd}. Take expectation and then the desired estimate follows from \eqref{e.mean.opL2}, \eqref{e.uvw} and \eqref{e.l.weps}. 
\end{proof}
%%%%%%%%%%%%%%

%%%%%%%%%%%%%%
\subsection{Scaling factor of the random fluctuations in weak topology}
\label{s.wt}

In this section we aim to find the correct scaling factor such that the random fluctuation $u^\eps - \E\,u^\eps$, normalized properly according to this factor, converges with respect to the weak topology. For that purpose, we fix an arbitrary $\varphi \in L^2(D)$ with unit norm, and estimate $\E\,(u^\eps - \E\,u^\eps, \varphi)^2$.

Using the series expansion formula \eqref{e.Neumann}, we have
\begin{equation*}
\begin{aligned}
(u^\eps - \E\,u^\eps, \varphi) =\ -(\cG_\eps \nu_\eps v^\eps, \varphi) & + (\cG_\eps \nu^\eps \cG_\eps \nu^\eps v^\eps - \E\,(\cG_\eps \nu^\eps \cG_\eps \nu^\eps v^\eps), \varphi) \\
& + (\cG_\eps \nu^\eps \cG_\eps \nu^\eps (u^\eps - v^\eps) - \E\,(\cG_\eps \nu^\eps \cG_\eps \nu^\eps (u^\eps - v^\eps)), \varphi).
\end{aligned}
\end{equation*}
Since the operators $\cG_\eps$ and $\cG_\eps \nu^\eps \cG_\eps$ are self-adjoint on $L^2(D)$, we can move it to $\varphi$. Set $\psi^\eps = \cG_\eps \varphi$. The above
expression becomes
\begin{equation}\label{e.fw}
\begin{aligned}
(u^\eps - \E u^\eps, \varphi)  & = - (\nu^\eps v^\eps, \psi^\eps) + \left[ (\nu^\eps \cG_\eps \nu ^\eps v^\eps, \psi^\eps) - \E\,(\nu^\eps \cG_\eps \nu ^\eps v^\eps, \psi^\eps) \right] \\
& \quad\quad\quad\quad\quad\quad\; + \left[ (\nu ^\eps (u^\eps - v^\eps), \cG_\eps \nu^\eps \psi^\eps) - \E\,(\nu ^\eps (u^\eps - v^\eps), \cG_\eps \nu^\eps \psi^\eps) \right]\\
& := I^\eps_1 + [I^\eps_2 - \E\,I^\eps_2] + [I^\eps_3 - \E\,I^\eps_3].
\end{aligned}
\end{equation}
The aim now is to control the variances of $I^\eps_j$, $j = 1,2,3$.

\medskip

{\it Estimate for $I^\eps_1$}. For $I^\eps_1$, which is mean zero, we have
\begin{equation*}
\begin{aligned}
\E\,(I^\eps_1)^2 &= \E\left(\int_{D} \nu^\eps(x) v^\eps(x) \psi^\eps(x) dx\right)^2 = \int_{D^2} R\left(\frac{x-y}{\eps}\right) v^\eps(x)v^\eps(y)\psi^\eps(x)\psi^\eps(y) dx dy\\
&= \int_{\R^d} \left[ R^\eps * (v^\eps(y) \psi^\eps(y) \bone_{D}(y))\right](x) \, v^\eps(x) \psi^\eps(x) \bone_{D}(x) dx\\
&\le C\|R^\eps\|_{L^1(\R^d)} \|v^\eps \psi^\eps \bone_D\|^2_{L^2(\R^d)}.
\end{aligned}
\end{equation*}
Here, $R^\eps(y) = R(\frac{y}{\eps})$ is a shorthand notation. To obtain the last inequality, we applied H\"older's and Young's inequalities. Note that $\|R^\eps\|_{L^1(\R^d)} = \eps^d \|R\|_{L^1(\R^d)}$. Note also that $f, \varphi \in L^2(D)$ implies that $v^\eps, \psi^\eps \in H^2(D)$ which is embedded in $L^4(D)$ for all $2 \le d \le 7$. As a result, we conclude that $\E\,|I^\eps_1| \le C\eps^{\frac d 2}$.

\medskip

{\it Estimate for $I^\eps_2$}. Before calculating the variance of $I^\eps$, we first confirm that $\E(I^\eps_2)^2$ is larger than $\eps^d$ for $d \ge 4$. Hence, the target scaling $\eps^{\frac d 2}$ is correct (with respect to the weak topology) for the fluctuation $u^\eps - v^\eps - \E (u^\eps - v^\eps)$, not the difference itself. By direct computation, for $d\ge 3$,
\begin{equation}\label{e.GR}
\begin{aligned}
\E \left(I^\eps_2\right)^2 &= \int_{D^4} R\left(\frac{x-y}{\eps}\right) R\left(\frac{x'-y'}{\eps}\right) G_\eps(x,y) G_\eps(x',y') v^\eps(y)v^\eps(y')\psi^\eps(x)\psi^\eps(x') dx' dy' dx dy\\
&\lesssim \int_{D^4} \left|R\left(\frac{x-y}{\eps}\right) R\left(\frac{x'-y'}{\eps}\right)\right| \frac{|v^\eps(y) v^\eps(y')\psi^\eps(x)\psi^\eps(x')|}{|x-y|^{d-2}|x'-y'|^{d-2}} dy' dx' dy dx.
\end{aligned}
\end{equation}
For $d=2$, the last integral above should be replaced by
\begin{equation*}
\int_{D^4} \left|R\left(\frac{x-y}{\eps}\right) R\left(\frac{x'-y'}{\eps}\right) v^\eps(y) v^\eps(y')\psi^\eps(x)\psi^\eps(x') [1+|\log|x-y||][1+|\log|x'-y'||] \right|dy' dx' dy dx.
\end{equation*}
After the change of variable
\begin{equation*}
\frac{x-y}{\eps} \mapsto x, \quad \frac{x'-y'}{\eps} \mapsto x', \quad y \to y, \quad y' \to y',
\end{equation*}
the integral to be controlled, for $d \ge 3$, becomes
\begin{equation*}
\eps^{4} \int_{B^2_{\rho/\eps}} \int_{D^2}  |R(x)R(x')|\frac{|v^\eps(y)\psi^\eps(y+\eps x) v^\eps(y') \psi^\eps(y'+\eps x')|}{|x|^{d-2}|x'|^{d-2}} dy' dy dx' dx.
\end{equation*}
Integrating over $y$ and $y'$ and then over $x$ and $x'$, we find that the integral above is finite. Hence, $\E(I^\eps_2)^2$ is of order $\eps^4$ when $d\ge 3$. When $d=2$, the change of variables in the logarithmic functions yield the term $|\log\eps|^2$, and we have $\E(I^\eps_2)^2$ is of order $\eps^4|\log\eps|^2$. We confirm that the scaling of the second term in \eqref{e.Neumann}, when integrated against test functions, can be larger than or comparable to that of the leading term, when $d \ge 4$.

\medskip

We show next that if the mean is removed, the second term in \eqref{e.Neumann} is smaller than the leading one in all dimensions. Using the definition of $\Psi_\nu$ in \eqref{e.Psi} and the estimate in Lemma \ref{l.cumu}, we bound $\E\,(I^\eps_2 - \E I^\eps_2)^2 = \Var\,(I^\eps_2)$, for $d\ge 3$, by
\begin{equation*}
\begin{aligned}
\Var\,(I^\eps_2) & = &&\int_{D^4} \Psi_\nu\left(\frac{x}{\eps},\frac{y}{\eps},\frac{x'}{\eps},\frac{y'}{\eps}\right) G_\eps(x,y) G_\eps(x',y') v^\eps(y)v^\eps(y')\psi^\eps(x)\psi^\eps(x') dx' dy' dx dy\\
& \le &&C\int_{D^4} \vartheta \left(\frac{x-x'}{\eps}\right) \vartheta\left(\frac{y-y'}{\eps}\right) \frac{|v^\eps(y)v^\eps(y')\psi^\eps(x)\psi^\eps(x')|}{|x-y|^{d-2}|x'-y'|^{d-2}} dx' dy' dx dy\\
& &&+C\int_{D^4} \vartheta \left(\frac{x-y'}{\eps}\right) \vartheta\left(\frac{y-x'}{\eps}\right) \frac{|v^\eps(y)v^\eps(y')\psi^\eps(x)\psi^\eps(x')|}{|x-y|^{d-2}|x'-y'|^{d-2}} dx' dy' dx dy.
\end{aligned}
\end{equation*}

The second integral above is essentially the same with the first one if we interchange $x'$ and $y'$. Hence, we focus only on the first one. After the change of variables
\begin{equation*}
\frac{x-x'}{\eps} \mapsto x, \quad \frac{y-y'}{\eps} \mapsto y, \quad y - x' \mapsto x', \quad y' \mapsto y',
\end{equation*}
the first integral becomes
\begin{equation*}
  C\eps^{2d}\int_{\R^{2d}} dx dy \int_{B_\rho}dx' \int_D dy' \vartheta(x)\vartheta(y) \frac{|v^\eps(y'+\eps y)v^\eps(y')\psi^\eps(y'-x'+\eps x + \eps y)\psi^\eps(y'- x'+\eps y)|}{|x'-\eps x|^{d-2}|x'-\eps y|^{d-2}}.
\end{equation*}
Integrate over $y'$ first and use the fact that $\|v^\eps\|_{L^4} \le C$ and $\|\psi^\eps\|_{L^4(D)} \le C$. Then the above integral is bounded by, for $d \ne 4$,
\begin{equation*}
  \int_{\R^{2d}} dx dy \int_{B_\rho} \vartheta(x)\vartheta(y) \frac{C\eps^{2d}\ dx'}{|x'-\eps x|^{d-2}|x'-\eps y|^{d-2}} \le \int_{\R^{3d}} \frac{C\eps^{2d-(d-4)\vee0} \vartheta(x)\vartheta(y)\ dx'dxdy}{|x-y|^{(d-4)\vee0}}.
\end{equation*}
where we integrated over $x'$ and used \eqref{e.intpot} to have the inequality. The resulting integral is clearly finite. Hence we conclude that $\Var\,(I^\eps_2) \le C\eps^{2d}$ for $d= 3$ and that it is of order $\eps^{d+4}$ for $d \ge 5$.

When $d=2$, there is only logarithmic singularity to start with in the expression of $\Var\,(I^\eps_2)$, and we find $\Var\,(I^\eps_2) \le C\eps^{2d}$.

When $d=4$, the integral that remains after we integrate over $x'$ has a term of the form $(\log|\eps(x-y)|) \bone_{\eps(x-y) \in B_{2\rho}}$. It follows then that $\Var\,(I^\eps_2) \le C\eps^8|\log\eps|$.

To summarize, for $d \ge 2$, we have $\E|I^\eps_2 - \E\,I^\eps_2| \ll \E|I^\eps_1|$. That is, when the series expansion is integrated against test functions and the mean is removed, the second terms is much smaller than the leading term.   
%%%%%%%

\medskip

{\it Estimate for $I^\eps_3$}. For the last term, we control it by the crude estimate $\E(I^\eps_3 - \E\,I^\eps_3)^2 \le 2\E(I^\eps_3)^2$. From the expression $I^\eps_3 =  (\nu^\eps(u^\eps-v^\eps),\cG_\eps \nu^\eps \psi^\eps)$, we have
\begin{equation*}
\E\,|I^\eps_3| \le \E \left(\|\nu^\eps\|_{L^\infty} \|u^\eps - v^\eps\|_{L^2} \|\cG_\eps \nu^\eps \psi^\eps\|_{L^2}\right) \le C \left(\E\|u^\eps - v^\eps\|^2_{L^2} \,\E\|\cG_\eps \nu^\eps \psi^\eps\|_{L^2}^2\right)^{\frac 1 2}.
\end{equation*}
Since $\cG_\eps \nu^\eps \psi^\eps$ is exactly of the form of $w^\eps$ defined in \eqref{e.weps}. Owing to \eqref{e.uvw} and Lemma \ref{l.weps}, we conclude that $\E\,|I^\eps_3|$ is of order $\eps^d$ for $d=2,3$, of order $\eps^4|\log\eps|$ for $d=4$, and of order $\eps^4$ for $d \ge 5$. Hence, for all $2 \le d \le 7$, the truncation term in the Neumann series, with respect to the weak topology, has a scaling factor that is smaller than that of the leading term (which is of order $\eps^{\frac d 2}$).

\medskip

\begin{remark} We find that for $2 \le d \le 7$, the random fluctuation $u^\eps - \E\,u^\eps$ scales like $\eps^{\frac{d}{2}}$ when integrated against test functions, and the leading term is the dominating one. We do not expect the dimension constraint $d \le 7$ is not intrinsic. Firstly, it is related to the fact that we stopped at the second order iteration in the Neumann series, and had to control the last term by the crude estimate given by the Minkowski inequality (not taking advantage of removing the mean). Secondly, it is also needed when we claim that $\psi^\eps = \cG_\eps \nu^\eps \varphi$ is in $L^4(D)$. In general, if we assume stronger condition, namely $f \in C(\overline{D})$, then $v^\eps$ is always bounded, and we only need $\psi^\eps \in L^2(D)$ which holds in all dimensions if $\varphi \in L^2(D)$.
\end{remark}

%%%%%%%%%%%%%
%\subsection{Proof of Theorem \ref{t.size}}

We conclude this section by collecting the facts obtained above to give a proof of Theorem \ref{t.size}.

\begin{proof}[Proof of Theorem \ref{t.size}] Let $v^\eps$ be as defined in \eqref{e.opde}. In view of \eqref{e.uepsd.f} and Minkowski inequality, we have 
$$
\E \|u^\eps - \E\,u^\eps\|^2_{L^2} \le \E \left[ 2\|u^\eps - v^\eps\|^2_{L^2} + 2\,\|\E(u^\eps - v^\eps)\|^2_{L^2} \right] \le 4\,\E \|u^\eps - v^\eps\|^2_{L^2}.
$$
Owing to \eqref{e.uvw} and Lemma \ref{l.weps}, we have \eqref{e.t.size.f}.

In view of \eqref{e.uepsd}, \eqref{e.uvw}, Lemma \ref{l.weps} and Theorem \ref{t.homogp}, we have
\begin{equation*}
\E \|u^\eps - u\|_{L^2} \le \E\|u^\eps - v^\eps\|_{L^2} + \|v^\eps - u\|_{L^2} \le C\eps\|f\|_{L^2}.
\end{equation*}
This proves \eqref{e.t.size}.

Finally, to estimate $\E |(u^\eps - \E u^\eps, \varphi)_{L^2}|$ for an arbitrarily fied $\varphi \in L^2(D)$, without loss of generality we can assume $\|\varphi\|_{L^2} = 1$. Then this term is precisely what has been studied immediately above. With $I^\eps_j$, $j=1,2,3$, defined earlier, we have showed that for $2 \le d \le 7$, $\E\,|\sum_{j=1}^3 (I^\eps_j - \E\,I^\eps_j)| \le C\eps^{\frac d 2}$, which is precisely \eqref{e.t.size.fw}.
\end{proof}

%%%%%%%%%%%%%%
%%%%%%%%%%%%%%
\section{Limiting Distribution of the Random Fluctuation}
\label{s.cdist}

In this section, we study the limiting distribution of the scaled random fluctuation, $\eps^{-\frac d 2} \left(u^\eps - \E\,u^\eps\right)$, in functional spaces. As mentioned earlier, the choice of space depends on dimension. When $d = 1$, convergence in law in $C(D)$ of the random fluctuation was proved in \cite{BP-99,B-CLH-08}. We prove Theorem \ref{t.sl.23} below, which establishes convergence in law of the random fluctuation in $L^2(D)$ for $d = 2,3$, and in $H^{-1}(D)$ for $d = 4,5$.

Multiplying $\eps^{-\frac d 2}$ to the series expansion \eqref{e.Neumann}, we obtain the following expression for the scaled random fluctuation:
\begin{equation}\label{e.expansion}
- \frac{\cG_\eps \nu_\eps v^\eps}{\sqrt{\eps^d}} + \frac{\left(\cG_\eps \nu_\eps \cG_\eps \nu_\eps v^\eps - \E \cG_\eps \nu_\eps \cG_\eps \nu_\eps v^\eps \right)}{\sqrt{\eps^d}} + \frac{\left(\cG_\eps \nu_\eps \cG_\eps \nu_\eps (u^\eps - v^\eps) - \E \cG_\eps \nu_\eps \cG_\eps \nu_\eps (u^\eps - v^\eps)\right)}{\sqrt{\eps^d}}.
\end{equation}
Our strategy, as in \cite{B-CLH-08,BJ-CMS-11}, is to prove that the leading term $X^\eps = - \eps^{-\frac d 2} \cG_\eps \nu^\eps v^\eps$ contributes and converges in law to the right distribution depicted by Theorem \ref{t.sl.23}, and show that the other terms converge in stronger mode to the zero function and hence has no contribution to the limiting law. 

At the purely formal level, all these steps are the same as in the setting of non-oscillatory diffusion coefficients. Indeed, we already established controls for the second and last terms above in the previous section. Moreover, the $\eps$ dependence in $\cG_\eps$ and $v^\eps$ is not a problem, as we will see later, for the convergence of the characteristic functions of $X^\eps$, thanks to the fact that $\cG_\eps \varphi \to \cG \varphi$ in $L^2$ for any $\varphi \in H^{-1}(D)$. This dependence, however, does impose difficulty on showing the tightness of the measures of $\{X^\eps\}_{\eps}$. As discussed in section \ref{s.discussion}, the old approach for tightness in \cite{BGGJ12_AA} fails and new ideas are needed.

Our new approach is to use some non-optimal but convenient tightness criteria, described in Theorems \ref{t.tight} and \ref{t.tight2}, for probability measures on $H^k(D)$ that are induced by processes in $H^{k+s}(D)$, $k=-1,0$ and $s > 0$. Since we do need $s$ to be fractional in $(0,1)$, we recall some definitions regarding fractional Sobolev space; see \cite{DPV12} for reference. Given an open set $K \subset \R^d$, the fractional Sobolev space $H^s_0(K)$, for $s \in (0,1)$, is the closure of $C^\infty_0(K)$ in the norm
\begin{equation*}
\|u\|_{H^s(K)}^2 := \|u\|^2_{L^2(K)} + \int_{K^2} \frac{|u(x) - u(y)|^2}{|x-y|^{d+2s}} dx dy.
\end{equation*}
When $K = \R^d$, an equivalent norm for $ u \in H^s(\R^d)$ is
\begin{equation}\label{e.HsRd}
\|u\|_{H^s(\R^d)}^2 := \int_{\R^d} (1+|\xi|^2)^s |\scF u|^2(\xi) d\xi.
\end{equation}
Moreover, for $s \in (0,1)$, $H^{-s}(K)$ is defined to be the dual space $(H^s_0(K))'$, and in particular when $K = \R^d$, the above norm for $H^{-s}(\R^d)$ is still valid.

%%%%%%%%%%
\subsection{Limiting distribution in $L^2(D)$ for dimension two and three}
\label{s.cL2}

For $d = 2,3$, we prove that the leading term $X^\eps$ in \eqref{e.Neumann} converges in law in $L^2(D)$ and show that the other terms vanish in the limit. The next lemma, together with Theorem \ref{t.tight}, yields tightness of $X^\eps$, which is the key step.

%%%%%%%%%%%%%%
% \subsubsection*{Limiting distribution of the leading term}

\begin{lemma}\label{l.Hs} Suppose that the conditions of Theorem \ref{t.size} are satisfied. Assume further that $d=2,3$. Then for any $s \in (0,\frac{1}{2})$, there exists a constant $C$, depending only on the universal parameters and $s$, such that
\begin{equation}\label{e.Hs}
\E \left\| \eps^{-\frac d 2} \cG_\eps \nu^\eps v^\eps \right\|_{H^s}^2 \le C.
\end{equation}
\end{lemma}

\begin{proof} For each fixed $\omega \in \Omega$ and $\eps > 0$, $\eps^{-\frac d 2} \cG_\eps \nu^\eps v^\eps$ belongs to $H^1_0(D)$ and hence also $H^s_0(D)$ for any $s \in (0,1)$. In particular, its $H^s$ semi-norm has the expression
\begin{equation*}
\left[\eps^{-\frac d 2} \cG_\eps \nu^\eps v^\eps \right]_{H^s(D)}^2 = \frac{1}{\eps^d}\int_{D^2}  \frac{|(\cG_\eps \nu^\eps v^\eps)(x) - (\cG_\eps \nu^\eps v^\eps)(y)|^2}{|x-y|^{d+2s}} dy dx.
\end{equation*}
Take expectation and use $L^4$ bounds of $v^\eps$, we have
\begin{equation*}
\E \left[ \eps^{-\frac d 2} \cG_\eps \nu^\eps v^\eps \right]_{H^s(D)}^2 \le \frac{C}{\eps^d} \int_{D^4}\frac{|(G_\eps(x,z)-G_\eps(y,z))(G_\eps(x,\xi)-G_\eps(y,\xi))|}{|x-y|^{d+2s}} \left|R\left(\frac{z-\xi}{\eps}\right)\right| \ d\xi dz dy dx.
\end{equation*}

We claim: there exists $C$ depending only on the universal parameters and $s$, such that for all $\xi, z \in D$,
\begin{equation}\label{e.Hs.key}
\int_{D^2} \frac{|(G_\eps(x,z)-G_\eps(y,z))(G_\eps(x,\xi)-G_\eps(y,\xi))|}{|x-y|^{d+2s}} dy dx \le C,
\end{equation}
where the integral is understood in the principal value sense, as the limit of the integrals over $D\setminus \{B_\delta(z)\cup B_\delta(\xi)\}$, $\delta \to 0$. Since we obtain uniform in $\delta$ bounds for these integrals, we simplify notation and denote the integration domain still by $D^2$. We decompose the integration region $D^2$ into three parts $D^2_j$, $j=1,2,3$ as follows: in $D^2_1$, one of the points in $\{z,\xi\}$, namely $\xi$ without loss of generality, lies outside $B_\rho(x)\cup B_\rho(y)$ where $\rho = |x-y|$; in $D^2_2$, one of the points, namely $z$ without loss of generality, lies in $B_\rho(x)$ and satisfies $|z-x| \le |z-y|$ and at the same time $\xi \in B_\rho(y)$ and $|\xi - y| \le |\xi - x|$; in $D^2_3$, $\xi$ and $z$ cluster around one of the points in $\{x,y\}$; without loss of generality, assume this point is $x$, so  $z, \xi \in B_\rho(x) \cap \{\eta \,|\,|\eta - x| < |\eta -y|\}$. In Figure \ref{f.Hs}, the relative positions between $\{x,y,z,\xi\}$ are illustrated for each case.  

\begin{figure}
\begin{center}
\caption{\label{f.Hs} \it Decomposition criteria of the domain of integration based on the relative position between four points. Left: $(x,y) \in D^2_1$. Middle: $(x,y)\in D^2_2$. Right: $(x,y) \in D^2_3$.}
\bigskip

\includegraphics[width=5cm]{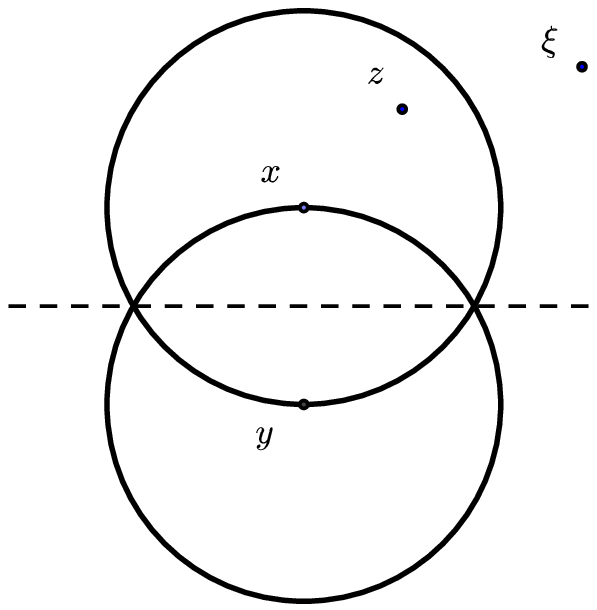}
\includegraphics[width=5cm]{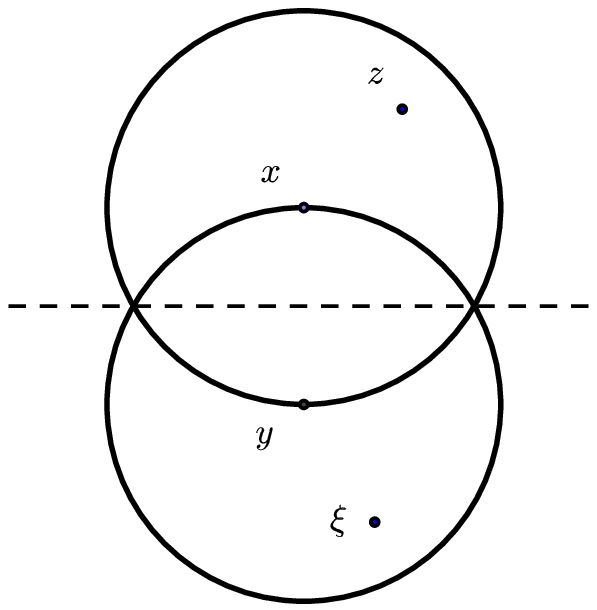}
\includegraphics[width=5cm]{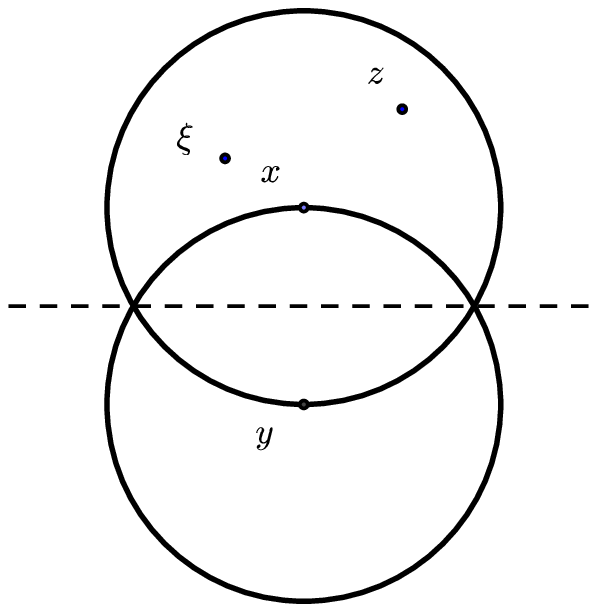}
\end{center}
\end{figure}

Let $I_j$ be the integral over $D^2_j$ of the integrand in \eqref{e.Hs.key}. We estimate $I_j$, $j=1,2,3$, separately and we focus on the case of $d = 3$. It is clear that when $d=2$, the only change is that the Green's function has logarithmic bound, and the analysis below can be adapted. 

On $D^2_1$, without loss of generality, we assume that $|z-x| \le |z-y|$ (if otherwise, we would switch the role of $x$ and $y$). Hence $|G_\eps(x,z) - G_\eps(y,z)| \le C|x-z|^{2-d}$. By mean value theorem,
\begin{equation*}
|G_\eps(x,\xi) - G_\eps(y,\xi)| \le |\nabla G_\eps(\eta,\xi)||x-y|, \quad \text{for some } \eta \text{ between } x \text{ and } y.
\end{equation*}
By the Gradient bound \eqref{e.DGsing} and the fact that $|\eta - \xi| \ge |y-\xi|/2$, we have
$$
|\nabla G_\eps(\eta,\xi)| \le \frac{C}{|\eta-\xi|^{d-1}} \le \frac{C}{|y-\xi|^{d-1}},
\quad\text{hence}\quad
|G_\eps(x,\xi) - G_\eps(y,\xi)| \le \frac{C|x-y|}{|y-\xi|^{d-1}}.
$$
As a result, we have,
\begin{equation*}
I_1 \le \int_{D^2_1} \frac{C}{|x-y|^{d+2s-1}} \frac{1}{|x-z|^{d-2}} \frac{1}{|y-\xi|^{d-1}} \ dx dy.
\end{equation*}
Integrate over $x$ first and then over $y$, using \eqref{e.intpot} in each step; we find as long as $0<s<\frac{1}{2}$, $I_1 \le C$ for some $C$ that only depend on the universal parameters and $s$.

On $D^2_2$, we have $|G_\eps(x,z) - G_\eps(y,z)| \le C|x-z|^{2-d}$ and $|G_\eps(x,\xi) - G_\eps(y,\xi)| \le C|y-\xi|^{2-d}$. At the same time, $|x-z| \le |x-y|$ and $|y-\xi| \le |x-y|$, so we may split the singularity into the integrals over $x$ and $y$ so that each of them is essentially not singular. That is,
\begin{equation*}
I_2 \le \int_{D^2_2} \frac{C}{|x-z|^{\frac{d}{2}+s} |y-\xi|^{\frac{d}{2}+s}} \frac{1}{|x-z|^{d-2}} \frac{1}{|y-\xi|^{d-2}}\ dx dy.
\end{equation*}
We note that the integral above can be separated, and as long as $0 < s < 2 - \frac{d}{2} = \frac{1}{2}$, each integral is finite and hence $I_2 \le C$.

On $D^2_3$, we assume without loss of generality that $z$ and $\xi$ cluster around $x$. Then we have $|G_\eps(x,\eta)-G_\eps(y,\eta)| \le C|x-\eta|^{2-d}$ for $\eta \in \{z,\xi\}$. At the same time, $|x-y| > |y - z|$. As a result, we have
\begin{equation*}
I_3 \le \int_{D^2_3} \frac{C}{|y-z|^{d-\tau} |x-z|^{2s+\tau}} \frac{1}{|x-z|^{d-2}} \frac{1}{|x-\xi|^{d-2}}\ dx dy.
\end{equation*}
We choose $\tau > 0$ so the integral over $y$ is uniformly bounded. The integral over $x$ is also bounded as long as $2s+\tau < (4-d)\wedge2 = 1$, and we have $I_3 \le C$. We note that for any $s \in (0,\frac{1}{2})$ there exists $\tau \in (0,1-2s)$ satisfying the constraint $2s+\tau < 1$.

The above bounds are uniform in $\delta$. Therefore, taking the limit $\delta \to 0$, we proved \eqref{e.Hs.key}. Integrate over $z$ and $\xi$ in the integral expression of $\E\,[\eps^{-\frac{d}{2}}\cG_\eps \nu^\eps v^\eps]_{H^s}^2$; in particular, integrating $R(\cdot/\eps)$ yields a factor of $\eps^d$ that cancels the one in the denominator. We conclude that $\E\,[\eps^{-\frac{d}{2}}\cG_\eps \nu^\eps v^\eps]_{H^s}^2 \le C$ for each fixed $s \in (0,\frac{1}{2})$. Combining this with $\E\,\|\eps^{-\frac{d}{2}}\cG_\eps \nu^\eps v^\eps\|_{L^2}^2 \le C$, which is due to \eqref{e.weps} for $d=2,3$, we proved \eqref{e.Hs}.
\end{proof}

\begin{remark}\label{r.tight} The key step in the proof above is to derive \eqref{e.Hs.key}, which concerns only the Green's function $G_\eps$ and hence is obtained from a purely deterministic argument. Indeed, the scaling factor $\eps^{-\frac d 2}$ plays a role only afterward when we integrate against $R^\eps$, and it disappears in the final estimate because it is the right scaling for integrals of $R^\eps$. In section \ref{s.lrc} where we consider the case of long range correlated random potential $q(x,\omega)$, the scaling in $X^\eps$ will be different, but the tightness of (the measures of) $X^\eps$, with the right scaling, is obtained in the same way as above.
\end{remark}

Next we address the convergence of the characteristic function of the measure of $X^\eps$. In view of Theorem \ref{t.Partha}, this amounts to proving

%%%%%%%%%%%%%%
\begin{lemma}\label{l.ldist} Assume {\upshape(S)}. For any fixed $\varphi \in L^2(D)$, we have
\begin{equation}\label{e.ldist}
-\frac{1}{\sqrt{\eps^d}} (\cG_\eps \nu^\eps v^\eps, \varphi)_{L^2} \xrightarrow{\mathrm{distribution}} \cN \left(0,\sigma^2_\varphi\right), \quad\text{where}\quad %\sigma^2_\varphi: = \sigma^2 \int_D u^2(x) \left(\int_D G(x,y) \varphi(y)dy\right)^2\ dx.
\sigma^2_\varphi: = \sigma^2 \int_D u^2(x) (\cG\varphi)^2(x) \ dx.
\end{equation}
\end{lemma}

\begin{proof}
Moving the operator $\cG_\eps$ to $\varphi$, we have
\begin{equation*}
%\frac{1}{\sqrt{\eps^d}} (w^\eps,\varphi)_{L^2} = 
-\frac{1}{\sqrt{\eps^d}} (\cG_\eps \nu^\eps v^\eps, \varphi)_{L^2} = -\frac{1}{\sqrt{\eps^d}} \int_D \nu\left(\frac{x}{\eps}\right) v^\eps(x) \psi^\eps(x) dx 
\end{equation*}
where $\psi^\eps = \cG_\eps \varphi$. Let $I^\eps_1[\varphi]$ denote the random variable above. Set $\psi = \cG \varphi$ and introduce
\begin{equation}
J^\eps_1[\varphi] := -\frac{1}{\sqrt{\eps^d}} \int_D \nu\left(\frac{x}{\eps}\right) u(x) \psi(x) dx.
\end{equation}
Since $\nu(x,\omega)$ is a stationary ergodic random field that has short range correlation, $u\psi \in L^2(D)$, we apply the well known functional central limit theorem, see e.g. Theorem in Bal \cite{B-CLH-08}, and obtain
\begin{equation}
J^\eps_1[\varphi] \xrightarrow{\rm{distribution}} I_1[\varphi] := \sigma \left(\int_D G(x,y) u(y) dW(y), \varphi\right)_{L^2} \sim \cN \left(0,\sigma^2 \int_D (u(y)\psi(y))^2 dy\right).
\end{equation}
The last relation $\sim$ above means equal in law. We note that
\begin{equation*}
\E |J^\eps_1[\varphi] - I^\eps_1[\varphi]|^2 = \frac{1}{\eps^d} \E \left(\int_D \nu\left(\frac{x}{\eps}\right)(v^\eps \psi^\eps - u\psi) dx\right)^2 \le C\|v^\eps \psi^\eps - u\psi\|_{L^2}^2,
\end{equation*}
and from periodic homogenization theory, we have $v^\eps \to u$ in $L^2$, $\psi^\eps \to \psi$ in $L^2$, as $\eps \to 0$; moreover, $v^\eps$ and $\psi^\eps$ are bounded in $L^\infty$ since $H^2(D)$ is embedded in $L^\infty(D)$ for $d = 2,3$. As a consequence, the right hand side above converges to zero as $\eps \to 0$. As a result,
\begin{equation*}
I^\eps_1[\varphi] = J^\eps_1[\varphi] + (I^\eps_1[\varphi] - J^\eps_1[\varphi])
\end{equation*}
is a sum of a term that converges in distribution to $I_1[\varphi]$ and a term that converges to zero in $L^2(\Omega)$. The desired result follows immediately.
\end{proof}

%%%%%%%%%%%%%%
%\subsubsection*{Proof of Theorem \ref{t.sl.23} (i)}

Finally, we collect the facts obtained above to give a proof of Theorem \ref{t.sl.23} (i).
 
\begin{proof}[Proof of Theorem \ref{t.sl.23} {\upshape(i)}]
Owing to Lemma \ref{l.energy2} and Lemma \ref{l.energy3}, for $d=2,3$, we have
\begin{equation*}
\E\,\|\eps^{-\frac d 2} \left(\cG_\eps \nu_\eps \cG_\eps \nu_\eps v^\eps - \E \cG_\eps \nu_\eps \cG_\eps \nu_\eps v^\eps \right) + \eps^{-\frac d 2}\left(\cG_\eps \nu_\eps \cG_\eps \nu_\eps (u^\eps - v^\eps) - \E \cG_\eps \nu_\eps \cG_\eps \nu_\eps (u^\eps - v^\eps)\right)\|_{L^2} \le \eps^{\frac d2}.
\end{equation*}
By Chebyshev inequality, these two terms, as random elements of $L^2(D)$, converge in probability to the zero function. It follows that the limiting distribution of $\eps^{-\frac d 2}(u^\eps - \E\, u^\eps)$ is given by that of the leading term $X^\eps(\omega) := -\eps^{-\frac d 2} \cG_\eps \nu^\eps v^\eps$.

Let $X$ be the right hand side of \eqref{e.sl.23}. It is a random element of $L^2(D)$ defined on some probability space $(\tilde{\Omega},\tilde{\cF},\tilde{\bP})$ on which the Wiener process $W(y,\tilde{\omega})$ is defined. Let the distribution of $X$ be $P^X$ and its characteristic function be $\phi^{P^X}$. We note that, for any $\varphi \in L^2(D)$, the inner product $(X,\varphi)$ has Gaussian distribution $\cN(0,\sigma^2_\varphi)$, with $\sigma^2_\varphi$ defined in \eqref{e.ldist}. Indeed,
\begin{equation*}
\E^{P^X}(X,\varphi) = \sigma\, \E^{P^X} \int_D \left(\int_{D} G(x,y) \varphi(x)\,dx\right) u(y) \,dW(y) = 0,
\end{equation*}
and
\begin{equation*}
\E^{P^X}(X,\varphi)^2 = \sigma^2\, \E^{P^X} \left(\int_D \left(\int_{D} G(x,y)\varphi(x) dx\right) u(y)\, dW(y)\right)^2 = \sigma^2 \int_D (\cG \varphi)^2 u^2\, dy,
\end{equation*}
This shows $(X,\varphi) \sim \cN(0,\sigma^2_\varphi)$ in law. By Lemma \ref{l.ldist}, for any fixed $\varphi \in L^2(D)$, the random variable $(X^\eps,\varphi)$ converges in distribution to $\cN(0,\sigma^2_\varphi)$. This shows that, as mentioned in Remark \ref{r.Partha}, the characteristic function of the law of $X^\eps$ converges to that of $X$. In view of Lemma \ref{l.Hs} and Theorem \ref{t.tight}, the distribution of $\{X^\eps\}_{\eps \in (0,1)}$ in $L^2(D)$ is tight as well. Consequently, by applying Theorem \ref{t.Partha}, we complete the proof of Theorem \ref{t.sl.23} (i).
\end{proof}

%%%%%%%%%%%%%%
%%%%%%%%%%%%%%
\subsection{Limiting distribution in $H^{-1}(D)$ for dimension four and five}
\label{s.cHs}

For dimension $d \ge 4$, we do not expect $\eps^{-\frac d 2}(u^\eps - \E\,u^\eps)$ to convergence in distribution in $L^2(D)$, because as shown in \eqref{e.t.size.f}, the fluctuations scale like $\eps^2|\log\eps|^{\frac 1 2}$ for $d=4$, and scale like $\eps^2$ for $d\ge 5$. In both cases, the scaling is much stronger than $\eps^{\frac d 2}$. Nevertheless, we prove that convergence in law in $H^{-1}(D)$ holds.

As before, the key step is to show that the probability measure in $H^{-1}(D)$ of the scaled leading term $\{X^\eps\}: = -\eps^{-\frac d 2} \cG_\eps \nu^\eps v^\eps$ in the expansion \eqref{e.Neumann} is tight, and to show that the characteristic function of this measure converges.

Let us first address the characteristic function $\phi^{X^\eps}$. We note that $L^2(D)$ is naturally embedded to $H^{-1}(D)$. For any $f \in L^2$, the linear form $L_f: H^1_0(D) \to \R$ given by $L_f(\psi) = (f,\psi)$ is clearly an element of $H^{-1}(D)$, and 
\begin{equation*}
\|L_f\|_{H^{-1}(D)} = \sup_{\psi \in H^1_0(D), \ \|\psi\|_{H^1} \le 1}\ L_f(\psi) \le \|f\|_{L^2}.
\end{equation*}
We henceforth identify $L_f \in H^{-1}(D)$ with $f$ when $f\in L^2(D)$. For any $\ell \in H^{-1}(D)$, let $l$ be the element in $H^1_0(D)$ that is related to $\ell$ by Riesz representation. Then we have 
\begin{equation*}
(f, \ell)_{H^{-1}(D)} = L_f(l) = (f,l).
\end{equation*}
That is, the $H^{-1}(D)$ inner product of $f \in L^2(D)$ with $\ell$ is the same as the $L^2$ inner product of $f$ with $l$. As a result, Remark \ref{r.Partha} applies for distribution on $H^{-1}(D)$: to show $\phi^{P^{X^\eps}}$ converges to $\phi^{P^X}$ as characteristic function of distributions in $H^{-1}(D)$, it suffices to prove $(X^\eps,h) \to (X,h)$ in distribution as random variables, for each fixed $h \in L^2(D)$.

Now we address the tightness of the measures of $\{X^\eps\}$. Our strategy is to control the mean of $\|X^\eps\|_{H^{-s}(D)}$ for some $s \in (0,1)$ and then apply Theorem \ref{t.tight2}. To this purpose, we first observe that $X^\eps \in L^2(D)$ and hence $X^\eps \in H^{-s}(D)$ if we set
\begin{equation}
X^\eps: H^s_0(D) \to \R
\quad\quad
\text{by}
\quad\quad
X^\eps (h) = \int_D X^\eps(x) h(x) dx.
\end{equation}
For any $h \in H^s_0(D)$, the above clearly defines a continuous linear functional. Moreover, if we identify the function $X^\eps$ with its extension to $\R^d$ by zero outside $D$, the above also defines an element in $H^{-s}(\R^d)$. Since $\partial D$ is regular (as a matter of fact, $C^{0,1}$ boundary is sufficient), any $h \in H^s_0(D)$ can be extended continuously to $Eh \in H^s(\R^d)$ which satisfies $\|Eh\|_{H^s(\R^d)} \le T \|h\|_{H^s(D)}$; see \cite[Theorem 5.4]{DPV12}; by duality $H^{-s}(\R^d)$ is continuously embedded in $H^{-s}(D)$. If fact, we have
\begin{equation}\label{e.XepsH-s}
\begin{aligned}
\|X^\eps\|_{H^{-s}(D)} &:= \sup_{w\in H^s_0(D), \|w\|_{H^s(D)} \le 1} (X^\eps,w)_{L^2} = \sup_{w\in H^s(D), \|w\|_{H^s_0(D)} \le 1} (X^\eps, Ew)_{L^2} \\
&\le  \sup_{v\in H^s(\R^d), \|v\|_{H^s(\R^d)} \le T} (X^\eps,v)_{L^2} \le T \|X^\eps\|_{H^{-s}(\R^d)}.
\end{aligned}
\end{equation}
We note that $\|X^\eps\|_{H^{-s}(\R^d)}$ can be calculated using the formula \eqref{e.HsRd}.

Consider, for each fixed $y \in D$, the Green's function $G_\eps(\cdot,y)$ for the Dirichlet problem \eqref{e.opde}. Extend $G_\eps(\cdot,y)$ to $\R^d$ by zero outside $D$, and let $G^y_\eps$ denote the extended function. Then $G^y_\eps$ defines naturally a linear form on $H^{s}(\R^d)$ by
\begin{equation}
\begin{aligned}
G^y_\eps : \quad H^s(\R^d) \quad &\to \quad &&\R,\\
 \quad  h \quad &\mapsto \quad && G^y_\eps(h) := \int_{\R^d} G^y_\eps(x)h(x)\ dx = \int_D G_\eps(x,y) h(x) dx,
 %\langle G^y_\eps, h\rangle := \int_{\R^d} G^y_\eps(x)h(x)\ dx = \int_D G_\eps(x,y) h(x) dx.
\end{aligned}
\end{equation}
provided the integral is finite. Since $\cG_\eps$ is self-adjoint and by the Green's function representation, $w(y) := G^y_\eps(h)$ is the solution to the Dirichlet problem $\cL_\eps w = h$ on $D$, with zero boundary condition. Note that the restriction of $h$ on $D$ is in $H^s(D)$. Invoking elliptic regularity, we find that $w$ is bounded in $H^{s+2}(D)$. Let $s \in (0,1)$ if $d = 4$ and $s \in (\frac{1}{2},1)$ if $d =5$, then by embedding theorem of fractional Sobolev spaces, $H^{s+2}(D) \subset C^{0,\alpha}(D)$ with $\alpha = s+2 - \frac{d}{2} \in (0,1)$; see \cite[Theorem 1.4.4.1]{Grisvard}. As a result, $|G^y_\eps(h)| \le C\|h\|_{H^s}$ where $C$ only depends on the universal constants and the index $s$. We hence proved the following fact:

\begin{lemma}\label{l.Hsequiv} Assume {\upshape(A)} and $\ol{q} \ge 0$. Identify $G_\eps(\cdot,y)$, for each fixed $y \in D$, with the element in $H^{-s}(\R^d)$ defined above. Suppose $s \in (0,1)$ for $d=4$ and $s \in (\frac{1}{2},1)$ for $d=5$, then there exists $C > 0$, depending only on universal parameters and $s$, such that
\begin{equation}\label{e.GH-s}
\|G_\eps(\cdot,y)\|_{H^{-s}(\R^d)} \le C.
\end{equation}
\end{lemma}

Using this fact and the Fourier transform formula for the $H^{-s}(\R^d)$ norm, we can prove the following control of $\|X^\eps\|_{H^{-s}(\R^d)}$ which, together with Theorem \ref{t.tight2}, yields tightness of $\{X^\eps\}$. 

\begin{lemma}\label{l.H-s} Suppose that the conditions of Theorem \ref{t.size} are satisfied. Assume further that $d = 4,5$. Let $s \in (0,1)$ if $d = 4$ and $s \in (\frac{1}{2},1)$ if $d = 5$. Then there exists a constant $C > 0$, depending only on the universal parameters and on $s$, such that 
\begin{equation}\label{e.H-s}
\E \left\| X^\eps \right\|_{H^{-s}(D)}^2 \le C.
\end{equation}
\end{lemma}

\begin{proof} We identify $X^\eps$ with the element in $H^{-s}(D) \subset H^{-s}(\R^d)$ defined earlier. In view of \eqref{e.XepsH-s}, we have
\begin{equation*}
\|X^\eps\|_{H^{-s}(D)}^2 \le C\|X^\eps\|_{H^{-s}(\R^d)} = \frac{1}{\eps^d} \int_{\R^d} |\scF X^\eps(\xi)|^2 (1+|\xi|^2)^{-s} d\xi,
\end{equation*}
where $\scF X^\eps$ denotes the Fourier transform of the (extended) function $X^\eps$. Using the integral representation of $X^\eps$, we rewrite the above as
\begin{equation*}
\|X^\eps\|_{H^{-s}}^2 = \frac{1}{\eps^d} \int_{\R^d} d\xi (1+|\xi|^2)^{-s} \int_{\R^{2d}}dx dy \int_{D^2} dz dt e^{i\xi\cdot(x-y)} G_\eps(x,z)G_\eps(y,t) \nu^\eps(z) v^\eps(z) \nu^\eps(t) v^\eps(t),
\end{equation*}
where the Green's functions are extended by zero to $\R^d$ for their first variables. Take expectation in this formula; we have
\begin{equation*}
\E\,\|X^\eps\|_{H^{-s}}^2 = \frac{1}{\eps^d} \int_{\R^{2d}} \left(\int_{\R^{3d}}  \frac{e^{i\xi\cdot(x-y)} G_\eps(x,z)G_\eps(y,t)}{(1+|\xi|^2)^{s}} dx dy d\xi \right) R\left(\frac{z-t}{\eps}\right)v^\eps(z) v^\eps(t) dtdz.
\end{equation*}
We claim that for any $z$ and $t$ in $D$, 
\begin{equation}\label{e.tight.H-s}
\left\lvert \int_{\R^{3d}}  \frac{e^{i\xi\cdot(x-y)} G_\eps(x,z)G_\eps(y,t)}{(1+|\xi|^2)^{s}} dx dy d\xi \right\rvert \le C.
\end{equation}
Indeed, we recognize the quantity inside the absolute value sign to be
\begin{equation*}
\int_{\R^d} \frac{ \scF G^z_\eps(\xi) \, \overline{\scF G^t_\eps(\xi)}}{(1+|\xi|^2)^s} d\xi \le \left(\int_{\R^d} |\scF G_\eps^z(\xi)|^2 (1+|\xi|^2)^{-s}\right)^{\frac 1 2} \left(\int_{\R^d} |\scF G^t_\eps(\xi)|^2 (1+|\xi|^2)^{-s}\right)^{\frac 1 2}.
\end{equation*}
The term on the right hand side is precisely $\|G^z_\eps\|_{H^{-s}(\R^d)} \|G^t_\eps\|_{H^{-s}(\R^d)}$. In view of Lemma \ref{l.Hsequiv} we can apply \eqref{e.GH-s} to get an upper bound for the quantity above and prove \eqref{e.GH-s}. Then \eqref{e.tight.H-s} follows, which in turn completes the proof.
\end{proof}

%%%%%%%%%%%%%%
%\subsubsection*{Proof of Theorem \ref{t.sl.23} (ii)}

Finally, we conclude this section by collecting the facts above and prove Theorem \ref{t.sl.23} (ii).

\begin{proof}[Proof of Theorem \ref{t.sl.23} {\upshape(ii)}] \noindent{\it Step 1: Limiting distribution of the leading term.} In view of Theorem \ref{t.tight} and Lemma \ref{l.H-s}, the probability measures on $H^{-1}(D)$ induced by $\{X^\eps\}$ is tight. To check the limit of the characteristic functions of $\{P^{X^\eps}\}$, it suffices to prove \eqref{e.t.Partha3}. This is done in Lemma \ref{l.ldist}. By Theorem \ref{t.Partha}, we conclude that $X^\eps \to X$ in distribution on $H^{-1}(D)$ where $X$ is defined to be the right hand side of \eqref{e.sl.23}.

\medskip

\noindent{\it Step 2: Convergence to zero of the higher order terms.} By Lemma \ref{l.energy2}, and $d = 4,5$, we see that the second term in $u^\eps - \E\,u^\eps$, i.e. $\cG_\eps \nu^\eps \cG_\eps \nu^\eps v^\eps - \E\,\cG_\eps \nu^\eps \cG_\eps \nu^\eps v^\eps$ converges, in $L^2(\Omega,L^2(D))$ and hence in $L^2(\Omega,H^{-1}(D))$, to the zero function. Similarly, the remainder term $\cG_\eps \nu^\eps \cG_\eps \nu^\eps (u^\eps - v^\eps) - \E\,\cG_\eps \nu^\eps \cG_\eps \nu^\eps (u^\eps-v^\eps)$ converges to the zero function in $L^1(\Omega,H^{-1}(D))$. These convergence results are stronger than the mode of convergence in distribution in $H^{-1}(D)$. The proof of Theorem \ref{t.sl.23} (ii) is thus complete.
\end{proof}

%%%%%%%%%%%%%%
%%%%%%%%%%%%%%
\section{The Long Range Correlated Setting}
\label{s.lrc}

In this section, we consider the case when the random potential $q(x,\omega)$ is constructed from a long range correlated Gaussian random field, as described in the assumptions (L) in section \ref{sec:results}. 

\medskip

For the scaling of the homogenization error, we have the following analogue of Theorem \ref{t.size}. We focus on the error $u^\eps - \E u^\eps$ because, as seen earlier, the main contribution to the deterministic error $\E u^\eps - u$ comes from the periodic oscillation in the diffusion coefficients, and Theorem \ref{t.homogp} (i) holds independent of the correlation length of $\nu(x,\omega)$.

\begin{theorem} \label{t.size.l}
Let $D \subset \R^d$ be an open bounded $C^{1,1}$ domain, $u^\eps$ and $u$ be the solutions to \eqref{e.rpde} and \eqref{e.hpde} respectively. Suppose that {\upshape(A)(P)} and {\upshape(L)} hold, $f \in L^2(D)$. Then, there exists positive constant $C$, which depends only on the universal parameters, such that if $2 \le d \le 5$ and $0 <\alpha \le d$ or $6 \le d \le 7$ and $0< \alpha < 6$,
\begin{equation}\label{e.t.size.f.l}
\E\,\|u^\eps - \E u^\eps\|_{L^2} \le \begin{cases}
C \eps^{\frac{\alpha}{2}\wedge 2} \|f\|_{L^2}, \quad&\text{if }\; d \ne 4,\\
C \eps^{\frac \alpha 2} \|f\|_{L^2}, \quad&\text{if }\; d = 4.% \,\alpha < 4.%\\
%C \eps^2 |\log \eps|^{\frac 1 2} \|f\|_{L^2}, \quad&\text{if }\; d=4,\, \alpha = 4.
\end{cases}
\end{equation}
Moreover, for any $\varphi \in L^2(D)$, $2 \le d \le 7$ and $0 < \alpha < d$,
\begin{equation}\label{e.t.size.fw.l}
\E\,\left|(u^\eps - \E u^\eps, \varphi)_{L^2}\right| \le C \eps^{\frac \alpha 2} \|\varphi\|_{L^2} \|f\|_{L^2}.
\end{equation}
\end{theorem}

The proof of \eqref{e.t.size.f} is detailed in section \ref{s.size.l}, and the proof of \eqref{e.t.size.fw} is found in section \ref{s.size.fw}. This result shows that the random fluctuation $u^\eps - \E u^\eps$ caused by the long range correlated random potential scales like $\eps^{(\alpha\wedge4)/2}$ in energy norm, and scales like $\eps^{d/2}$ with respect to the weak topology. Since $\alpha < d$, we note that the random fluctuation in this setting is larger than the case of short range correlated potential. We mention that if $\alpha < 2$, then the random fluctuations dominate the deterministic fluctuation caused by the periodic diffusion. 

The next result exhibits the limiting law of the rescaled random fluctuation $\eps^{-\frac{\alpha}{2}}(u^\eps - \E u^\eps)$. In the presentation, we define formally $W^\alpha(dy)$ as $\dot{W^\alpha} (y) dy$; $\dot{W^\alpha}(y)$ is a centered stationary Gaussian random field with covariance function $\mathbf{E} (\dot{W^\alpha}(x)\dot{W^\alpha}(y)) = \kappa 
|x-y|^{-\alpha}$, where $\mathbf{E}$ denotes the expectation with respect to the distribution of $\dot{W}^\alpha$. Here, $\kappa = \kappa_gV_1^2 > 0$, where $\kappa_g$ and $V_1$ are defined in \eqref{e.tail} and \eqref{e.odd}.

\begin{theorem}
\label{thm:main3}
Suppose that the assumptions in Theorem \ref{t.size.l} hold. Let $\kappa$ be defined as in \eqref{e.odd} and $G(x,y)$ be the Green's function of \eqref{e.hpde}. Let $W^\alpha(dy)$ be defined as above. Then 
\begin{itemize}
\item[\upshape(i)] For $d = 2,3$, as $\eps \to 0$,
\begin{equation}
\frac{u_\eps - \E\{u_\eps\}}{\sqrt{\eps^{\alpha}}} \xrightarrow[\eps \to 
0]{\mathrm{distribution}} \sqrt{\kappa} \int_{D} G(x, y) u (y) W^\alpha(dy), \quad\quad \text{in } L^2(D).
\label{e.main3}
\end{equation}
\item[\upshape(2)] For $d = 4,5$, as $\eps \to 0$, the above holds as convergence in law in $H^{-1}(D)$.
\end{itemize}
\end{theorem}

\begin{remark}\label{rem:main3}
The right hand side of \eqref{e.main3} is an integral with respect to the multiparameter Gaussian random processes $W^\alpha$, we refer to \cite{Khoshnevisan} for the theory. Let $X$ denote the result of the integral. When $d \le 4$, $\alpha < 4$, the Green's function $G(x,\cdot)$ is in $L^{d/(d-\alpha/2)}$ and $X$ is a random element in $L^2(D)$. In general, $X$ is understood through the Fourier transform of its distribution. Given $h^* \in H^{-1}(D)$, $\phi^{P^X}(h^*)$ is defined to be $\mathbf{E} \exp \left( i \sqrt{\kappa} \int_D \langle G(\cdot,y), h^*(\cdot) \rangle u(y) W^\alpha(dy) \right)$. In particular, for any fixed positive integer $N$ and functions $\{\varphi_i; ~ 1 \le i \le 
N\}$ in $L^2(D)$, the random variables $I_i : = \langle 
X, \varphi_i\rangle = \sqrt{\kappa} \int_D \langle G(\cdot,y), \varphi_i(\cdot)\rangle u(y) W^\alpha(dy)$, $i = 1, 
\cdots, N$, is joint Gaussian, centered and have covariance matrix $\Sigma_{ij} := \mathbf{E} (I_i I_j)$ given by
\begin{equation}
\Sigma_{ij} := \kappa \int_{D^2} \frac{ (u \G \varphi_i) 
(y) (u \G \varphi_j)(z)}{|y-z|^\alpha} dy dz.
\label{e.covmat}
\end{equation}
\end{remark}

\smallskip

%%%%%%%%%%
\subsection{Scaling of the energy in the random fluctuation}
\label{s.size.l}

In this section, we prove the first part of Theorem \ref{t.size.l}, which characterizes how does the random fluctuation $u^\eps - \E\,u^\eps$ scale in the energy norm.

The approach is the same as in the proof of Theorem \ref{t.size}: Lemma \ref{lem:keyL2} below provides estimate for the leading term in the series expansion \eqref{e.Neumann}, and hence on $\E \|u^\eps - v^\eps\|^2$. This, together with the estimate in Lemma \ref{l.mean.opL2.l} on the operator norm, controls the remainder term of \eqref{e.Neumann}. Finally, Lemma \ref{l.energy2.l} below controls the energy norm of the second term in \eqref{e.Neumann}.

The first result below is already in \cite{BGGJ12_AA}; we record the details here for the sake of completeness, and for demonstration of the key techniques.

\begin{lemma}\label{lem:keyL2}
Let $\cG_\eps$ be as defined below \eqref{e.Leps}; let $\nu(x,\omega)$ be the random field constructed in {\upshape(L)}. In particular, the autocorrelation function $R(x)$ has the asymptotic behavior like $\kappa_g |x|^{-\alpha}$ with $\alpha \le d$. Then, there exists constant $C$ that depends only on the universal parameters such that, for any $f \in L^2(D)$, we have:
\begin{equation}
\E ~ \|\cG_\eps \nu_\eps f\|_{L^2}^2 \le \begin{cases}
C \eps^{\alpha\wedge4} \|f\|^2_{L^2}, \quad&\text{if }\; d \ne 4,\\
C \eps^{\alpha} \|f\|_{L^2}^2, \quad&\text{if }\; d = 4, \,\alpha < 4,\\
C \eps^4 |\log \eps| \|f\|_{L^2}^2, \quad&\text{if }\; d=4,\, \alpha = 4.
\end{cases}
\label{eq:keyL2}
\end{equation}
\end{lemma}

\begin{proof}
The $L^2$ norm of $\cG_\eps \nu^\eps f$ has the following expression:
\begin{equation*}
\|\cG_\eps \nu^\eps f\|_{L^2}^2 = \int_D \left(\int_D G_\eps(x,y) \nu^\eps(y) f(y) dy\right)^2 
dx.
\end{equation*}
After writing the integrand as a double integral and taking expectation, we 
have
\begin{equation}
\E \|\cG_\eps \nu^\eps f\|_{L^2}^2 = \int_{D^3}  G_\eps(x,y)G_\eps(x,z) R^\eps(y-z)  f(y) f(z) dy 
dz dx.
\label{eq:L2expr}
\end{equation}
Use \eqref{e.Gsing} to bound the Green's functions. Integrate over $x$ and 
apply \eqref{e.intpot}. We get
\begin{equation}
\E \|\cG_\eps \nu^\eps f\|_{L^2}^2 \le C \int_{D^2}  \frac{1 + \bone_{d=4}|\log|y-z||}{|y-z|^{(d-4)\vee 0}} 
|R_\eps(y-z) f(y) f(z)| dy dz.
\label{eq:L2type}
\end{equation}
Change variable $(y, y-z) \mapsto (y,z)$. The above integral is bounded by
\begin{equation*}
\int_D \int_{B_\rho} \frac{1+\bone_{d=4}|\log|z||}{|z|^{(d-4)\vee0}} |R^\eps(z) f(y) f(y-z)| dy dz.
\end{equation*}
Here and in the sequel $B_\rho$ denotes the ball centered at zero with radius $\rho$ which is the diameter of $D$. Integrate over $y$ first, then we have:
\begin{equation}
\E \|\cG_\eps \nu^\eps f\|_{L^2}^2 \le C \|f\|_{L^2}^2 \int_{B_{\rho}} 
\frac{|R_\eps(z)|(1+\bone_{d=4}|\log|z||)}{|z|^{(d-4)\vee0}} dz.
\label{eq:intRpot}
\end{equation}
Decompose the integration region into two parts:
\begin{equation}\label{e.D.dec}
\left\{\begin{aligned}
& D_1: = \{|z\eps^{-1}| \le T\} \cap B_{\rho}, & & \text{on which we have } 
|R_\eps| \le M^2,\\
& D_2: = \{|z\eps^{-1}| > T\} \cap B_{\rho}, & & \text{on which we have } 
|R_\eps| \le C\eps^\alpha |x|^{-\alpha},
\end{aligned}
\right.
\end{equation}
where the bounds on $R^\eps$ above are obtained by (L2) and by Lemma \ref{lem:tail}; the constant $T$ is defined in that lemma. The integration on $D_1$ can be calculated explicitly, and it is bounded by
\begin{equation*}
C \int_{B_{\eps T}} \frac{1+\bone_{d=4}|\log|z||}{|z|^{(d-4)\vee0}} dz = C\int_{0}^{\eps T} \frac{r^{d-1}(1+\bone_{d=4}|\log r|)}{r^{(d-4)\vee 0}} dr.
\end{equation*}
When $d \ne 4$, this integral is of order $\eps^{d\wedge4}$. When $d=4$, this integral is of order $\eps^4|\log\eps|$.

Similarly, the integration over $D_2$ is bounded by
\begin{equation*}
C\eps^\alpha \int_{B_{\rho}\setminus B_{\eps T}} \frac{1+\bone_{d=4}|\log|z||}{|z|^{(d-4)\vee0 + 
\alpha}} dz = C\eps^\alpha \int_{\eps T}^{\rho} \frac{1+\bone_{d=4}|\log r|}{r^{(d-4)\vee 0 - d + 1 +  \alpha}} dr.
\end{equation*}
When $d \ne 4$, this integral is $C\eps^\alpha(\rho^{-(d-4)\wedge 0 + d-\alpha} - (T\eps)^{-(d-4)\wedge 0 + d -\alpha})$, and hence is of order $\eps^{\alpha\wedge4}$. When $d=4$, this integral has leading term $C\eps^\alpha(\rho^{4-\alpha}\log\rho - (T\eps)^{4-\alpha}|\log|\eps T||)$, and hence is of order $\eps^\alpha$ if $\alpha < 4$ and of order $\eps^4|\log \eps|$ if $\alpha = d = 4$.

Combining the above estimates on the integrals over $D_1$ and $D_2$, we obtain \eqref{eq:keyL2}.
\end{proof}

The above proof differs from that of Lemma \ref{l.weps} in the short range correlation setting only in the steps of controlling the integrals of $R^\eps$. In the long range correlation setting, because $R(x)$ is not integrable, we cannot expect to gain a factor of $\eps^d$ by scaling the variable in $R^\eps$. Instead, we gain a factor of $\eps^\alpha$ by using the asymptotic of $R^\eps$ outside an $T\eps$ ball; see Lemma \ref{lem:tail}. This method of controlling the integrals of $R^\eps$ will be used constantly in the sequel.

\medskip

\begin{lemma} \label{l.mean.opL2.l} Under the same assumptions of Theorem \ref{t.size.l}, there exists some constant $C$ depending only on the universal parameters, such that
\begin{equation}\label{e.mean.opL2.l}
\E \|\cG_\eps \nu^\eps \cG_\eps\|^2_{L^2 \to L^2} \le \begin{cases}
C\eps^\alpha, \quad&\text{if } d = 2,3,\\
C\eps^\alpha, \quad&\text{if } d = 4,\, \alpha < 4,\\
C\eps^4|\log\eps|^2, \quad&\text{if } d = 4,\, \alpha = 4,\\
C\eps^{\alpha\wedge(8-d)}, \quad&\text{if } d \ge 5.
\end{cases}
\end{equation}
\end{lemma}

\begin{proof} Following the proof of Lemma \ref{l.mean.opL2}, we have
\begin{equation*}
\E \|\cG_\eps \nu^\eps \cG_\eps\|^2_{L^2 \to L^2} \le C\int_{D^2} \left(\frac{1+\bone_{d=4}|\log|y-\eta||}{|y-\eta|^{(d-4)\vee0}} \right)^2 R\left(\frac{y-\eta}{\eps}\right) dy d\eta.
\end{equation*}
Proceeding as in the proof of Lemma \ref{lem:keyL2}, we obtain \eqref{e.mean.opL2.l}.
\end{proof}

In view of the relation \eqref{e.uvw}, we obtain immediately that, with $u^\eps$ and $v^\eps$ defined in \eqref{e.rpde} and \eqref{e.opde} respectively, $\E\,\|u^\eps - v^\eps\|^2_{L^2}$ satisfies exactly the same estimates in \eqref{eq:keyL2}. Combine this and the above result, we find that $\E\|\cG_\eps \nu^\eps \cG_\eps \nu^\eps (u^\eps - v^\eps)\|_{L^2}$ is of order $\eps^{\alpha}$ for $ d = 2,3,4$ and $\alpha < 4$, of order $\eps^{4}|\log\eps|^{\frac 3 2}$ if $\alpha = d = 4$, of order $\eps^{2+(\alpha\wedge(8-d))/2}$ for $d \ge 5$.

\begin{lemma} \label{l.energy2.l} Under the assumptions of Theorem \ref{t.size.l}, then there exists some constant $C$, depending only on the universal parameters, and
\begin{equation}\label{e.energy2.l}
\E \left\| \cG_\eps \nu^\eps \cG_\eps \nu^\eps v^\eps - \E  \cG_\eps \nu^\eps \cG_\eps \nu^\eps v^\eps \right\|^2_{L^2(D)} \ll \eps^\alpha.%\le 
%\begin{cases}
%C\eps^{2\alpha} \quad&\text{if } \; d=2,3,\\
%C\eps^8|\log\eps|^2 \quad&\text{if }\; d =4,\\
%C\eps^8 \quad&\text{if }\; 5 \le d \le 7.
%\end{cases}
\end{equation}
\end{lemma}

\begin{proof} The expression for the quantity $I^\eps_2 := \E \left\| \cG_\eps \nu^\eps \cG_\eps \nu^\eps v^\eps - \E  \cG_\eps \nu^\eps \cG_\eps \nu^\eps v^\eps \right\|^2_{L^2(D)}$ is found below \eqref{e.energy2}. Integrate over $x$ and apply the estimates \eqref{e.Gsing}, \eqref{e.bddsG} and \eqref{eq:cumu}. We find, for $d\ge 3$, $I^\eps_2$ can be bounded by
\begin{equation*}
\begin{aligned}
\sum_{p \ne \{(1,2),(3,4)\}} \int_{D^4} & \frac{C (1+\bone_{d=4}|\log|y-y'||)|v^\eps(z) v^\eps(z')|}{|y-y'|^{(d-4)\vee 0} |y-z|^{d-2} |y'-z'|^{d-2}}|R_\eps(x_{p(1)}-x_{p(2)}) 
R_\eps(x_{p(3)} - x_{p(4)})| dz dy' dz dy,
\end{aligned}
\end{equation*}
where the ordered set $\{x_1,x_2,x_3,x_4\}$ refers to $\{y,z,y',z'\}$, and $p$ runs over the set $\mathcal{U}$ defined in \eqref{eq:Udef} below, and $p \ne \{(1,2),(3,4)\}$. When $d=2$, the potential terms $|y-z|^{2-d}$ and $|y'-z'|^{2-d}$, which result from an application of the bounds \eqref{e.Gsing} on the Green's functions, should be replaced by $|\log|y-z||$ and $|\log|y'-z'||$. In the summation above, there are 14 terms and each of them contains a product of two $R^\eps$ functions evaluated at the difference vectors between a pair of points in $\{y, z, y',z'\}$; more 
importantly, at most one of the difference vectors is in $\{y-z, y'-z'\}$, in which case that $R^\eps$ shares the same point of evaluation with one of the bounds on the Green's functions.
\begin{figure}[ht]
\center{
\caption{\label{fig:points} Difference vectors between four points.}
\setlength{\unitlength}{0.45 cm}
\begin{picture}(10,10)
\put(1,5){\line(1,0){8}}
\put(4,1){\line(1,4){2}}
\put(1,5){\circle*{0.2}}
\put(9,5){\circle*{0.2}}
\put(4,1){\circle*{0.2}}
\put(6,9){\circle*{0.2}}
\put(0.5,4.5){$y$}
\put(9.5,5.5){$z$}
\put(3,1){$y'$}
\put(6.5,9){$z'$}
\multiput(1,5)(1,-1.333){3}{\line(3,-4){0.8}}
\multiput(6,9)(1,-1.333){3}{\line(3,-4){0.8}}
\end{picture}
\begin{picture}(10,10)
\put(1,5){\line(1,0){8}}
\put(4,1){\line(1,4){2}}
\put(1,5){\circle*{0.2}}
\put(9,5){\circle*{0.2}}
\put(4,1){\circle*{0.2}}
\put(6,9){\circle*{0.2}}
\put(0.5,4.5){$y$}
\put(9.5,5.5){$z$}
\put(3,1){$y'$}
\put(6.5,9){$z'$}
\multiput(1,5)(1,-1.333){3}{\line(3,-4){0.8}}
\multiput(1.2,4.9)(1,0){8}{\line(1,0){0.8}}
\end{picture}
\begin{picture}(10,10)
\put(1,5){\line(1,0){8}}
\put(4,1){\line(1,4){2}}
\put(1,5){\circle*{0.2}}
\put(9,5){\circle*{0.2}}
\put(4,1){\circle*{0.2}}
\put(6,9){\circle*{0.2}}
\put(0.5,4.5){$y$}
\put(9.5,5.5){$z$}
\put(3,1){$y'$}
\put(6.5,9){$z'$}
\multiput(1,5)(1,-1.333){3}{\line(3,-4){0.8}}
\multiput(1,5)(1,0.8){5}{\line(5,4){0.8}}
\end{picture}
}
\end{figure}
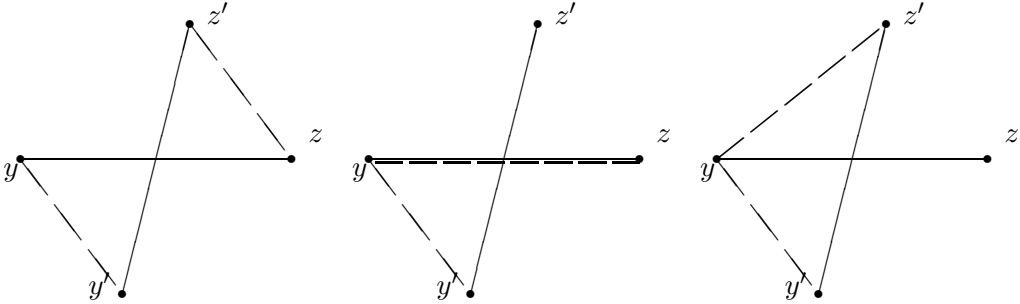

We can divide the fourteen choices of $p$ into three categories as shown in 
Figure \ref{fig:points}. In the first category, the two vectors in the correlation functions are linear 
independent with both of the vectors in the Green's functions; in the 
second category, one of the Green's function shares the same argument with 
one of the correlation function; finally in the third category, the vector 
in one of the Green's function is a linear combination of the two vectors 
of the correlation functions.

\medskip

{\it Case 1}. In the first category, we consider a typical term of the form:
\begin{equation}
I^\eps_{21} = \int_{D^4} \frac{C(1+\bone_{d=4}|\log|y-y'||) |v^\eps(z) v^\eps(z')|}{|y-y'|^{(d-4)\vee0} |y-z|^{d-2} |y'-z'|^{d-2}} 
|R^\eps(y-y') R^\eps(z-z')| \, dz' dy' dz dy.
\label{eq:J_1def}
\end{equation}
Change variable as follows:
\begin{equation*}
y - y' \mapsto y, \quad z- y' \mapsto y', \quad z - z' \mapsto z, \quad  z' \mapsto z'.
\end{equation*}
The integral can be bounded by
\begin{equation*}
\int_{B_{\rho}^3} dy' dz dy \int_D dz' \frac{C(1+\bone_{d=4}|\log|y||)|v^\eps(z')v^\eps(z' + z) R^\eps(y) R^\eps(z)|}{|y|^{(d-4)\vee0} |y'- y|^{d-2} |y' - z|^{d-2}}.
\end{equation*}
The integration over $z'$ yields a bound by $\|v^\eps\|_{L^2}^2$. Then an integration over $y'$ and an application of \eqref{e.intpot} yield
\begin{equation}
I^\eps_{21} \le C\|v^\eps\|_{L^2}^2 \int_{B_\rho^2} \frac{(1+\bone_{d=4}|\log|y||)(1+\bone_{d=4}|\log|y-z||) |R^\eps(y) R^\eps(z)|}{|y|^{(d-4)\vee0} |y-z|^{(d-4)\vee0}}\, dz dy.
\label{eq:J_1est}
\end{equation}
This integral is the same as in the estimate of $I^\eps_{21}$ in the proof of Lemma 4.2. When $d=3$, the above is essentially the square of the integral of $R^\eps$, and is of order $\eps^{2\alpha}$ if $\alpha<d$ and of order $\eps^{2d}|\log\eps|^2$ if $\alpha=d$. When $d=2$, the integration over $y'$ earlier eliminates the mild singularity posed by the logarithmic functions, and the same conclusion as $d=3$ hold.

When $d\ge 5$, applying the Hardy-Littlewood-Sobolev inequality again, we have for some $1 < p,r < \infty$
\begin{equation}
\int_{B_\rho^2} \frac{\left| R^\eps(y)/|y|^{d-4}\right| |R^\eps(z)|}{|y-z|^{d-4}} \,dzdy \le C\left\|\frac{R^\eps}{|y|^{d-4}}\right\|_{L^p(B_\rho)} \|R^\eps\|_{L^r(B_\rho)}, \quad \quad \frac{1}{p} + \frac{d-4}{d} + \frac{1}{r} = 2.
\label{eq:Hardy1}
\end{equation}
If $0 < \alpha < d$, we can take $p = \frac{d}{d+4-\alpha-\tau}$, $r = \frac{d}{\alpha+\tau}$ for some $(4-\alpha)\vee0 < \tau < d-\alpha$. These choices satisfy $\alpha r < d$ and $p<\frac{d}{d-4}$. Hence, according to Lemma \ref{eq:RLp}, $\|R^\eps\|_{L^r}$ are bounded by $C\eps^{\alpha}$ and $\|R^\eps |y|^{-d+4}\|_{L^p}$ is of order $o(1)$. If $5 \le \alpha = d \le 7$, we can choose $p = \frac{d}{d-\tau}$ and $r = \frac{d}{4+\tau}$ with some $0 < \tau <4$. Then we check that $\|R^\eps |y|^{-d+4}\|_{L^p} \|R^\eps\|_{L^r} \le C\eps^8 \ll \eps^\alpha$. 

When $d=4$, we note that the function $K(x) := 1 + \bone_{B_{2\rho}}(x)|\log|x||$ belongs to the weak Lebesgue space $L^q_w(\R^d)$ for all $q \in (1,\infty)$. Apply the refined Young's inequality \cite[Theorem 1.5]{Chemin}, we have
\begin{equation}
\int_{B^2_\rho} |R^\eps(y)K(y) K(y-z)R^\eps(z)| dz dy \le C\|R^\eps K\|_{L^p(B_\rho)}\|R^\eps\|_{L^r(B_\rho)}, \quad \quad \frac{1}{p} + \frac{1}{r} < 2.
\label{eq:Young1}
\end{equation}
When $\alpha < 4$, we can choose $p = r = \frac{4}{\alpha+\tau}$ with $0<\tau<4-\alpha$. Then by  \eqref{eq:RLp.4} and \eqref{eq:RLp}, $\|R^\eps K\|_{L^p}\|R^\eps\|_{L^r}$ is of order $\eps^{2\alpha}$. When $\alpha = d =4$, we can choose $p = r = \frac{4}{4-\tau}$ with $0< \tau < 2$. Then $\|R^\eps K\|_{L^p}\|R^\eps\|_{L^r}$ is of order $\eps^{8-2\tau}|\log \eps| \ll \eps^\alpha$.

\medskip

{\it Case 2}. In the second category, we consider a typical term of the form:
\begin{equation}
I^\eps_{22} = \int_{D^4} \frac{C(1+\bone_{d=4}|\log|y-y'||) |v^\eps(z) v^\eps(z')|}{|y-y'|^{(d-4)\vee0} |y-z|^{d-2} |y'-z'|^{d-2}} 
|R^\eps(y-y') R^\eps(y-z)| \, dz' dy' dz dy.
\label{eq:J_2def}
\end{equation}
This time we use the following change of variables,
\begin{equation*}
y - y' \mapsto y, \quad y' - z' \mapsto y', \quad y - z \mapsto z, \quad  z' \mapsto z'.
\end{equation*}
The integral can then be bounded by
\begin{equation*}
\int_{B_{\rho}^3} dy' dz dy \int_D dz' \frac{C(1+\bone_{d=4}|\log|y||)|v^\eps(z')v^\eps(z' + y + y' - z)| |R^\eps(y) R^\eps(z)|}{|y|^{(d-4)\vee0} |z|^{d-2} |y'|^{d-2}}.
\end{equation*}
Integrate over $z'$ first and then integrate over $y'$. Then we get
\begin{equation}
I^\eps_{22} \le C\|v^\eps\|_{L^2}^2 \int_{B^2_\rho} 
\frac{(1+\bone_{d=4}|\log|y||)|R^\eps(y) R^\eps(z)|}{|y|^{(d-4)\vee0} |z|^{d-2}} dzdy.
\label{eq:J_2est}
\end{equation}

When $d=3$, the integral over $y$ yields a term of order $\eps^\alpha$ if $\alpha < 3$, and of order $\eps^3|\log \eps|$ if $\alpha = 3$; the integral over $z$ is of order $\eps^\alpha$ if $\alpha < 2$, of order $\eps^2|\log\eps|$ if $\alpha = 2$, and of order $\eps^2$ if $2 < \alpha \le 3$. In all cases, $I^\eps_{22} \ll \eps^\alpha$ for $d = 3$. When $d = 2$, the singularity $|z|^{-d+2}$ is replaced by a logarithmic singularity, and the conclusion $I^\eps_{22} \ll \eps^\alpha$ holds. 

When $5 \le d \le 7$, if $\alpha < 4$, then the integral over $y$ is of order $\eps^\alpha$; that over $z$ is of order $\eps^{2\wedge \alpha}$ if $\alpha \ne 2$ and of order $\eps^2|\log \eps|$ if $\alpha = 2$. If $\alpha = 4$, the integral over $y$ is of order $\eps^4|\log\eps|$ and that over $z$ is of order $\eps^2$. If $\alpha > 4$, the integral over $y$ is of order $\eps^4$ and that over $z$ is of order $\eps^2$. We conclude that if $\alpha < 6$, then $I^\eps_{22} \ll \eps^\alpha$.

When $d = 4$, the integral over $y$ is of order $\eps^\alpha$ if $\alpha < 4$, and of order $\eps^4|\log \eps|$ if $\alpha = 4$. The integral over $z$ is of order $\eps^\alpha$ if $\alpha < 2$, of order $\eps^2|\log \eps|$ if $\alpha = 2$, and of order $\eps^2$ if $\alpha = 2$. In any case, $I^\eps_{22} \ll \eps^\alpha$.

\medskip

{\it Case 3}. For the third category, we consider a typical term of the form:
\begin{equation}
I^\eps_{23} = \int_{D^4} \frac{C(1+\bone_{d=4}|\log|y-y'||) |v^\eps(z) v^\eps(z')|}{|y-y'|^{(d-4)\vee0} |y-z|^{d-2} |y'-z'|^{d-2}} 
|R^\eps(y-y') R^\eps(y-z')| \, dz' dy' dz dy.
\label{eq:J_3def}
\end{equation}
Change variables according to
\begin{equation*}
y - y' \mapsto y, \quad y'  \mapsto y', \quad y - z \mapsto z, \quad  y - z' \mapsto z'.
\end{equation*}
The integral can then be bounded by
\begin{equation*}
\int_{B_{\rho}^3} dy' dz' dz \int_D dy' \frac{C(1+\bone_{d=4}|\log|y||)|v^\eps(y-z)v^\eps(y'+y-z')| |R^\eps(y) R^\eps(z')|}{|y|^{(d-4)\vee0} |z|^{d-2} |y-z'|^{d-2}}.
\end{equation*}
Integrate over $y'$ and $z$ and apply the $\|v^\eps\|_{L^\infty}$ bounds; then we are left to control
\begin{equation}
I^\eps_{23} \le C\|v^\eps\|_{L^\infty}^2 \int_{B_\rho^2} 
\frac{(1+\bone_{d=4}|\log|y||) |R^\eps(y) R^\eps(z')|}{|y|^{(d-4)\vee 0} |y-z'|^{d-2}}  dz'dy.
\label{eq:J_3est}
\end{equation}
This term can be estimated in a similar way as what we have done for 
\eqref{eq:J_1est}. The difference is that, in the application of Hardy-Littlewood-Sobolev inequality, the factor $(d-4)/d$ should be replaced by $(d-2)/d$. In particular, for $d \ge 3$, we have, for $p,r \in (1,\infty)$,
\begin{equation*}
I^\eps_{23} \le C\left \| \frac{(1+\bone_{d=4}|\log|y||)R^\eps}{|y|^{(d-4)\vee0}}\right\|_{L^p(B_\rho)} \|R^\eps\|_{L^r(B_\rho)}, \quad \quad \frac{1}{p} + \frac{d-2}{d} + \frac{1}{r} = 2.
\end{equation*}

When $d = 3$, set $p = r = \frac{6}{5}$. Then the above is of order $\eps^{2\alpha}$ if $\alpha < \frac{5}{2}$, of order $\eps^{5}|\log\eps|^{\frac 5 3}$ if $\alpha = \frac{5}{2}$, and of order $\eps^{5}$ if $\frac{5}{2} < \alpha \le 3$. In all cases, $I^\eps_{23} \ll \eps^\alpha$.

When $5 \le d \le 7$ and $0 < \alpha < 6$, we choose $p = \frac{d}{d+2-\alpha-\tau}$ and $r = \frac{d}{\alpha+\tau}$ for some $0 < \tau < 6-\alpha$, and check that $r\alpha < d$ and $p < \frac{d}{d-4}$. This implies that $\|R^\eps\|_{L^r} \le C\eps^\alpha$, and $\|R^\eps|y|^{-d+4}\|_{L^p}$ is of order $o(1)$. When $\alpha \ge 6$, we are only able to choose $p,r$ satisfying the above constraint, such that $\alpha r > d$ and $(d-2+\alpha)p > d$ and $I^\eps_{23} \le C\eps^6$.

When $d = 4$, choose $p = r = \frac{4}{3}$. If $\alpha < 3$, then $\alpha p = \alpha r < 4$ and both $\|(1+|\log|y||)R^\eps(y)\|_{L^p}$ and $\|R^\eps\|_{L^r}$ are of order $\eps^{\alpha}$, and the above quantity is bounded by $C\eps^{2\alpha}$. If $\alpha = 3$, the above is of order $\eps^{2\alpha}|\log\eps|^{\frac 5 2}$. If $3 < \alpha \le 4$, the above is of order $\eps^6|\log\eps|$. In all cases, we have $I^\eps_{23} \ll \eps^{\alpha}$.

When $d = 2$, the singular term in \eqref{eq:J_3est} is $K(y-z')$ with $K(x) := \bone_{B_{2\rho}}(x)|\log|x||$. Note $K \in L^q_w(\R^d)$ for all $q \in (1,\infty)$, and hence $I^\eps_{23} \le \|R^\eps\|_{L^p}\|R^\eps\|_{L^r}$ with any $p, r \in (1,\infty)$ satisfying $p^{-1} + r^{-1} < 2$. Set $p = r > 1$ so that $p<2d/\alpha$, we find that $ \|R^\eps\|_{L^p}\|R^\eps\|_{L^r} \le C\eps^{2d/p} \ll \eps^\alpha$.

\medskip

Combining all of the results above, we finish the proof of the lemma.
\end{proof}

In view of the discussion at the beginning of this subsection, we proved \eqref{e.t.size.f.l}.
%%%%%%%%%%%
\subsection{Scaling factor of the random fluctuations in weak topology}
\label{s.size.fw}

In this section, we prove the second part of Theorem \ref{t.size.l}, which characterizes how does the random fluctuation $u^\eps - \E\,u^\eps$ scale when it is integrated against test functions.

The approach is the same as in section \ref{s.wt}. We fix an arbitrary $\varphi \in L^2(D)$ with unit norm, and start with the expression \eqref{e.fw} for $(u^\eps - \E\,u^\eps,\varphi)$. Let $I^\eps_j$, $j=1,2,3$, be defined as in \eqref{e.fw}. Following the analysis in section \ref{s.wt}, we get $\E\,(I^\eps_1)^2 \le C\|R^\eps\|_{L^1(B_\rho)}$, which is of order $\eps^\alpha$ in view of $\alpha < d$ and \eqref{eq:RLp}. In the estimate of $\Var\,(I^\eps_2)$, \eqref{e.cumu} must be replaced by \eqref{eq:cumu}, so the analysis is more involved. We provide the details and allow $\alpha = d$ in the analysis of $\Var\,(I^\eps_2)$.

\medskip

The expression for the variance of $I^\eps_2$ is, for $d \ge 3$,
\begin{equation*}
\begin{aligned}
& \Var\,(I^\eps_2) = \int_{D^4} \Psi_\nu\left(\frac{x}{\eps},\frac{y}{\eps},\frac{x'}{\eps},\frac{y'}{\eps}\right) G_\eps(x,y) G_\eps(x',y') v^\eps(y)v^\eps(y')\psi^\eps(x)\psi^\eps(x') dx' dy' dx dy\\
\le \ & \sum_{p \ne \{(1,2),(3,4)\}} \int_{D^4} \frac{C|v^\eps(y) v^\eps(z') \psi^\eps(x) \psi^\eps(x')|}{|x-y|^{d-2} |x'-y'|^{d-2}}|R_\eps(x_{p(1)}-x_{p(2)}) 
R_\eps(x_{p(3)} - x_{p(4)})| dy' dx' dy dx,
\end{aligned}
\end{equation*}
where in the summation $p$ runs over $\mathcal{U} \setminus \{(1,2),(3,4)\}$, and the ordered set $\{x_1,x_2,x_3,x_4\}$ refers to $\{x,y,x',y'\}$. For $d = 2$, the above bounds $|x-y|^{-d+2}$ and $|x'-y'|^{-d+2}$ should be replaced by $1+ |\log|x-y||$ and $1 + |\log|x'-y'||$. These integrals in the sum can be estimated in the same way as those of in the proof of Lemma \ref{l.energy2.l}, by considering three categories of $p$. We note that the integrals here are less singular than before because the absence of $|y-y'|^{-(d-4)}$ term.

\medskip

{\it Case 1}. In the first category, we consider a typical term of the form:
\begin{equation}
I^\eps_{21} = \int_{D^4} \frac{|v^\eps(y) v^\eps(y') \psi^\eps(x) \psi^\eps(x')|}{|x-y|^{d-2} |x'-y'|^{d-2}} 
|R^\eps(x-x') R^\eps(y-y')| \, dy' dx' dy dx.
\label{eq:I_1def}
\end{equation}
Change variable as follows:
\begin{equation*}
x - x' \mapsto x, \quad y - x' \mapsto x', \quad y - y' \mapsto y, \quad  y' \mapsto y'.
\end{equation*}
The integral can be bounded by
\begin{equation*}
\int_{B_{\rho}^3} dx' dy dx \int_D dy' \frac{C|v^\eps(y')v^\eps(y' + y) \psi^\eps(y'+y-x')\psi^\eps(y'+y+x-x') R^\eps(x) R^\eps(y)|}{|x'- x|^{d-2} |x' - y|^{d-2}}.
\end{equation*}
The integration over $y'$ yields a bound by $C\|v^\eps\|_{L^4}^2\|\psi^\eps\|_{L^4}^2$. Then an integration over $x'$ and an application of \eqref{e.intpot} yield
\begin{equation}
I^\eps_{21} \le C\int_{B_\rho^2} \frac{(1+\bone_{d=4}|\log|x-y||) |R^\eps(x) R^\eps(y)|}{|x-y|^{(d-4)\vee0}}\, dy dx.
\label{eq:I_1est}
\end{equation}

When $d=3$, the above is essentially the square of the integral of $R^\eps$, which is of order $\eps^{2\alpha}$ if $\alpha<d$ and of order $\eps^{2\alpha}|\log\eps|^2$ if $\alpha = d$. Hence, $I^\eps_{21} \ll \eps^\alpha$. When $d=2$, the integration over $x'$ earlier eliminates the mild singularity of logarithmic functions, and the same conclusion holds.

When $d\ge 5$, applying the Hardy-Littlewood-Sobolev inequality, we have for some $1 < p,r < \infty$
\begin{equation}
\int_{B_\rho^2} \frac{|R^\eps(x) R^\eps(y)|}{|x-y|^{d-4}} \,dydx \le C\left\|R^\eps\right\|_{L^p(B_\rho)} \|R^\eps\|_{L^r(B_\rho)}, \quad \quad \frac{1}{p} + \frac{d-4}{d} + \frac{1}{r} = 2.
\label{eq:Hardy1}
\end{equation}
If $0 < \alpha < d$, we can choose $p = \frac{d}{d+4-\alpha-\tau}$, $r = \frac{d}{\alpha+\tau}$ for some $(4-\alpha)\vee0 < \tau < d-\alpha$. Then $\alpha r < d$ and hence $\|R^\eps\|_{L^r} \le C\eps^\alpha$. On the other hand, $\|R^\eps\|_{L^p}$ is of order $o(1)$. If $\alpha = d$, we choose $p = \frac{d}{4+\tau}$ and $r = \frac{d}{d-\tau}$, with $0<\tau<d-4$. Then $\alpha r > d$ and $\alpha p > d $. It follows that $\|R^\eps\|_{L^p} \|R^\eps\|_{L^r} \le C\eps^{d(p^{-1}+r^{-1})}$ which is of order $\eps^{d+4}$. In all cases, we have $I^\eps_{21} \ll \eps^\alpha$.

When $d=4$, the $|x-y|^{-d+4}$ above is replaced by $K(x-y)$ with $K := 1 + \bone_{B_{2\rho}}(x)|\log|x||$, and $K$ belongs to $L^q_w(\R^d)$ for all $q \in (1,\infty)$. By refined Young's inequality, we have
\begin{equation}
\int_{B^2_\rho} \left| R^\eps(y)K(y-z)R^\eps(z)\right| dz dy \le C\|R^\eps\|_{L^p(B_\rho)}\|R^\eps\|_{L^r(B_\rho)}, \quad \quad \frac{1}{p} + \frac{1}{r} < 2.
\label{eq:Young1}
\end{equation}
If $\alpha < 4$, we can choose $p = r = \frac{4}{\alpha+\tau}$ with $0<\tau<4-\alpha$. Then $r\alpha< d$ and $\|R^\eps\|_{L^p}\|R^\eps\|_{L^r}$ is of order $\eps^{2\alpha}$. When $\alpha =4$, we can choose $p = r = \frac{4}{4-\tau}$ with $0< \tau < 2$. Then $\|R^\eps\|_{L^p}\|R^\eps\|_{L^r}$ is of order $\eps^{8-2\tau}$. In all cases, $I^\eps_{21} \ll \eps^\alpha$.

\medskip

{\it Case 2}. In the second category, we consider a typical term of the form:
\begin{equation}
I^\eps_{22} = \int_{D^4} \frac{C |v^\eps(y) v^\eps(y') \psi^\eps(x) \psi^\eps(x')|}{ |x-y|^{d-2} |x'-y'|^{d-2}} 
|R^\eps(x-x') R^\eps(x-y)| \, dy' dx' dy dx.
\label{eq:I_2def}
\end{equation}
This time we use the following change of variables,
\begin{equation*}
x - y \mapsto x, \quad x - x' \mapsto x', \quad y \mapsto y, \quad  x' - y' \mapsto y'.
\end{equation*}
The integral can then be bounded by
\begin{equation*}
\int_{B_{\rho}^3} dy' dx' dx \int_D dy \frac{C|v^\eps(y)v^\eps(x+y-x'-y') \psi^\eps(y+x) \psi^\eps(x+y-x') R^\eps(x) R^\eps(x')|}{|x|^{d-2} |y'|^{d-2}}.
\end{equation*}
Integrate over $y$ first and then integrate over $y'$. Then we get
\begin{equation}
I^\eps_{22} \le C\|v^\eps\|_{L^\infty}^2 \|\psi^\eps\|_{L^2}^2 \int_{B^2_\rho} 
\frac{|R^\eps(x) R^\eps(x')|}{|x|^{d-2}} dx' dx.
\label{eq:I_2est}
\end{equation}

The integrals over $x'$ and $x$ are separated. When $d \ge 3$ and $\alpha < d$, the integral over $x'$ is of order $\eps^\alpha$, and that over $x$ is always of order $o(1)$. Hence $I^\eps_{22} \ll \eps^\alpha$. When $\alpha = d$, the integral over $x'$ is of order $\eps^d|\log\eps|$, and that over $x$ is of order at most $\eps^2$. In all cases, $I^\eps_{22} \ll \eps^\alpha$.

When $d = 2$, the above quantity amounts to $\|R^\eps|\log|x||\|_{L^1(B_\rho)} \|R^\eps\|_{L^1(B_\rho)}$ which is of order $\eps^{2\alpha}$ if $\alpha < 2$ and of order $\eps^{4}|\log\eps|^3$. In all cases, $I^\eps_{22} \ll \eps^{\alpha}$.

\medskip

{\it Case 3}. For the third category, we consider a typical term of the form:
\begin{equation}
I^\eps_{23} = \int_{D^4} \frac{C|v^\eps(y) v^\eps(y') \psi^\eps(x) \psi^\eps(x')|}{|x - y|^{d-2} |x'-y'|^{d-2}} 
|R^\eps(x-x') R^\eps(x-y')| \, dy' dx' dy dx.
\label{eq:I_3def}
\end{equation}
Change variables according to
\begin{equation*}
x \mapsto x, \quad x - x'  \mapsto x', \quad x - y  \mapsto y, \quad  x - y' \mapsto y'.
\end{equation*}
The integral can then be bounded by
\begin{equation*}
\int_{B_{\rho}^3} dy' dx' dy \int_D dx \frac{C|v^\eps(x-y)v^\eps(x-y') \psi^\eps(x) \psi^\eps(x-x') R^\eps(x') R^\eps(y')|}{|y|^{d-2} |x'-y'|^{d-2}}.
\end{equation*}
Integrate over $x$ first and then over $y$; we obtain
\begin{equation}
I^\eps_{23} \le C\|v^\eps\|_{L^\infty}^2 \|\psi^\eps\|_{L^2}^2 \int_{B_\rho^2} 
\frac{|R^\eps(x') R^\eps(y')|}{|x'-y'|^{d-2}}  dx' dy'.
\label{eq:I_3est}
\end{equation}
Applying Hardy-Littlewood-Sobolev inequality again to the integral above, we have, for $d \ge 3$ and for $p,r \in (1,\infty)$,
\begin{equation*}
I^\eps_{23} \le C\left \|R^\eps \right\|_{L^p(B_\rho)} \|R^\eps\|_{L^r(B_\rho)}, \quad\quad \frac{1}{p} + \frac{d-2}{d} + \frac{1}{r} = 2.
\end{equation*}

When $d \ge 3$ and $0 < \alpha < d$, set $p = \frac{d}{d+2-\alpha-\tau}$ and $r = \frac{d}{\alpha+\tau}$ with $(2-\alpha)\vee 0 < \tau < d-\alpha$. Then $\alpha r < d$ and $p > 1$. Hence $\|R^\eps\|_{L^r}$ is of order $\eps^\alpha$ and $\|R^\eps\|_{L^p}$ is of order $o(1)$. If $\alpha = d$, then set $r = \frac{d}{d-\tau}$ and $p = \frac{d}{2+\tau}$ with $0< \tau < d-2$. Then $r\alpha > d$ and $p \alpha > d$, and hence $\|R^\eps\|_{L^p}\|R^\eps\|_{L^r} \le C\eps^{d(p^{-1}+r^{-1})} = C\eps^{d+2}$. In all case, $I^\eps_{23} \ll \eps^\alpha$.

When $d = 2$, the singular term in \eqref{eq:I_3est} is $K(x'-y')$ with $K(x) = \bone_{B_{2\rho}}(x)|\log|x||$. Note that $K \in L^q_w(\R^d)$ for all $q \in (1,\infty)$, and hence $I^\eps_{23} \le \|R^\eps\|_{L^p}\|R^\eps\|_{L^r}$ for any $p, r \in (1,\infty)$ satisfying $p^{-1} + r^{-1} < 2$. Set $p = r > 1$ so that $p<2d/\alpha$, we find that $ \|R^\eps\|_{L^p}\|R^\eps\|_{L^r} \le C\eps^{2d/p} \ll \eps^\alpha$.

\medskip

Combining all of the results above, we proved that there exists some constant $C$ that depends only on the universal parameters and the data $\varphi$, such that for any $d \ge 2$, $0 < \alpha \le d$, 
\begin{equation*}
\Var\, (I^\eps_2) \ll C \eps^\alpha.
\end{equation*}
Finally, for the truncation term $I^\eps_3$, we have $\E\,|I^\eps_3| \le C\left(\E\|u^\eps - v^\eps\|^2_{L^2} \,\E\|\cG_\eps \nu^\eps \psi^\eps\|_{L^2}^2\right)^{\frac 1 2}$ which, in view of \eqref{eq:keyL2}, is of order $\eps^{\alpha\wedge 4}$ or $\eps^4|\log \eps|$, much smaller than $\eps^{\frac \alpha 2}$ for all $d \le 7$.

We hence conclude with \eqref{e.t.size.fw.l}. In particular, for $2 \le d \le 7$, the leading term in the series expansion \eqref{e.Neumann} is dominating in the weak topology, and it scales like $\eps^{\frac \alpha 2}$.

%%%%%%%%%%%%%%
\subsection{Limiting distribution of $u^\eps - \E\,u^\eps$}

We multiply the series expansion \eqref{e.Neumann} with the proper scaling factor of the long range correlation setting, and get the following expression for the scaled random fluctuation $\eps^{-\frac \alpha 2} \left(u^\eps - \E\,u^\eps\right)$:
\begin{equation}\label{e.expansion.l}
- \frac{\cG_\eps \nu_\eps v^\eps}{\sqrt{\eps^\alpha}} + \frac{\left(\cG_\eps \nu_\eps \cG_\eps \nu_\eps v^\eps - \E \cG_\eps \nu_\eps \cG_\eps \nu_\eps v^\eps \right)}{\sqrt{\eps^\alpha}} + \frac{\left(\cG_\eps \nu_\eps \cG_\eps \nu_\eps (u^\eps - v^\eps) - \E \cG_\eps \nu_\eps \cG_\eps \nu_\eps (u^\eps - v^\eps)\right)}{\sqrt{\eps^\alpha}}.
\end{equation}
Following the approach of section \ref{s.cdist}, when $2 \le d \le 5$ and $0 < \alpha < d$, the first term above dominates and this term alone contributes to the limiting distribution in $L^2(D)$ for $d = 2,3$, and in $H^{-1}(D)$ for $d = 4,5$. We hence set $X^\eps = - \eps^{-\frac \alpha 2} \cG_\eps \nu^\eps v^\eps$.

\subsubsection*{Tightness of the probability measures of $\{X^\eps\}$}

In view of Remark \ref{r.tight}, the tightness of the probability measures induced by $\{X^\eps\}$ on $L^2(D)$, when $d = 2,3$, follows from Lemma \ref{l.Hs}, with $\eps^{-\frac d 2}$ replaced by $\eps^{-\frac \alpha 2}$. Similarly, the tightness of the probability measures induced by $\{X^\eps\}$ on $H^{-1}(D)$, when $d = 4,5$, follows from Lemma \ref{l.H-s}.

\subsubsection*{Limits of the characteristic functions of $\{X^\eps\}$}

We have shown in section \ref{s.cdist} that for the limit of the characteristic functions of the measures of $\{X^\eps\}$, whether in $L^2(D)$ for $d = 2,3$ or in $H^{-1}(D)$ for $d = 4,5$, it suffices to characterize the limiting distribution, for an arbitrarily fixed $\varphi \in L^2(D)$, of $I^\eps_1[\varphi] := (X^\eps,\varphi) = -\eps^{-\frac \alpha 2} (\nu^\eps v^\eps, \psi^\eps)$, where $\psi^\eps = \cG_\eps \varphi$. Following the proof of Lemma \ref{l.ldist}, we introduce $J^\eps_1[\varphi] = -\eps^{-\frac \alpha 2} (\nu^\eps u, \psi)$ with $\psi = \cG \varphi$. For the random field $\nu(x,\omega) = \Phi(g(x,\omega))$ constructed in (L), applying Lemma 4.3 in \cite{BGGJ12_AA}, we have
\begin{equation}
J^\eps_1[\varphi] = -\frac{1}{\eps^{\frac{\alpha}{2}}} \int_{D}{\nu\left( \frac{y}{\eps} \right)u(y) \psi(y) dy}\xrightarrow[\eps \to 
0]{\mathrm{distribution}}
-\sqrt{\kappa} \int_{D} u(y) \psi(y) \ W^{\alpha}(d y),
\label{eq:wcon}
\end{equation}
as convergence in distribution of real random variables. Since $I^\eps_1[\varphi] - J^\eps_1[\varphi]$ converges in $L^2(\Omega)$ to zero, $I^\eps_1[\varphi]$ has the same limiting distribution. This establishes the convergence of characteristic functions of the measures of $X^\eps$.

\subsubsection*{Proof of Theorem \ref{thm:main3}}
 
The arguments above proved that $X^\eps$ converges in law to $X$, in $L^2(D)$ for $d = 2,3$ and in $H^{-1}(D)$ for $d = 4,5$, where $X$ is defined as the right hand side of \eqref{e.main3}. In view of \eqref{e.energy2.l}, the second term in \eqref{e.expansion.l} converges in $L^2(\Omega,L^2(D))$ to the zero function. Similarly, due to \eqref{e.uvw}, \eqref{eq:keyL2} and \eqref{e.mean.opL2.l}, the last term in \eqref{e.expansion.l} converges in $L^1(\Omega,L^2(D))$ to the zero function; see the remark below \eqref{e.mean.opL2.l}. The conclusions of Theorem \ref{thm:main3} follows immediately.

%%%%%%%%%%%%%%
%%%%%%%%%%%%%%
\section{Further Discussions}
\label{s.discussion}

In this section, we make some comments on the assumptions of the random potentials used in this paper, and discuss the case in higher dimensions $d \ge 6$.

\medskip

{\bf An alternative condition for (S).} In the short range correlation setting for $\nu(x,\omega)$, we assumed the condition (S). Upon applying Lemma \ref{l.cumu}, we can bound the (partial) fourth order moment $\Psi_\nu$ by the sum of two terms, each consisting product of a pair of functions $\vartheta\in L^\infty\cap L^1(\R^d)$. However, as remarked earlier, (S) essentially requires the mixing coefficient $\varrho(r)$, and hence $R(|x|)$, to behave like $o(r^{-2d})$ at infinity, much stronger than $R(x)$ being integrable.

We remark that (S) is assumed mainly to simplify the presentation and the $o(r^{-2d})$ decay of $\varrho$ is not necessary. In fact, an alternative assumption used in \cite{BJ-CMS-11} to control fourth order moments is: there exists $\vartheta: \R^d \to \R_+$ in $L^1\cap L^\infty(\R^d)$, such that \eqref{eq:cumu} hold. This is clearly a much more general assumption, and it is satisfied if $\nu(x,\omega) = \Phi(g(x,\omega))$ with $\Phi$ satisfying (L2) and (L3), and $g(x,\omega)$ being a centered stationary Gaussian random field with correlation function $R_g = o(|x|^{-d})$ as $|x| \to \infty$.

The conclusions of Theorem \ref{t.sl.23} still hold if (S) is replaced by the above alternative assumption. Indeed, we only need to modify the control of $\Psi_\nu$ in the proof of Lemma \ref{l.energy2} and in section \ref{s.wt}. For instance, in the first inequality in the proof of Lemma \ref{l.energy2}, we have more but finitely many integrals instead of two on the right hand side. Nevertheless, in all of these integrals, at most one of the function $\vartheta$ has the same variable with one of the Green's function, and all of them can be controlled in the same spirit of the proof of Lemma \ref{l.energy2.l}.

\medskip

{\bf Comparison with the case of non-oscillatory diffusion.} The main results of this paper shows that the framework developed in \cite{B-CLH-08,BJ-CMS-11,BGGJ12_AA}, in the setting of non-oscillatory differential operator with oscillatory random potential, applies even when the differential operator is also oscillatory, as long as we have uniform in $\eps$ control of the Green's functions and their gradients, i.e. \eqref{e.Gsing} and \eqref{e.DGsing}, and provided that there is no random correlation between the diffusion coefficients and the potential. At the formal level, there is no difference in the proof, and the usual strategy using (truncated) series expansion applies. However, the role played by the oscillatory diffusion coefficients becomes prominent in getting the tightness of the measures of $\{X^\eps\}$.

Let us recall the previous method used for tightness in the setting of non-oscillatory differential operator. Set $\cL  : = - \sum_{i,j=1}^d \ol{a}_{ij} \frac{\partial^2}{\partial x_i \partial x_j} + \ol{q}$, and consider the Dirichlet problem $(\cL + \nu^\eps) u^\eps = f$ in $D$ with zero boundary condition. Then $u^\eps$ homogenizes to $u$, the solution of \eqref{e.hpde}. As in \cite{BGGJ12_AA}, the limiting distribution of $\eps^{-\frac d 2} (u^\eps - \E u^\eps)$, say, in the short range correlation setting, is characterize by that of $X^\eps = -\eps^{-\frac d 2} \cG \nu^\eps u$. To prove tightness of the measures $\{X^\eps\}$ in $L^2(D)$, the strategy of \cite{BGGJ12_AA} is to use the spectral representation of $L^2(D)$. Note that $\cL$ is formally self-adjoint and its inverse, i.e. $\cG$, is compact on $L^2(D)$. Hence, $\cL$ admits real eigenvalues $\{\lambda_k\}_{k=1}^\infty$ such that,
$$
0 \le \ol{q} < \lambda_1 \le \lambda_2 \le \cdots, \quad\quad \lambda_k \to \infty \quad\text{as}\quad k \to \infty,
$$
and eigenfunctions $\{\phi_k\}_{k=1}^\infty$, $\|\phi_k\|_{L^2} = 1$, such that
\begin{equation*}
\left\{
\begin{aligned}
\cL \phi_k &= \lambda_k \phi_k, \quad & &\text{in } D,\\
\phi_k &= 0, \quad & &\text{on } \partial D.
\end{aligned}
\right.
\end{equation*}
Moreover, $\{\phi_k\}$ form an orthonormal basis of $L^2(D)$ and we have the following representation of the space $\cH^0(D) = L^2(D)$ and the Sobolev space $\cH^1(D) = H^1_0(D)$; see \cite[Section 6.5]{Evans}: for $s = 0,1$,
\begin{equation}\label{e.cHsdef}
\cH^{s}(D) = \ol{\left\{ f \in C^\infty(D) \;:\; \sum_{k=1}^{\infty} \left(f, \phi_k \right)_{L^2}^2 \lambda^{s}_k < \infty \right\}}, \quad\text{and}\quad \|v\|^2_{\cH^{s}} := \sum_{k=1}^\infty  \left( f, \phi_k\right)^2_{L^2} \lambda_k^s.
\end{equation}
A natural criterion for tightness of (the measures of) $\{X^\eps\}$ is that their measures do not scatter to higher and higher modes. More precisely, let $P_N$ denote the projection operator in $L^2(D)$ to the space $W_N: = \mathrm{span}\,\{\phi_1, \cdots, \phi_N\}$ spanned by the first $N$ modes. Then $\{X^\eps\}$ is tight if $\E\|X^\eps\|_{L^2} \le C$ and
\begin{equation}\label{e.t.L2.proj}
\lim_{N \to \infty}\, \sup_{\eps \in (0,1)} \, \E \|X^\eps - P_N X^\eps\|_{L^2} = 0.
\end{equation}
Using the representation formula in \eqref{e.cHsdef}, and the fact that $\cG \phi_k = (\lambda_k)^{-1} \phi_k$, we have
\begin{equation*}
 \E \|X^\eps - P_N X^\eps\|_{L^2}^2 = \frac{1}{\eps^d} \sum_{k=N+1}^\infty \E (\cG \nu^\eps u, \phi_k)^2 = \frac{1}{\eps^d} \sum_{k=N+1}^\infty \frac{1}{\lambda_k^2} \E (\nu^\eps u, \phi_k)^2.
\end{equation*}
As in section \ref{s.wt}, we have $\sup_{\eps \in (0,1)} \sup_k \E (\nu^\eps u, \phi^k)^2 \le C$. In view of the Weyl's asymptotic formula for the eigenvalues, $\lambda_k \approx k^{\frac 2 d}$ for $k$ large, we conclude that
\begin{equation*}
 \sup_{\eps \in (0,1)} \, \E \|X^\eps - P_N X^\eps\|_{L^2}^2 \lesssim \sum_{k = N+1}^\infty \frac{1}{\lambda^2_k} \lesssim  \sum_{k = N+1}^\infty \frac{1}{k^{\frac 4 d}} .
\end{equation*}
Hence, for $d = 2,3$, we obtain tightness of $\{X^\eps\}$ for free, as byproduct of the analysis in \ref{s.wt}.

In the setting of this paper, $\cL$ above is replaced by $\cL_\eps$ defined in \eqref{e.Leps}. The above approach for tightness fails completely. On the one hand, if we replace the eigen pairs $(\lambda_k, \phi_k)_k$ by $(\lambda^\eps_k, \phi^\eps_k)_k$ where the latter solve the eigenvalue problems associated to $\cL_\eps$, then instead of \eqref{e.t.L2.proj}, we obtain
\begin{equation*}
\lim_{N \to \infty}\, \sup_{\eps \in (0,1)} \, \E \|X^\eps - P^\eps_N X^\eps\|_{L^2} = 0,
\end{equation*}
where $P^\eps_N$ is the projection to $W^\eps_N: = \mathrm{span}\,\{\phi^\eps_1,\cdots,\phi^\eps_N\}$. This is useless because, a priori, the basis $(\phi^\eps_k)_k$ changes with $\eps$, and it is not clear that the union (over $\eps \in (0,1)$) of unit balls in $W^\eps_N$ is still compact for all $N$. On the other hand, if we fix a spectral representation, say, using $(\lambda_k, \phi_k)_k$ defined before. Then we no longer have the relation $\cG_\eps \phi_k = (\lambda_k)^{-1} \phi_k$. It is not difficult to check that $\|\nabla \cG_\eps \phi_k\|_{L^2} \approx \frac{1}{\sqrt{\lambda_k}}$ and this estimate is sharp. An application of Poincar\'e inequality yields that $\|\cG_\eps \phi_k\|_{L^2} \le C/\sqrt{\lambda_k} \approx k^{-\frac 1 d}$, with $C$ uniform in $\eps$ and $k$. It is not clear at all how to improve this estimate. Consequently, in view of the estimate on $I^\eps_1$ in section \ref{s.wt}, we have 
\begin{equation*}
 \sup_{\eps \in (0,1)} \, \E \|X^\eps - P_N X^\eps\|_{L^2}^2 = \frac{1}{\eps^d} \sum_{k = N+1}^\infty \E(\nu^\eps u, \cG_\eps \phi_k)^2  \le \sum_{k=N+1}^\infty C \|\cG_\eps \phi_k\|_{L^2}^2 \approx \sum_{k = N+1}^\infty \frac{1}{\lambda_k} \approx \sum_{k = N+1}^\infty \frac{1}{k^{\frac 2 d}} .
\end{equation*}
This fails to show \eqref{e.t.L2.proj} or the tightness of $\{X^\eps\}$, even for $d = 2$.

In view of the analysis above, we find that the above approach for tightness, which is natural for non-oscillatory differential operators, fails completely in the presence of fast oscillations in the diffusion coefficients. The new approach used in sections \ref{s.cL2} and \ref{s.cHs} is necessary and more stable.

\medskip

{\bf Higher dimensional case.} We comment on the restrictions of dimensions in our results. Let us consider the short range correlation setting only. With respect to the weak topology, we expect that the scaling factor for the random fluctuation should be $\eps^{-\frac{d}{2}}$ in all dimensions. More precisely, for any fixed $\varphi \in L^2(D)$, we expect that $\eps^{-\frac d 2}(u^\eps - \E u^\eps,\varphi)$ converges in distribution for all dimensions. However, in this paper we prove this only for $d \le 7$; see the control of $\Var\,(I^\eps_3)$ in section \ref{s.wt}. This constraint is used to control the truncation term in the series expansion \eqref{e.Neumann}, and we do not expect it to be intrinsic. In fact, if we have more information about the random field, we can truncate later in \eqref{e.fw} until the last term is under control, and use higher order moments estimate in the spirit of Lemma \ref{l.cumu} and Lemma \ref{lem:fourth} to control the terms in between.

When certain norm topology is considered, however, we expect that the dimension plays an intrinsic role in the choice of the functional space. Indeed, for the leading term $X^\eps = -\eps^{-\frac{d}{2}} \cG_\eps \nu^\eps v^\eps$ to converge in law in $L^2(D)$, it is necessary that $\E \|X^\eps\|_{L^2}^2$ is controlled uniformly in $\eps$. Since the homogeneity $|x-y|^{-d+2}$ of the Green's function is relatively more and more singular in higher dimensions, this is simply not true. In the setting of this paper, we expect that convergence in law in $H^{-k}(D)$, for certain $k > 0$ increasing with respect to $d$ could be proved. To establish this in detail, we need again to push the series expansion further and have more information of the random field. We note that other choices of functional space could be convenient for this endeavor as well,  in particular, the $\cH^{-s}$ space defined in \eqref{e.cHsdef} for larger $s$, or certain negative H\"older spaces; we refer to \cite{BGGJ12_AA,GM15} for such considerations.

%%%%%%%%%%%%%%
%%%%%%%%%%%%%%
%\section*{Acknowledgments} The author would like to thank Guillaume Bal and Hung Tran for helpful discussions.

%%%%%%%%%%%%%%
%%%%%%%%%%%%%%
\bigskip

\appendix

%%%%%%%%%%%%%%
%%%%%%%%%%%%%%
\section{Some technical results}
\label{s.append}
%%%%%%%%%%%%%%
%%%%%%%%%%%%%%

\subsection{Tightness criteria for probability measures in functional spaces}
\label{s.tightness}

We first present a tightness criterion for the probability measures $\{P^{X^\eps}\}_{\eps \in (0,1)}$ on $L^2(D)$ induced by $\{X^\eps(\cdot,\omega)\}$ that are random elements in $H^s_0(D) \subset L^2(D)$, $s \in (0,1]$. 

\begin{theorem}[Tightness in $L^2(D)$]\label{t.tight} Let $\{X^\eps(\cdot,\omega)\}_{\eps\in (0,1)}$ be a family of random fields on the probability space $(\Omega,\cF,\bP)$, $X^\eps(\cdot,\omega) \in H^s_0(D)$ for some $0 < s \le 1$, for each fixed $\eps \in (0,1)$ and $\omega \in \Omega$. Suppose there exists $C > 0$, independent of $\eps$ and $\omega$, such that
\begin{equation}
\label{e.tight}
\E\,\|X^\eps\|_{H^s} \le C.
\end{equation}
Then the family of probability measures $\{P^{X^\eps}\}_{\eps\in(0,1)}$ on $L^2(D)$ is tight.
\end{theorem}

\begin{proof} By assumption $P^{X^\eps}$ concentrates on the subspace $H^s_0(D)$. For any fixed $\delta > 0$, set $M_\delta = C\delta^{-1}$ and define
\begin{equation*}
\mathcal{A}_\delta = \{f \in H^s_0(D) \;|\; \|f\|_{H^s} \le M_\delta\}.
\end{equation*}
Clearly, $\mathcal{A}_\delta$ is closed and bounded in $H^s_0(D)$. In light of the fact that the embedding $H^s_0(D) \hookrightarrow L^2(D)$ is compact \cite{PSV13}, we note that $\mathcal{A}_\delta$ is a compact set of $L^2(D)$. Now for any fixed $\eps \in (0,1)$, applying Chebyshev inequality, we find
\begin{equation*}
\begin{aligned}
P^{X^\eps}(\mathcal{A}_\delta) &= \bP(\{X^\eps \in H^s_0(D), \, \|X^\eps\|_{H^s} \le M_\delta\}) = 1 - \bP(\{\|X^\eps\|_{H^s} > M_\delta\})\\
&\ge 1 - \frac{\E \|X^\eps\|_{H^s}}{M_\delta} \ge 1 - \frac{C}{M_\delta} = 1 - \delta.
\end{aligned}
\end{equation*}
Since $\delta$ and $\eps$ are arbitrary, the above shows that $\{P^{X^\eps}\}_{\eps\in(0,1)}$ is tight.
\end{proof}

Next we give a similar tightness criterion for probability measures $\{P^{X^\eps}\}_{\eps \in (0,1)}$ on $H^{-1}(D)$ induced by $\{X^\eps(\cdot,\omega)\}$ which belongs to a smoother space.

\begin{theorem}[Tightness in $H^{-1}(D)$]\label{t.tight2} Let $\{X^\eps(\cdot,\omega)\}_{\eps\in (0,1)}$ be a family of random fields on the probability space $(\Omega,\cF,\bP)$, $X^\eps(\cdot,\omega) \in H^{-s}(D)$ for some $0 \le s < 1$, for each fixed $\eps \in (0,1)$ and $\omega \in \Omega$. Suppose there exists a constant $C > 0$, independent of $\eps$ and $\omega$, such that
\begin{equation}
\label{e.tight2}
\E\ \|X^\eps\|_{H^{-s}} \le C.
\end{equation}
Then the probability measures $\{P^{X^\eps}\}_{\eps\in(0,1)}$ on $H^{-1}(D)$ is tight.
\end{theorem}

\begin{proof} Since $D$ is a bounded open set with regular boundary, the embedding $H^1_0(D) \hookrightarrow H^s_0(D)$, for any $0 \le s < 1$, is compact \cite[Theorem 1.4.3.2]{Grisvard}. By duality, the embedding $H^{-s}(D) \hookrightarrow H^{-1}(D)$ is also compact. The rest of the proof is exactly the same as in the proof of the theorem above.
\end{proof}
%%%%%%%%%%%

%%%%%%%%%%%
\subsection{Functions of long range correlated Gaussian random field}

Here we record some results for the random potential $\nu(x,\omega) = \Phi(g(x,\omega))$ that is constructed in (L). In particular, we express the asymptotic behavior of its correlation function $R(x)$, and derive a (partial) fourth order moment for $\nu$.

\subsubsection{Autocorrelation function of the long range model}

\begin{lemma}\label{lem:tail}
Assume {\upshape(L1)(L2)} and let $\nu(x,\omega)$ be as constructed there. Set $V_1 = \E\{g_0 \Phi(g_0)\}$, $g_x$ being the underlying Gaussian random field in {\upshape(L)}. Then there exist constants $T, C >0$, depending only on the universal parameters, such that the autocorrelation function $R(x)$ of $q$ satisfies
\begin{equation}
|R(x) - V_1^2 R_g(x)| \le C R_g^2(x), \quad \text{for all } |x| \ge T,
\label{eq:tail1}
\end{equation}
where $R_g$ is the correlation function of $g$. Further,
\begin{equation}
|\E\{g(y) q(y+x)\}  - V_1 R_g(x)| \le C R_g^2(x), \quad \text{for all } |x| 
\ge T.
\label{eq:tail2}
\end{equation}
\end{lemma}

The proof of this result can be found in \cite{BGMP-AA-08, BGGJ12_AA}. It says that $\nu(x,\omega)$ inherits the heavy tail from the underlying Gaussian random field. The next result describes estimates on the integrals of $R$, possibly against some potential function that has singularity at the origin.

\begin{lemma}\label{lem:RLp}
Assume {\upshape(L1)(L2)} and let $\nu(x,\omega)$ be as constructed there. Let $R(x)$ be the correlation function of $\nu$, $R^\eps(x)$ be the rescaled function $R(\eps^{-1} x)$, and $B_\rho$ be the open ball centered at zero with radius $\rho$. Then there exists $C > 0$, depending only on the universal parameters, on $\rho$ and $p$ below, such that
\begin{itemize}
\item[\upshape(i)] For any $1 \le p < \infty$,
\begin{equation}
\| R^\eps \|_{L^p(B_\rho)} \le \left\{
\begin{aligned}
& C \eps^\alpha, & & \alpha p < d,\\
& C \eps^\alpha |\log \eps|^{\frac 1 p}, & & \alpha p = d,\\
& C \eps^{\frac d p}, & & \alpha p > d.
\end{aligned}
\right.
\label{eq:RLp}
\end{equation}

\item[\upshape(ii)] For $d = 2,4$, and for any $1 \le p < \infty$,
\begin{equation}\label{eq:RLp.4}
\left\| R^\eps(y) |\log|y|| \right\|_{L^p(B_\rho)} \le \left\{
\begin{aligned}
& C \eps^\alpha, & & \alpha p < d,\\%4,\\
& C \eps^{\alpha} |\log \eps|^{1+{\frac 1 p}}, & & \alpha p = d,\\%4,\\
& C \eps^{\frac d p} |\log \eps|, & & \alpha p > d.%4.
\end{aligned}
\right.
\end{equation}

\item[\upshape(iii)] For $d \ge 3$, and for any $1 \le p < \frac{d}{d-2}$,
\begin{equation}\label{eq:RLp.2}
\left\| \frac{R^\eps}{|y|^{d-2}} \right\|_{L^p(B_\rho)} \le \left\{
\begin{aligned}
& C \eps^\alpha, & & \alpha p < d - p(d-2),\\
& C \eps^\alpha |\log \eps|^{\frac 1 p}, & & \alpha p = d - p(d-2),\\
& C \eps^{{\frac d p} - (d-2)}, & & \alpha p > d - p(d-2).
\end{aligned}
\right.
\end{equation}

\item[\upshape(iv)] For $d \ge 5$, and for any $1 \le p < \frac{d}{d-4}$,
\begin{equation}\label{eq:RLp.5}
\left\| \frac{R^\eps}{|y|^{d-4}} \right\|_{L^p(B_\rho)} \le \left\{
\begin{aligned}
& C \eps^\alpha, & & \alpha p < d - p(d-4),\\
& C \eps^\alpha |\log \eps|^{\frac 1 p}, & & \alpha p = d - p(d-4),\\
& C \eps^{{\frac d p} - (d-4)}, & & \alpha p > d - p(d-4).
\end{aligned}
\right.
\end{equation}
\end{itemize}
\end{lemma}

\begin{proof} The proof of (i) can be found in \cite{BGGJ12_AA}. We present the proof of (iv) here; the other terms above can be proved in exactly the same way. Using the decomposition \eqref{e.D.dec}, we control the integral of $|R^\eps|^p |y|^{-p(d-4)}$ on $D_1$ and $D_2$ separately. 

On $D_1$, we have
\begin{equation*}
\int_{B_{T\eps}} \frac{|R^\eps|^p}{|y|^{p(d-4)}} dy \le \|R^\eps\|_{L^\infty}^p \int_0^{T\eps} \frac{r^{d-1}}{r^{p(d-4)}} dr \le C\eps^{d-p(d-4)}.
\end{equation*}

On $D_2$, we have
\begin{equation*}
\int_{B_\rho\setminus B_{T\eps}} \frac{|R^\eps|^p}{|y|^{p(d-4)}} dy \le \eps^{\alpha p} \int_{T\eps}^\rho \frac{r^{d-1}}{r^{p(d-4) + \alpha p}} dr =
\begin{cases}
C\eps^{\alpha p} \left(\left. r^{d- \alpha p - p(d-4)}\right)\right\rvert^\rho_{T\eps}, \quad& \alpha p \ne d - p(d-4),\\
C\eps^{\alpha p} \left(\left. \log r\right)\right\rvert^\rho_{T\eps}, \quad& \alpha p = d - p(d-4).\\
\end{cases}
\end{equation*}
Combining these estimates, we proved \eqref{eq:RLp.5}.
\end{proof}

\subsubsection{Fourth-order moments of $\nu(x,\omega)$}

Finally, we present a non-asymptotic estimate for the four-moments 
of $\nu(x,\omega)$ constructed in {\upshape (L1)(L2)}, with the additional assumption 
{\upshape (L3)}. In the following, we denote by 
$\mathcal{U}$ the collections of two pairs of unordered numbers in the set $\{1,2,3,4\}$, 
\begin{equation}
\mathcal{U} := \big\{ p = \{\big(p(1), p(2)\big), \big(p(3), p(4)\big)\}  
~|~ p(i) \in \{1,2,3,4\},\, p(1) \ne p(2), \, p(3) \ne p(4) \big\}.
\label{eq:Udef}
\end{equation}
As members in a set, the pairs $(p(1),p(2))$ and $(p(3),p(4))$ are required 
to be distinct; however, the two pairs can have one common index. There are three 
elements in $\mathcal{U}$ that collect all four numbers.  They are 
precisely $\{(1,2), (3,4)\}$, $\{(1,3),(2,4)\}$ and $\{(1,4), (2,3)\}$. Let 
$\mathcal{U}_*$ denote the subset formed by these three elements, and 
let $\mathcal{U}^*$ be its complement. 

\begin{lemma}\label{lem:fourth}
Assume {\upshape(L)} and let $\nu(x,\omega)$ be as constructed there. Then there exists $\vartheta : \R^d \to \R_+$, bounded and satisfying $\vartheta(x) \sim |x|^{-\alpha}$ as $|x| \to \infty$, and some $C > 0$ depending only on the universal parameters, such that for any four points $\{x_i \in \R^d; ~ 1 \le i \le 4\}$,
\begin{equation}
   \left| \E\prod_{i=1}^4 \nu(x_i)  - \sum_{p \in \mathcal{U}_*} 
R(x_{p(1)} - x_{p(2)})R(x_{p(3)}-x_{p(4)}) \right|%\\
   \le  C \sum_{p \in \mathcal{U}^*} \vartheta(x_{p(1)} - x_{p(2)})\vartheta(x_{p(3)} - 
x_{p(4)}).
\label{eq:cumu}
\end{equation}
\end{lemma}

We refer to \cite[Proposition 4.1]{BJ-CMS-11} for the proof of this result. In particular, $\vartheta$ above can be chosen as the autocorrelation function $R(x)$ of $\nu(x,\omega)$. As discussed earlier, \eqref{eq:cumu} can be viewed as an alternative for the estimates in Lemma \ref{l.cumu}.

%%%%%%%%%%%
%%%%%%%%%%%
\bibliographystyle{abbrv}%{plain}
\bibliography{bj_pr}

\begin{thebibliography}{10}

\bibitem{AL87}
M.~Avellaneda and F.-H. Lin.
\newblock Compactness methods in the theory of homogenization.
\newblock {\em Comm. Pure Appl. Math.}, 40(6):803--847, 1987.

\bibitem{AL91_Lp}
M.~Avellaneda and F.-H. Lin.
\newblock {$L^p$} bounds on singular integrals in homogenization.
\newblock {\em Comm. Pure Appl. Math.}, 44(8-9):897--910, 1991.

\bibitem{Chemin}
H.~Bahouri, J.-Y. Chemin, and R.~Danchin.
\newblock {\em Fourier analysis and nonlinear partial differential equations},
  volume 343 of {\em Grundlehren der Mathematischen Wissenschaften [Fundamental
  Principles of Mathematical Sciences]}.
\newblock Springer, Heidelberg, 2011.

\bibitem{B-CLH-08}
G.~Bal.
\newblock Central limits and homogenization in random media.
\newblock {\em Multiscale Model. Simul.}, 7(2):677--702, 2008.

\bibitem{BGGJ12_AA}
G.~Bal, J.~Garnier, Y.~Gu, and W.~Jing.
\newblock Corrector theory for elliptic equations with long-range correlated
  random potential.
\newblock {\em Asymptot. Anal.}, 77(3-4):123--145, 2012.

\bibitem{BGMP-AA-08}
G.~Bal, J.~Garnier, S.~Motsch, and V.~Perrier.
\newblock Random integrals and correctors in homogenization.
\newblock {\em Asymptot. Anal.}, 59(1-2):1--26, 2008.

\bibitem{BJ-DCDS-10}
G.~Bal and W.~Jing.
\newblock Homogenization and corrector theory for linear transport in random
  media.
\newblock {\em Discrete Contin. Dyn. Syst.}, 28(4):1311--1343, 2010.

\bibitem{BJ-CMS-11}
G.~Bal and W.~Jing.
\newblock Corrector theory for elliptic equations in random media with singular
  {G}reen's function. {A}pplication to random boundaries.
\newblock {\em Commun. Math. Sci.}, 19(2):383--411, 2011.

\bibitem{B-CPM}
P.~Billingsley.
\newblock {\em Convergence of probability measures}.
\newblock Wiley Series in Probability and Statistics: Probability and
  Statistics. John Wiley \& Sons Inc., New York, second edition, 1999.
\newblock A Wiley-Interscience Publication.

\bibitem{BP-99}
A.~Bourgeat and A.~Piatnitski.
\newblock Estimates in probability of the residual between the random and the
  homogenized solutions of one-dimensional second-order operator.
\newblock {\em Asymptot. Anal.}, 21(3-4):303--315, 1999.

\bibitem{DPV12}
E.~Di~Nezza, G.~Palatucci, and E.~Valdinoci.
\newblock Hitchhiker's guide to the fractional {S}obolev spaces.
\newblock {\em Bull. Sci. Math.}, 136(5):521--573, 2012.

\bibitem{Evans}
L.~C. Evans.
\newblock {\em Partial differential equations}, volume~19 of {\em Graduate
  Studies in Mathematics}.
\newblock American Mathematical Society, Providence, RI, 1998.

\bibitem{FOP-82}
R.~Figari, E.~Orlandi, and G.~Papanicolaou.
\newblock Mean field and {G}aussian approximation for partial differential
  equations with random coefficients.
\newblock {\em SIAM J. Appl. Math.}, 42(5):1069--1077, 1982.

\bibitem{GO14}
A.~Gloria and F.~Otto.
\newblock Quantitative results on the corrector equation in stochastic
  homogenization.
\newblock {\em Preprint}, 2014.

\bibitem{G06}
G.~Griso.
\newblock Interior error estimate for periodic homogenization.
\newblock {\em Anal. Appl. (Singap.)}, 4(1):61--79, 2006.

\bibitem{Grisvard}
P.~Grisvard.
\newblock {\em Elliptic problems in nonsmooth domains}, volume~24 of {\em
  Monographs and Studies in Mathematics}.
\newblock Pitman (Advanced Publishing Program), Boston, MA, 1985.

\bibitem{GM15}
Y.~Gu and J.-C. Mourrat.
\newblock Scaling limit of flucutations in stochastic homogenization.
\newblock {\em Preprint}, 2015.

\bibitem{HPP13}
M.~Hairer, E.~Pardoux, and A.~Piatnitski.
\newblock Random homogenisation of a highly oscillatory singular potential.
\newblock {\em Stoch. Partial Differ. Equ. Anal. Comput.}, 1(4):571--605, 2013.

\bibitem{Jikov_book}
V.~V. Jikov, S.~M. Kozlov, and O.~A. Ole{\u\i}nik.
\newblock {\em Homogenization of differential operators and integral
  functionals}.
\newblock Springer-Verlag, Berlin, 1994.

\bibitem{KLS12_ARMA}
C.~E. Kenig, F.~Lin, and Z.~Shen.
\newblock Convergence rates in {$L^2$} for elliptic homogenization problems.
\newblock {\em Arch. Ration. Mech. Anal.}, 203(3):1009--1036, 2012.

\bibitem{KLS14_GN}
C.~E. Kenig, F.~Lin, and Z.~Shen.
\newblock Periodic homogenization of {G}reen and {N}eumann functions.
\newblock {\em Comm. Pure Appl. Math.}, 67(8):1219--1262, 2014.

\bibitem{Khoshnevisan}
D.~Khoshnevisan.
\newblock {\em Multiparameter processes}.
\newblock Springer Monographs in Mathematics. Springer-Verlag, New York, 2002.
\newblock An introduction to random fields.

\bibitem{LL-A}
E.~H. Lieb and M.~Loss.
\newblock {\em Analysis}, volume~14 of {\em Graduate Studies in Mathematics}.
\newblock American Mathematical Society, Providence, RI, second edition, 2001.

\bibitem{MaO14}
D.~Marahrens and F.~Otto.
\newblock On annealed elliptic green function estimates.
\newblock {\em Preprint}, 2014.

\bibitem{MV97}
S.~Moskow and M.~Vogelius.
\newblock First-order corrections to the homogenised eigenvalues of a periodic
  composite medium. {A} convergence proof.
\newblock {\em Proc. Roy. Soc. Edinburgh Sect. A}, 127(6):1263--1299, 1997.

\bibitem{MN15}
J.-C. Mourrat and J.~Nolen.
\newblock Scaling limit of the corrector in stochastic homogenization.
\newblock {\em Preprint}, 2015.

\bibitem{MO14}
J.-C. Mourrat and F.~Otto.
\newblock Correlation structure of the corrector in stochastic homogenization.
\newblock {\em Preprint}, 2014.

\bibitem{PSV13}
G.~Palatucci, O.~Savin, and E.~Valdinoci.
\newblock Local and global minimizers for a variational energy involving a
  fractional norm.
\newblock {\em Ann. Mat. Pura Appl. (4)}, 192(4):673--718, 2013.

\bibitem{Partha}
K.~R. Parthasarathy.
\newblock {\em Probability measures on metric spaces}.
\newblock Probability and Mathematical Statistics, No. 3. Academic Press, Inc.,
  New York-London, 1967.

\end{thebibliography}
\end{document}